\documentclass{amsart}
\usepackage{todonotes}
\usepackage{graphicx, amsmath, amssymb, amsthm,subcaption}
\usepackage{tikz-cd}
\usepackage{spectralsequences}
\usepackage{hyperref} 
\usepackage{adjustbox}
\usepackage[backend=biber, style=numeric, maxbibnames=99, doi=true,url=false,
    giveninits=true, hyperref]{biblatex}
\renewbibmacro{in:}{}
\newcommand{\aln}[1]{\begin{align*}#1\end{align*}}
\newcommand{\st}{\text{ : }}

\newcommand{\complete}{\;\widehat{}}

\newcommand{\Z}{{\mathbb  Z}}

\newcommand{\R}{{\mathbb R}}
\newcommand{\F}{{\mathbb F}}

\newcommand{\BP}{BP}
\newcommand{\BPR}{\BP_{\R}}

\newcommand{\BPRn}[1][m]{\BPR\langle #1\rangle}

\DeclareMathOperator{\Hom}{Hom}
\DeclareMathOperator{\Ext}{Ext}

\DeclareMathOperator{\coker}{coker}

\DeclareMathOperator{\Spec}{Spec}

\newcommand{\m}[1]{{\protect\underline{#1}}}
\newcommand{\mZ}{\m{\Z}}

\newcommand{\cc}[1]{\mathcal #1}
\newcommand{\cA}{\cc{A}}

\newcommand{\Sp}{\mathcal Sp}

\mathchardef\mhyphen=45

\numberwithin{equation}{section}

\newtheorem{theorem}{Theorem}[section]
\newtheorem{lemma}[theorem]{Lemma}
\newtheorem{corollary}[theorem]{Corollary}

\newtheorem{proposition}[theorem]{Proposition}
\newtheorem{conjecture}[theorem]{Conjecture}

\newtheorem*{theorem*}{Theorem}
\newtheorem*{proposition*}{Proposition}

\theoremstyle{remark}
\newtheorem{remark}[theorem]{Remark}
\newtheorem{example}[theorem]{Example}

\theoremstyle{definition}

\newtheorem{definition}[theorem]{Definition}

\newcommand{\sseq}{spectral sequence}

\newcommand{\ov}{\bar{v}}

\let\oldwidehat\ov
\protected\def\ov{\oldwidehat}

\begin{document}
\title[The homological slice spectral sequence]{The homological slice spectral sequence in motivic and Real bordism}

\author{Christian Carrick}
\address{Mathematical Institute, Utrecht University, Utrecht, 3584 CD, the Netherlands}
\email{c.d.carrick@uu.nl}
\author{Michael A. Hill}
\address{University of California Los Angeles, Los Angeles, CA 90095}
\email{mikehill@math.ucla.edu}
\thanks{This material is based upon work supported by the National Science Foundation under Grant No. 2105019}
\author{Douglas C. Ravenel}
\address{Department of Mathematics University of Rochester, Rochester, NY 14627}
\email{dcravenel@gmail.com}


\begin{abstract}
For a motivic spectrum $E\in \mathcal{SH}(k)$, let $\Gamma(E)$ denote the global sections spectrum, where $E$ is viewed as a sheaf of spectra on $\mathrm{Sm}_k$. Voevodsky's slice filtration determines a spectral sequence converging to the homotopy groups of $\Gamma(E)$. In this paper, we introduce a spectral sequence converging instead to the mod 2 homology of $\Gamma(E)$ and study the case $E=BPGL\langle m\rangle$ for $k=\R$ in detail. We show that this spectral sequence contains the $\mathcal{A}_*$-comodule algebra $\mathcal{A}_*\square_{\mathcal{A}(m)_*}\F_2$ as permanent cycles, and we determine a family of differentials interpolating between $\mathcal{A}_*\square_{\mathcal{A}(0)_*}\F_2$ and $\mathcal{A}_*\square_{\mathcal{A}(m)_*}\F_2$. Using this, we compute the spectral sequence completely for $m\le 3$.

In the height 2 case, the Betti realization of $BPGL\langle 2\rangle$ is the $C_2$-spectrum $BP_\R\langle 2\rangle$, a form of which was shown by Hill and Meier to be an equivariant model for $\mathrm{tmf}_1(3)$. Our spectral sequence therefore gives a computation of the comodule algebra $H_*\mathrm{tmf}_0(3)$. As a consequence, we deduce a new ($2$-local) Wood-type splitting
\[\mathrm{tmf}\wedge X\simeq \mathrm{tmf}_0(3)\]
of $\mathrm{tmf}$-modules predicted by Davis and Mahowald, for $X$ a certain 10-cell complex. 
\end{abstract}
\maketitle
\tableofcontents

\section{Introduction}
The classical truncated Brown--Peterson spectrum $BP\langle m\rangle$ admits an action by the cyclic group of order 2, via the complex conjugation action on complex bordism, $MU$. The fixed points admit familiar geometric models for small heights $m$ and give higher height analogues of connective real $K$-theory $\mathrm{ko}$. We give a spectral sequence converging to the homology of these fixed point spectra by lifting the calculation to a simpler one in the $\R$-motivic stable homotopy category. 

\subsection{Motivation}
The chromatic approach to stable homotopy gives a method to compute the stable homotopy groups of spheres via a step-by-step procedure. This procedure comes from the height filtration on the moduli stack of formal groups, mirrored in stable homotopy by Bousfield localization. The chromatic convergence theorem of Hopkins--Ravenel \cite{ravorange} states that the $p$-local sphere spectrum may be recovered as the limit of its chromatic tower
\[S^0_{(p)}\simeq\varprojlim \big(\cdots\to L_nS^0\to L_{n-1}S^0\to\cdots\to L_0S^0\big)\]
where $L_n$ denotes Bousfield localization at a height $n$ Morava $E$-theory $E(k,\Gamma)$, for $\Gamma$ a height $n$ formal group over a perfect field $k$ of characteristic $p$. The spectrum $L_nS^0$ is built in finitely many steps from the $K(i)$-local spheres $L_{K(i)}S^0$ for $i\le n$, via chromatic fracture squares. The spectrum $L_{K(n)}S^0$, in turn, admits a description as a homotopy fixed-point spectrum 
\[L_{K(n)}S^0\simeq E(k,\Gamma)^{h\mathbb G_n}\]
by a theorem of Devinatz--Hopkins \cite{devhop}, where $\mathbb G_n=\mathrm{Aut}(\Gamma)$ is the Morava stabilizer group. This, in principle, reduces many questions in stable homotopy theory to understanding the $\mathbb G_n$-equivariant spectrum $E(k,\Gamma)$.

However, this is not easy to access in practice, even at small heights. It was an observation of Ravenel \cite{ravarf} and Hopkins--Miller \cite{hopmiller} that $L_{K(n)}S^0$ is well approximated by the fixed points of $E(k,\Gamma)$ at \textit{finite} subgroups of $\mathbb G_n$, and that these theories are more computable. These theories
\[EO_n(G):=E(k,\Gamma)^{hG}\]
are known as the Hopkins--Miller higher real $K$-theories.

The $EO_n(G)$'s have proven very effective at carrying rich information in an accessible way. For instance, studying these theories led to the solution of the Kervaire invariant problem, at primes $p\ge5$ by Ravenel \cite{ravarf} and at $p=2$ by Hill--Hopkins--Ravenel \cite{HHR}. At the prime $p=3$, they were used by Goerss--Henn--Mahowald--Rezk \cite{ghmr} to produce an explicit finite resolution of the $K(2)$-local sphere, giving a conceptual framework to the calculation of $\pi_*L_{K(2)}S^0$ by Shimomura--Wang \cite{SW}. We refer the reader also to the introduction of \cite{BHLSZ} for a nice discussion of these theories and their history.

\subsubsection*{Connective models} Working with the $EO_n(G)$'s in practice is limited by the size of these theories. The $EO_n(G)$'s are non-connective, and their mod $p$ homology vanishes, making it impossible to understand these theories directly from the point of view of the Adams spectral sequence. Moreover, the homotopy groups of the $EO_n(G)$'s are not degreewise finitely generated, so passing to the connective cover does not give a substantial improvement. A key feature of the study of the $EO_n(G)$'s is thus a search for good connective models $eo_n(G)$ with strong finiteness properties.

At heights $h=2^{n-1}m$ with $p=2$, the group $\mathbb G_h$ contains a subgroup isomorphic to $G=C_{2^n}$, and Beaudry--Hill--Shi--Zeng defined and studied candidates for good connective models of $EO_h(G)$ in this case via Real bordism \cite{BHSZ}. These theories---known as the $BP^{((G))}\langle m\rangle$'s---are particularly accessible as they arise as the fixed points of a $G$-spectrum whose action comes from geometry, as opposed to the action on $E(k,\Gamma)$, which is defined via obstruction theory. An important manifestation of this is that the $BP^{((G))}\langle m\rangle$'s have a well understood slice filtration, in the sense of Hill--Hopkins--Ravenel \cite{HHR}.

In upcoming work \cite{CH}, the first and second named authors use the slice filtration to establish the desired finiteness properties of the $BP^{((G))}\langle m\rangle$'s, showing that their fixed points are fp spectra of type $m|G|/2$, in the sense of Mahowald---Rezk \cite{MR}. For explicit calculations, however, we would like to know the comodules $H_*BP^{((G))}\langle m\rangle^G$, and the main aim of the present paper is to initiate such homology computations by exploring the case $G=C_2$.

\subsubsection*{The $G=C_2$ case} The cyclic group of order 2 is often used as a test case for computations in equivariant homotopy in general and for the study of the $EO_n(G)$'s and $BP^{((G))}\langle m\rangle$'s in particular. Quite a lot is known in this case; for instance, the homotopy fixed point spectral sequence computing $\pi_*EO_n(C_2)$ was completely computed at all heights by Hahn--Shi in \cite{hahnshi}. Li--Shi--Wang--Xu showed that a large class of elements in the homotopy groups of spheres is detected by the $EO_n(C_2)$'s \cite{LSWX}. The $EO_n(C_2)$'s were studied extensively by Kitchloo--Wilson; they used these theories to prove new nonimmersion results for real projective spaces \cite{KW1} \cite{KW2}, demonstrating that these theories are useful far beyond their role as a test case.

For $G=C_2$, the connective models $BP^{((G))}\langle m\rangle$ were first defined by Hu--Kriz \cite{HK} and are also known as the $BP_\R\langle m\rangle$'s, as they model the classical truncated Brown--Peterson spectra $BP\langle m\rangle$ with their $C_2$-action via complex conjugation. At small heights $m$, the $BP_\R\langle m\rangle$'s and their fixed points admit familiar geometric models:

\begin{table}[!htbp]
\begin{tabular}{l | c | c }
$m$& $BP_\R\langle m\rangle$ & $BP_\R\langle m\rangle^{C_2}$\\
\hline\hline
-1&$H\underline{\F_2}$&$H\F_2$\\
0&$H\underline{\Z}_{(2)}$&$H\Z_{(2)}$\\
1&$\mathrm{k}_\R$&$\mathrm{ko}$\\
2&$\mathrm{tmf}_1(3)$&$\mathrm{tmf}_0(3)$
\end{tabular}
\end{table}

Here $\mathrm{k}_\R$ denotes connective Real $K$-theory in the sense of Atiyah \cite{atiyah}, and $\mathrm{tmf}_1(3)$ and $\mathrm{tmf}_0(3)$ are spectra of topological modular forms with level structure (see \cite{hilllawson} and \cite{hillmeier}).

\subsubsection*{$\R$-motivic homotopy} For a real variety $X$, its set of complex points $X(\mathbb C)$ comes equipped with the complex analytic topology, with respect to which the action of $C_2=\mathrm{Gal}(\mathbb C/\mathbb R)$ is continuous. This extends to a symmetric monoidal left adjoint 
\[\beta: \mathcal{SH}(\R)\to \Sp^{C_2}\]
from the $\R$-motivic stable homotopy category to the category of genuine $C_2$-spectra, called Betti realization. Real bordism and the $BP_\R\langle m\rangle$'s lift along $\beta$ to $\R$-motivic analogues, known as motivic bordism and the $BPGL\langle m\rangle$'s.

Consequently, we lift the calculation of $H_*BP_\R\langle m\rangle^{C_2}$ along $\beta$ to a simpler and more fundamental calculation in the $\R$-motivic stable homotopy category, and it is in this setting in which we do our computations. In the motivic setting, the role of fixed points $(-)^{C_2}$ is played by the global sections functor $\Gamma:\mathcal{SH}(k)\to \Sp$, the right adjoint to the unique colimit-preserving symmetric monoidal functor
\[i_*:\Sp\to \mathcal{SH}(k)\]
For $E\in \mathcal{SH}(k)$, $\Gamma(E)$ may be computed as the mapping spectrum $F(\Spec(k)_+,E)$ in $\mathcal{SH}(k)$ and thus is the global sections of $E$, viewed as a sheaf of spectra on $\mathrm{Sm}_k$, the category of smooth schemes over $\Spec(k)$.

The advantage of working motivically is that we may discard the so-called negative cone in our calculations. As we show in Section \ref{secequivariant}, the presence of the negative cone in the $C_2$-equivariant homology of a point results in serious complications. In Section \ref{sec2}, we give a method of comparing $C_2$-equivariant calculations of the kind studied here to the corresponding $\R$-motivic calculations, and we use this to show that the results must agree in non-positive weights. In practice, this means we may discard the negative cone and still recover the correct result in degrees $a+b\sigma$ for $b\le0$ and, in particular, in integer degrees.

\subsection{Main results}

\subsubsection*{Section 2} We begin by defining the homological slice spectral sequence (HSSS) in Section \ref{sec2}. This spectral sequence arises by smashing the slice tower of $E\in\mathcal{SH}(k)$ with the motivic spectrum $i_*H\F_2$, giving a spectral sequence of $\mathcal{A}_*$-comodules. We show that when $E$ is slice connective, this spectral sequence has strong vanishing lines and converges to the comodule $H_*\Gamma(E)$ (Proposition \ref{propvanishinglines}).

For our applications to the $BP_\R\langle m\rangle$'s, we compare the global sections functor $\Gamma$ and the genuine fixed points functor $(-)^{C_2}$, producing a natural transformation
\[\Gamma(E)\to\beta(E)^{C_2}\]
which we show to be an equivalence in many cases. In particular,
\begin{proposition}[Corollary \ref{cor:betaE}]
    For all $m$, there is an equivalence of spectra
    \[\Gamma(BPGL\langle m\rangle)\to BP_\R\langle m\rangle^{C_2}\]
\end{proposition}

\subsubsection*{Section 3} In \cite{BehrensShah}, Behrens--Shah lift the Tate square in $C_2$-equivariant homotopy along $\beta$ to an arithmetic square
\[
\begin{tikzcd}
    E\arrow[r]\arrow[d]& E[\rho^{-1}]\arrow[d]\\
    E^{\complete}_\rho[\tau^{-1}]\arrow[r]&E^{\complete}_\rho[\tau^{-1}][\rho^{-1}]
\end{tikzcd}
\]
for $E\in\mathcal{SH}(\R)$. In contrast to the Tate square, this is not usually a pullback, but we show that it becomes a pullback after applying $\Gamma(-)$ in many cases. For $E=i_*H\F_2\otimes BPGL$, we completely determine the effect of the arithmetic square on bigraded homotopy groups; the result in weight zero is as follows.

\begin{proposition}[Proposition \ref{proparithmeticsquarecomputation}]
    On homotopy groups, the arithmetic square of $i_*H\F_2\otimes BPGL$ is given by
\[
\begin{tikzcd}
H_*\Gamma(BPGL)\arrow[r]\arrow[d]&\mathcal A_*\arrow[d,"\varphi"]&\xi_i\arrow[d,mapsto]\\
\F_2[z,\chi_i|i\ge1]^{\complete}_{z}\arrow[r,hookrightarrow]&\F_2[z,\chi_i|i\ge1]^{\complete}_{z}[z^{-1}]&\chi_iz^{2^i-1}+\chi_{i-1}z^{-1}
\end{tikzcd}
\]
with $|z|=-1$ and $|\chi_i|=2(2^i-1)$.
\end{proposition}

This leads to a somewhat explicit description of the $\mathcal{A}_*$-comodule $H_*\Gamma(BPGL)$ via the associated Mayer--Vietoris sequence (Theorem \ref{thmhomologybpr}). This corrects an error of Hu--Kriz, who gave similar formulas in \cite{BPO} (see Remark \ref{hukrizerror}). The $\rho$-local part of this arithmetic square is complicated by the following unexpected fact about the reduction map $BPGL\to M\F_2$ from the motivic Brown--Peterson spectrum to the mod 2 motivic Eilenberg--Maclane spectrum.

\begin{theorem}[Theorem \ref{thmtatefrob}]\label{mainthmtate} 
The composite
\[H\F_2\simeq\Gamma(BPGL[\rho^{-1}])\to \Gamma(M\F_2[\rho^{-1}])\simeq\prod\limits_{i\ge0}\Sigma^i H\F_2\]
has components $Sq^i:H\F_2\to \Sigma^i H\F_2$.

Equivariantly, identifying $\Phi^{C_2}(H\underline{\F_2})$ with the connective cover of $H\F_2^{tC_2}$, the composite 
\[H\F_2\simeq\Phi^{C_2}BP_\R\to \Phi^{C_2}H\underline{\F_2}\simeq\prod\limits_{i\ge0}\Sigma^i H\F_2\]
has components $Sq^i:H\F_2\to \Sigma^i H\F_2$.
\end{theorem}

We prove Theorem \ref{mainthmtate} as an application of Mahowald's theorem on $H\F_2$. The description of the map appearing here comes via comparison with the Tate-valued Frobenius of Nikolaus--Scholze \cite{NS}. We expect this fact to be of independent interest.

\subsubsection*{Section 4}
The HSSS for $BPGL\langle m\rangle$ has a straightforward $E_2$-page.
\begin{proposition}[Corollary \ref{corE_2description}]
The $E_2$-page of the HSSS for $BPGL\langle m\rangle$ is isomorphic, as an $\mathcal{A}_*$-comodule algebra, to
\[(\mathcal{A}_*\square_{\mathcal{A}(0)_*}\F_2)[\rho,x_1,\ov_1,\ldots,\ov_m]\]
where 
\[\mathcal{A}_*\square_{\mathcal{A}(0)_*}\F_2\cong H_*H\Z\cong\F_2[\zeta_1^2,\zeta_2,\ldots]\]
The tridegrees are given as follows
\aln{
|\zeta_i|&=(2^i-1,0,1-2^i)\\
|\rho|&=(-1,-1,1)\\
|x_1|&=(0,-1,0)\\
|\ov_i|&=(2(2^i-1),2^i-1,0)
}
where we use Adams trigrading $(t-s,w,s)$, representing (stem, weight, filtration). The classes $\rho$ and $\ov_i$ are comodule primitives, and
\[\psi(x_1)=1\otimes x_1+\xi_1\otimes \rho\]
\end{proposition}
Our main application of the results of Section \ref{secarithmeticsquare} on the  arithmetic square is to determine a large class of permanent cycles in the HSSS for $BPGL\langle m\rangle$. This comes via the edge homomorphism
\[H_*\Gamma(BPGL\langle m\rangle)\to H_*\Gamma(P^0BPGL\langle m\rangle)=H_*H\Z\]

\begin{theorem}[Theorem \ref{edgethm}]\label{mainthmedge} In weight zero, the image of the edge homomorphism in the HSSS for $BPGL\langle m\rangle$ is precisely
\[\mathcal{A}_*\square_{\mathcal{A}(m)_*}\F_2\subset \mathcal{A}_*\square_{\mathcal{A}(0)_*}\F_2\]    
\end{theorem}

The classes $x_1,\rho$, and $\ov_i$ are all permanent cycles, so the subalgebra
\[(\mathcal{A}_*\square_{\mathcal{A}(m)_*}\F_2)[\rho,x_1,\ov_1,\ldots,\ov_m]\]
of $E_2$ consists of permanent cycles. We use the arithmetic square to identify also a class of permanent cycles in negative weights.

\begin{proposition}[Corollary \ref{corx_mpermanent}]
The classes defined inductively $x_0=\rho$, $x_1$ the class above, and 
\[x_n=\sum\limits_{i=0}^{n-1}x_i\zeta_{n-i}^{2^i}\]
are permanent cycles in the HSSS for $BPGL\langle m\rangle$, for all $n$ and $m$. The tridegree of $x_n$ is $(2^n-2,-1,2-2^n)$, and the coaction is as follows
\[\psi(x_n)=\sum\limits_{i=0}^n\xi_i^{2^{n-i}}\otimes x_{n-i}\]
\end{proposition}

Heuristically, the $x_n$ classes may be thought of as indecomposable elements of the form $[\xi_n\rho]$. We determine a family of differentials in the HSSS for $BPGL\langle m\rangle$, which may be expressed in terms of the $x_n$'s. The classes $x_n$ appear as Massey products when taking homology with respect to these differentials.

\begin{theorem}[Theorem \ref{differentialsthm}]\label{mainthmdiff}
In the HSSS for $BPGL\langle m\rangle$, we have the differentials
\[d_{2^{i+1}-1}(\zeta_j^{2^{i+1-j}})=\overline{v}_{i}\rho^{2^i-1}\bigg(p_{j-1}\bigg(\frac{x_1}{\rho},\ldots,\frac{x_{j-1}}{\rho}\bigg)\bigg)^{2^{i+1-j}}\]
for all $1\le i\le m$ and $1\le j\le i+1$, where $p_j$ is the polynomial
\[\zeta_{j}=p_j(\xi_1,\ldots,\xi_j)\]
given by the inversion formulas in the Hopf algebra $\mathcal A_*$.
\end{theorem}

The classes $x_n$ are not divisible by $\rho$, but all of the $\rho$ exponents appearing are nonnegative when this expression is expanded. Using the heuristic $x_n=[\xi_n\rho]$, this expression may then be read more compactly as
\[d_{2^{i+1}-1}(\zeta_j^{2^{i+1-j}})=\overline{v}_{i}[\rho^{2^i-1}\zeta_{j-1}^{2^{i+1-j}}]\]
As each $\zeta_j^{2^{i+1-j}}$ supports a nonzero $d_{2^{i+1}-1}$, its square is a nonzero cycle on the next page. This stops when $i=m+1$, thus interpolating between 
\[\mathcal{A}_*\square_{\mathcal{A}(0)_*}\F_2=\F_2[\zeta_1^2,\zeta_2,\ldots]\]
and
\[\mathcal{A}_*\square_{\mathcal{A}(m)_*}\F_2=\F_2[\zeta_1^{2^{m+1}},\zeta_2^{2^m},\ldots,\zeta_{m+1}^2,\zeta_{m+2},\ldots]\]

\subsubsection*{Section 5} These differentials are explicit but complicated due to the inversion formulas in $\mathcal{A}_*$, which makes computing homology with respect to these differentials quite difficult, though purely algebraic. This limits our understanding of the behavior of this spectral sequence in general beyond Theorems \ref{mainthmdiff} and \ref{mainthmedge}. However, we make the following conjecture.

\begin{conjecture}\label{conj}
All differentials in the HSSS for $BPGL\langle m\rangle$ are generated under the Leibniz rule by those in Theorem \ref{mainthmdiff}. In particular, the spectral sequence collapses on $E_{2^{m+1}}$.
\end{conjecture}

Given Conjecture \ref{conj}, computing $H_*\Gamma(BPGL\langle m\rangle)$ becomes the purely algebraic problem of computing homology with respect to these differentials, modulo comodule algebra extension problems. In Section \ref{seccomputations}, we verify this conjecture for $m\le 3$, and compute the corresponding spectral sequences completely. 

At height 1, this recovers the classical computation of $H_*\mathrm{ko}$, without the use of the Wood cofiber sequence or knowledge of $\pi_*\mathrm{ko}$. At height 2, it is a theorem of Hill--Meier that the $C_2$-spectrum $\mathrm{tmf}_1(3)$ is a form of $BP_\R\langle 2\rangle$ \cite{hillmeier}, so our computation determines the $\mathcal{A}_*$-comodule algebra $H_*\mathrm{tmf}_0(3)$. 

\begin{theorem}[Theorem \ref{thmt03}]\label{mainthmhtmf}
    There is an isomorphism of $\mathcal{A}_*$-comodule algebras
    \[H_*\mathrm{tmf}_0(3)\cong \mathcal{A}_*\square_{\mathcal{A}(2)_*}M_2\]
    where $M_2$ is the $\mathcal{A}(2)_*$-comodule algebra given by the square zero extension $\F_2\{1\}\oplus \overline{M}_{2}$, and $\overline{M}_{2}$ is the $\mathcal{A}(2)_*$-comodule is displayed below in Figure \ref{figM2}.
\end{theorem}

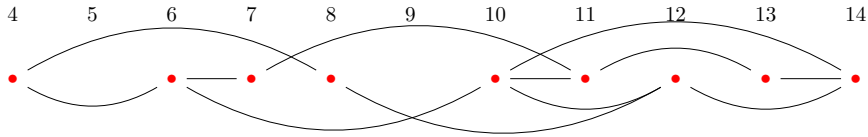
\begin{figure}[!htbp]
\centering
\[
\adjustbox{scale=0.8,center}{
\begin{tikzcd}
4&5&6&7&8&9&10&11&12&13&14\\
\color{red}{\bullet}\arrow[rr,dash,bend right]\arrow[rrrr,dash,bend left]&&\color{red}{\bullet}\arrow[r,dash]\arrow[rrrr,dash,bend right]&\color{red}{\bullet}\arrow[rrrr,dash,bend left]&\color{red}{\bullet}\arrow[rrrr,dash,bend right]&&\color{red}{\bullet}\arrow[r,dash]\arrow[rr,dash,bend right]\arrow[rrrr,dash,bend left]&\color{red}{\bullet}\arrow[rr,dash,bend left]&\color{red}{\bullet}\arrow[rr,dash,bend right]&\color{red}{\bullet}\arrow[r,dash]&\color{red}{\bullet}
\end{tikzcd}
}
\]
\caption{The $\mathcal{A}(2)_*$-comodule $\overline{M}_{2}$}
\label{figM2}
\end{figure}

Prior to the work of Hill--Lawson \cite{hilllawson}, it was not known how to produce a spectrum $\mathrm{Tmf}_0(3)$ (and thereby a connective model $\mathrm{tmf}_0(3)$ of the periodic spectrum $\mathrm{TMF}_0(3)$ of topological modular forms with level structure) as the global sections of a derived stack. In \cite{davismahowald}, Davis and Mahowald proposed several hands-on definitions that were suitable for computation. In particular, they construct a certain 10-cell complex $X$ \cite[Theorem 2.1]{davismahowald} and propose $\mathrm{tmf}\wedge X$ as such a connective model.

Our computation in Theorem \ref{mainthmhtmf} allows us to join this definition with the derived algebro-geometric one of Hill--Lawson. We use the Adams spectral sequence to construct a map $X\to \mathrm{tmf}_0(3)$, resulting in the following Wood-type splitting.

\begin{theorem}[Corollary \ref{cortmfsplitting}]
    There is a 2-local equivalence of $\mathrm{tmf}$-modules 
    \[\mathrm{tmf}_0(3)\simeq\mathrm{tmf}\wedge X \]
\end{theorem}

Rationally, $X$ has rank 4, corresponding to the fact that the map of stacks $\mathcal{M}_0(3)\to\mathcal{M}_{ell}$ is a 4-fold cover. Base changing, this also gives the corresponding splittings for $\mathrm{Tmf}_0(3)$ and $\mathrm{TMF}_0(3)$. 

At height 3, there is no known familiar geometric model of $BP_\R\langle 3\rangle$, but our computation extends the program of computing $H_*eo_n(G)$ at $p=2$ to height 3 for the first time, while also demonstrating the staggering complexity of these computations at heights $>2$. We refer the reader to Section \ref{subsecn=3computation} for a description of the comodule algebra $M_3$.

\begin{theorem}[Theorem \ref{einftym=3}]
The $E_\infty$ page of the HSSS for $BPGL\langle 3\rangle$ is isomorphic to the square zero extension
\[(\mathcal{A}_*\square_{\mathcal{A}(3)_*}M_3)\oplus (\mathcal{A}_*\square_{\mathcal{A}(2)_*}\overline{M}_{2})\]
for $M_3$ an explicit $\mathcal{A}(3)_*$-comodule algebra of dimension 165, where $\overline{M}_{3}\cdot\overline{M}_{2}=0$, and $M_2$ is the comodule of Theorem \ref{mainthmhtmf}.
\end{theorem}

We could, in principle, go further to height 4 and beyond; we see no reason our methods should not extend. However, we are unable to give a general description of a comodule $M_m$ at height $m$, and already at height 3, it is a significant task to give a proper account of the computation.  

\subsubsection*{Section 6} We finish in Section \ref{secequivariant} by defining and exploring the HSSS in equivariant stable homotopy. Whereas the slice spectral sequence of Hill--Hopkins--Ravenel \cite{HHR} has proven an effective tool to compute the homotopy groups of the fixed points $E^G$ of a $G$-spectrum $E$, this gives a way to compute the \textit{homology} of $E^G$.

When $G=C_2$, the HSSS for $BP_\R\langle m\rangle$ is closely related to the motivic spectral sequences studied in the body of the paper. However, the presence of the negative cone in the equivariant homology of a point results in some unexpected differences. We show in particular the that the $RO(C_2)$-graded HSSS of $k_\R$ has an exotic differential of the form
\[d_5(2u_{2\sigma}^{-1}\cdot \zeta_1^2\zeta_2)=\ov_1^2\]
This differential leaves a class in the negative cone and kills a class in the positive cone. 

\subsection{Acknowledgments}
The authors would like to thank Lennart Meier and Danny Shi for helpful conversations. The first and second-named authors would like to thank the Hausdorff Research Institute for Mathematics for its hospitality and support while some of this work was done. Our slice spectral sequence charts were created using Hood Chatham's {\tt spectralsequences} package, and our Adams spectral sequence charts were created using Bruner's {\tt ext} software.
\subsection{Notation and conventions}
\begin{enumerate}
\item We work at the prime $p=2$ in this paper, so homology $H_*(-)$ will always denote mod 2 homology $H_*(-;\F_2)$ and our Steenrod algebras are the 2-primary versions.
\item $\mathrm{Sm}_k$ denotes the category of smooth schemes over $\Spec(k)$, and $\mathcal{SH}(k)$ denotes the category of $k$-motivic spectra. $\mathcal{SH}(k)_{\mathrm{cell}}$ denotes the cellular category, the localizing subcategory generated by the bigraded spheres $S^{s,w}$.
\item $\pi_{s,w}^\R(-)$ denotes bigraded homotopy groups for an $\R$-motivic spectrum, where $s$ denotes the stem or topological degree, and $w$ denotes the weight.
\item $M\Z$ and $M\F_2$ are the integral and mod 2 $\R$-motivic homology spectra, so that $\pi_{*,*}^\R M\Z=\Z[\rho,\tau^2]/(2\rho)$ and $\pi_{*,*}^\R M\F_2=\F_2[\rho,\tau]$ where $|\rho|=(-1,-1)$ and $|\tau|=(0,-1)$. 
\item $\mathcal{A}_*$ denotes the classical mod 2 dual Steenrod algebra, so that 
\[\mathcal{A}_*=\F_2[\xi_1,\xi_2,\ldots]\]
where $|\xi_i|=2^i-1$ are the usual Milnor generators, and $\zeta_i$ denotes the conjugate of $\xi_i$. $\mathcal{A}(m)_*$ denotes the quotient Hopf algebra
\[\mathcal{A}(m)_*=\F_2[\xi_1,\ldots,\xi_{m+1}]/(\xi_1^{2^{m+1}},\ldots,\xi_{m+1}^2)\]
and
\[\mathcal{A}_*\square_{\mathcal{A}(m)_*}\F_2\cong\F_2[\zeta_1^{2^{m+1}},\zeta_2^{2^m},\ldots,\zeta_{m+1}^2,\zeta_{m+2},\ldots]\]
Dually, $\mathcal{A}$ denotes the mod 2 Steenrod algebra, $\mathcal{A}(m)$ the subalgebra generated by $Sq^{2^i}$ for $i\le m$, and 
\[\mathcal{A}/\!/\mathcal{A}(m)=\mathcal{A}\otimes_{\mathcal{A}(m)}\F_2\]
\item $\mathcal{A}^{\R}$ denotes the $\R$-motivic mod 2 Steenrod algebra, and we use the following notation for generators
\[\mathcal{A}^{\R}_{*,*}=\pi_{*,*}^\R M\F_2[\overline{t}_1,\overline{t}_2,\ldots,\tau_0,\tau_1,\ldots]/(\tau_i^2=\tau \overline{t}_{i+1}+\rho\tau_{i+1}+\rho\tau_0\overline{t}_{i+1})\]
where $|\overline{t}_i|=(2(2^i-1),2^i-1)$ and $|\tau_i|=(2(2^i-1)+1,2^i-1)$. We use the following notation for the quotient Hopf algebra
\[\mathcal{E}_{*,*}^{\R}(m)=\pi_{*,*}^\R M\F_2[\tau_0,\tau_1,\ldots,\tau_m]/(\tau_i^2=\rho\tau_{i+1})\]
Moreover, $c(-)$ denotes Hopf conjugation in $\mathcal{A}^{\R}_{*,*}$.
\item $\Sp$ denotes the category of spectra, and $\Sp^G$ denotes the category of genuine $G$-spectra. For a $G$-spectrum $X$, $\pi_\star(X)$ denotes the $RO(G)$-graded homotopy groups of $X$.
\item We let $i_*:\Sp\to \mathcal{SH}(k)$ denote the unique symmetric monoidal colimit-preserving functor and $\Gamma(-)$ its right adjoint. $\beta:\mathcal{SH}(\R)\to \Sp^{C_2}$ denotes the $C_2$-equivariant Betti realization functor.
\item Our spectral sequences are Adams trigraded, so that a class in tridegree $(t-s,w,s)$ is in stem $t-s$, weight $w$, and filtration $s$. All of the charts drawn in Section \ref{seccomputations} are in weight $w=0$.
\end{enumerate}

\section{The slice spectral sequence for generalized homology}\label{sec2}
In this section we define the HSSS in motivic stable homotopy and establish its basic properties. In Section \ref{subsecmotivicslice}, we work at the level of generality of $\mathcal{SH}(k)$, where $k$ is a perfect field, and we use the very effective slice filtration of Spitzweck-Ostvaer \cite{Spitz}. This allows us to have strong convergence in a wide generality, due to Morel's connectivity theorem. However, for the spectra we consider - namely, standard quotients of the algebraic cobordism spectrum $MGL$ - the corresponding slice towers for the effective, cellular effective, very effective, and cellular very effective all coincide, as shown by Heard \cite{Heard}.

This filtration determines a tower of motivic spectra, and our spectral sequences arise via smashing these towers pointwise with an ordinary spectrum, using that $\mathcal{SH}(k)$ is tensored over $\Sp$. This construction is analogous to one construction of the classical Atiyah--Hirzebruch spectral sequence, where one does the same with the Postnikov tower of a spectrum.

In Section \ref{subsecbetti}, we specialize to $\mathcal{SH}_{\mathrm{cell}}(\R)$, the cellular subcategory of $\R$-motivic spectra, where we use the results of Behrens and Shah \cite{BehrensShah} on Betti realization to compare with $C_2$-equivariant stable homotopy.

\subsection{General slice towers}\label{subsecmotivicslice}
We begin with the very effective slice filtration, a modification of Voevodsky's slice filtration \cite{Voevodsky}, developed by Spitzweck--Ostvaer \cite{Spitz} and further studied by Bachmann \cite{Bachmann}. The reader will notice we double the usual grading conventions, so that our slice towers are run at half speed; we explain this below in Remark \ref{remarkdoublegradings}. 

Following the discussion in \cite[Section 3]{Heard}, we let
\[\mathcal K_{2t}=\{\Sigma^{2a,a}\Sigma^{\infty}_+X;a\ge t,X\in\mathrm{Sm}_k\}\subset \mathcal{SH}(k)\]
and we set $\mathcal K_{2t+1}=\mathcal K_{2t}$. Letting $\mathcal{SH}(k)^{t,\mathrm{veff}}$ denote the full subcategory of $\mathcal{SH}(k)$ generated under colimits and extensions by $\mathcal K_t$, we have a filtration 
\[\cdots\subset \mathcal{SH}(k)^{t+1,\mathrm{veff}}\subset \mathcal{SH}(k)^{t,\mathrm{veff}}\subset \mathcal{SH}(k)^{t-1,\mathrm{veff}}\subset\cdots\]
of $\mathcal{SH}(k)$, and we say $E$ is slice $\ge t$ if $E\in \mathcal{SH}(k)^{t,\mathrm{veff}}$. Associated to any $E\in \mathcal{SH}(k)$, we have a natural cofiber sequence
\[P_{t+1}E\to E\to P^tE\]
such that $P_{t+1}E\in \mathcal K_{t+1}$, and 
\[\mathrm{Map}_{\mathcal{SH}(k)}(K,P^tE)\simeq*\]
for all $K\in \mathcal{SH}(k)^{t+1,\mathrm{veff}}$. We define the fiber $P^{t}_tE$ of the map
\[
    P^t E\to P^{t-1}E
\]
to be the \(t\)-slice of \(E\). The slice filtration gives a natural tower of spectra under \(E\):
\[
    \dots\to    P^t E\to P^{t-1}E\to\dots.
\]
The colimit is always contractible and the limit is \(E\).

Now, given any \(K\in \mathcal{SH}(k)\), we can smash the slice tower for \(E\) with \(K\) to produce a new filtered motivic spectrum \(K\otimes P^\bullet E\). Applying (bigraded) homotopy groups, this gives a spectral sequence.
\begin{definition}\label{definition:generalizedslicess}
    The \(K\)-homology slice spectral sequence for \(E\) is the spectral sequence associated to the filtered motivic spectrum \(K\otimes P^\bullet E\). We will generically denote this by 
    \[
        E_{\ast}^{s,w,t}(E;K).
    \]
    In particular, the \(E_2\)-term is given by 
    \[
        E_2^{s,w,t}(E;K)=K_{t-s,w}\big(P^t_t E\big).
    \]
    We use Adams grading $(t-s,w,s)$ so that the \(d_r\) differentials change tridegree by \((-1,0,r)\).
\end{definition}

\begin{remark}\label{remarkdoublegradings}
   We have chosen to double the gradings to cohere with two closely related spectral sequences when $k=\R$: slice spectral sequences in equivariant homotopy and the classical Atiyah--Hirzebruch spectral sequence. Applying Betti realization to the spectral sequences we consider recovers the analogous slice spectral sequences in $C_2$-equivariant homotopy (see Section \ref{secequivariant}), and if one base changes to $\mathbb C$ and inverts $\tau$, one recovers a classical Atiyah--Hirzebruch spectral sequence. Our choice of gradings recovers the standard grading conventions on these spectral sequences, in both cases.
\end{remark}

\begin{example}
When \(K=S^{0,0}\) is the sphere spectrum, this is the ordinary slice spectral sequence of $E$. When \(K=S^{0,0}[\rho^{-1}]\), this is the $\rho$-localized slice spectral sequence of $E$.
\end{example}

The ordinary slice spectral sequence converges strongly for any $E$, due to the Morel connectivity theorem \cite[Section 5.3]{Morel}. When we smash with an arbitrary motivic spectrum $K$, this is no longer guaranteed. For the following result, cf. \cite[Proposition 5.11]{Spitz}.

\begin{proposition}\label{propsliceconvergence}
    Suppose $K$ is slice bounded below. Then, the \(K\)-homology slice spectral sequence of $E$ converges conditionally in the sense that the natural map
    \[
        K\otimes E\to \lim_{\leftarrow} (K\otimes P^\bullet E)
    \]
    induces an isomorphism on bigraded homotopy groups, and $\mathrm{colim}(K\otimes P^\bullet E)\simeq*$.

    If, in addition, $E$ is slice bounded below, the spectral sequence converges strongly in the sense that   
    there are at most finitely many filtrations in which \(E_2\) is non-zero in any given stem.
\end{proposition}
\begin{proof}
    The spectrum $\mathrm{colim }P^\bullet E$ is contractible as the categories $\mathcal{K}_t$ form a set of compact generators of $\mathcal{SH}(k)$. Smashing with $K$, we use that the smash product commutes with colimits.
    
    One has a fiber sequence
    \[\lim_{\leftarrow}(K\otimes P_{t+1}E)\to K\otimes E\to \lim_{\leftarrow}(K\otimes P^tE)\]
    By assumption $K$ is slice $\ge n$ for some fixed $n$, so $K\otimes P_{t+1}E$ is slice $\ge n+t+1$ for all $t$. It follows from Morel's connectivity theorem that if $E\in \mathcal{SH}(k)^{t,\mathrm{veff}}$, then $\pi_{s,w}E=0$ for $s<t/2+w$. In a fixed weight $w$, therefore, $\pi_{s,w}(K\otimes P_{t+1}E)$ vanishes for $s<w+(n+t+1)/2$. This vanishing range for $s$ is strictly increasing in $t$, so the bigraded homotopy groups of the limit vanish. If $E$ is slice bounded below, then the same connectivity constraints imply that, in a fixed weight, only finitely many slices contribute to a given stem.
  \end{proof}

\begin{remark}
    Without the bound on \(K\), the result is much more delicate, even when \(E\) is \(0\)-connective, and the comparison map may fail to be an isomorphism. For example, when \(K=EGL(n)=\overline{v}_n^{-1}BPGL\langle n\rangle\) and \(E=MGL\), smashing the slice associated graded for \(E\) with \(K\) gives the zero spectrum.
\end{remark}

For our applications, both $K$ and $E$ will be slice connective, and this sharpens our convergence results to give strong vanishing lines.

\begin{proposition}\label{propvanishinglines}
    If $E$ and $K$ are slice $\ge0$, the $K$-homology slice spectral sequence of $E$ in weight zero is a right half-plane spectral sequence concentrated between the lines $y=x$ and $y=-x$. That is, the groups
    \[E_2^{s,0,t}(E;K)\]
    vanish if $s>t-s$ or $t<0$.
\end{proposition}
\begin{proof}
    If $E$ is slice $\ge0$, then $P^t_tE\simeq*$ for $t<0$, which gives the lower vanishing line. The upper vanishing line follows as above from Morel's connectivity theorem, which implies in particular that since $K\otimes P^t_tE$ is slice $\ge t$, $\pi_{t-s,0}(K\otimes P^t_tE)$ vanishes for $(t-s)<t/2$. This vanishing line is equivalent to \((t-s)<s\).
\end{proof}

The slice filtration is multiplicative, so we have the following.
\begin{proposition} 
    If \(K\) and \(E\) are both ring spectra, then this is a spectral sequence of algebras.
\end{proposition}

In fact, we have more structure that we can consider from naturality in \(K\). Since the spectral sequence arises from a filtration of \(E\) that is completely independent of \(K\), all of the structure commutes with \(K\)-cooperations.

\begin{proposition}
    If \(K\) is a flat homology theory, then the \(K\)-homology slice spectral sequence for any \(E\) is a spectral sequence of \(K_{*,*} K\)-comodules. If \(E\) is a ring, this is a spectral sequence of \(K_{*,*} K\)-comodule algebras.
\end{proposition}

\begin{example}\label{example:bredonSS}
Let $k=\R$ and $K=M\F_2$, the motivic mod 2 Eilenberg--Maclane spectrum. This gives a spectral sequence of $\mathcal A^\R_{*,*}$-comodules converging to the bigraded motivic homology of $E$ with $E_2$-page the bigraded motivic homology of the slice associated graded of $E$. If $E$ is a ring spectrum, this is a spectral sequence of $\mathcal A^\R_{*,*}$-comodule algebras.
\end{example}

\subsection{The homological slice spectral sequence (HSSS)}\label{subsechomologicalslice}

\subsubsection{Global sections homology}\label{subsubsecglobalsec} There is an essentially unique colimit-preserving symmetric monoidal functor
\[
    i_\ast\colon \Sp\to\mathcal{\mathcal{SH}}(k).
\]
This functor admits a right adjoint
\[\Gamma(-):\mathcal{\mathcal{SH}}(k)\to\Sp\]
which sends a motivic spectrum $E$ to the mapping spectrum
\[F(S^{0,0},E)\]
Viewing $E$ as a sheaf of spectra on $\mathrm{Sm}_k$, since $S^{0,0}=\Spec(k)_+$, we may think of this as taking the global sections of $E$. We now define our main spectral sequence of interest.

\begin{definition}
    For $E\in\mathcal{\mathcal{SH}}(k)$, the \textit{homological slice spectral sequence} (HSSS) of $E$ is the spectral sequence $E_*^{s,w,t}(E;i_*H\F_2)$ of Definition \ref{definition:generalizedslicess}.
\end{definition}

That is, we tensor the slice tower of $E$ with $i_*H\F_2$ and take the spectral sequence associated to the resulting tower. We begin with the following lemma, which follows immediately from the adjunction $i_*\dashv \Gamma$.

\begin{lemma}\label{lemma:homotopyofglobalsections}
    For $E\in\mathcal{\mathcal{SH}}(k)$, we have an isomorphism of abelian groups
    \[\pi_{s,w}E\cong \pi_s\Gamma(\Sigma^{0,-w}E)\]
    When $E$ is a ring spectrum, this is an isomorphism of bigraded rings.
\end{lemma}

Since the unit is compact in $\mathcal{\mathcal{SH}}(k)$, $i_*$ preserves compact objects and hence $\Gamma(-)$ admits a right adjoint. In particular $\Gamma(-)$ preserves colimits, and we deduce the following projection formula.

\begin{proposition}
    For any spectrum \(K\) and $E\in\mathcal{\mathcal{SH}}(k)$, we have
    \[
        \Gamma(i_\ast K\otimes E)\simeq K\otimes \Gamma(E)
    \]
\end{proposition}

Applying Lemma \ref{lemma:homotopyofglobalsections} in the case that \(K=H\F_2\), we have the following.

\begin{corollary}
    For any \(E\in\mathcal{\mathcal{SH}}(k)\), the homotopy groups of the spectrum \(\Gamma(i_\ast H\F_2\otimes E)\) are the mod \(2\) homology groups of \(\Gamma(E)\):
    \[
        (i_\ast H\F_2)_{*,0} E\cong H_\ast\Gamma(E)
    \]
\end{corollary}

Since \(H\F_2\) is an \(E_\infty\)-ring spectrum, so is \(i_\ast H\F_2\). Since \(H\F_2\otimes H\F_2\) is a tensor product of free associative algebras generated by spheres, the same is true upon taking \(i_\ast\).
\begin{proposition}
    We have an equivalence of associative \(i_\ast H\F_2\)-algebras
    \[
        i_\ast H\F_2^{\otimes 2}\simeq i_\ast H\F_2[\zeta_1,\dots],
    \]
    where just as classically, the topological degree of \(\zeta_i\) is \((2^i-1)\), and it is in weight 0.
\end{proposition}

In homotopy, the induced coaction on the \(\zeta_i\)'s is also the usual one, by functoriality. Put another way, the motivic spectrum \(i_\ast H\F_2\otimes E\) records all of the information of the mod \(2\) homology of \(\Gamma(E)\), even with the dual Steenrod coaction. 

\begin{corollary}
    The HSSS of $E$ is a spectral sequence of bigraded $\mathcal A_*$-comodules that converges to $H_*\Gamma(\Sigma^{0,-*}E)$. If $E$ is a ring spectrum, this is a spectral sequence of bigraded $\mathcal{A}_*$-comodule algebras.
\end{corollary}

\subsubsection{Global sections homology and motivic homology}\label{subsubseccanonicalmap}

The global sections functor is also lax monoidal, so given an \(E_\infty\)-monoid \(R\) in $\mathcal{\mathcal{SH}}(k)$, $\Gamma(R)$ is an \(E_\infty\)-ring spectrum. The counit of the adjunction $i_*\dashv \Gamma$ gives an \(E_\infty\)-map
\[
    i_\ast \Gamma(R)\to R.
\]
It follows from Lemma \ref{lemma:homotopyofglobalsections} that when \(R=M\Z\) or \(R=M\F_2\), the motivic Eilenberg--Maclane spectra, the global sections spectra are \(H\Z\) and \(H\F_2\), respectively. We therefore have canonical \(E_\infty\)-maps
\[
    i_\ast H\Z\to M\Z\text{ and }i_\ast H\F_2\to M\F_2
\]
We will use these to identify the slice associated graded in many cases.

Applying the functor $i_*$ to ordinary homotopy groups, we have a canonical map of algebras
\[
\mathcal A_*=\pi_\star(H\F_2\otimes H\F_2)\xrightarrow{i_*}\pi_{*,*}
(i_*H\F_2\otimes i_*H\F_2)\to \pi_{*,*}(M\F_2\otimes M\F_2)=:\mathcal A^k_{*,*}\]
If the latter dual Steenrod algebra is flat over $\pi_{*,*}^k(M\F_2)$, we have a Hopf algebroid structure, and our map of algebras is compatible with the coactions in the following sense.
\begin{proposition}
    If \(E\in\mathcal{\mathcal{SH}}(k)\), then we have a natural map
    \[
        H_*\Gamma(E)\to H_{*,*}(E;M\F_2)
    \]
    When $\mathcal A^k_{*,*}$ is flat, this is a map of $\mathcal A^k_{*,*}$-comodules via the map
    \[
        \cA_\ast\to \mathcal A^k_{*,*}
    \]
\end{proposition}

\subsection{Betti realization and the arithmetic square}\label{subsecbetti} In this section, we restrict our attention now to $k=\R$ and work in the cellular context. That is, we work in the category $\mathcal{SH}_{\mathrm{cell}}(\R)$, the localizing subcategory of $\mathcal{SH}(\R)$ generated by the bigraded spheres $S^{s,w}$. Moreover, we implicitly work in the $2$-complete context everywhere, so that our $\mathcal{SH}_{\mathrm{cell}}(\R)$ stands for $\mathcal{SH}_{\mathrm{cell}}(\R)^{\complete}_2$, and our $\Sp^{C_2}$ stands for $(\Sp^{C_2})^{\complete}_2$.

There is a strong relationship between $\R$-motivic stable homotopy and $C_2$-equivariant stable homotopy, which comes via the Betti realization functor
\[\beta:\mathcal{SH}(\R)\to\Sp^{C_2}\]
If $X$ is a smooth scheme over $\Spec(\R)$, its set of complex points $X(\mathbb C)$ can be equipped with the complex analytic topology, and the action of $C_2=\mathrm{Gal}(\mathbb C/\R)$ on $X(\mathbb C)$ is continuous. The functor $\beta$ is obtained by left Kan extension of $X\mapsto \Sigma^\infty_+X(\mathbb C)$ along the canonical functor $\mathrm{Sm}_{\R}\to \mathcal{SH}(\R)$.

Pushing and pulling computations along the functor $\beta$ has been incredibly fruitful (see \cite{DI}, \cite{BGI}, and \cite{GHIR} for example). There has therefore been an effort to make the connection provided by $\beta$ conceptually precise (see  \cite{DI}, \cite{HellerOrmsby}, \cite{bachmann2}, and \cite{BehrensShah}). Behrens--Shah combined and extended these results to characterize $\beta$ as a localization functor on $\mathcal{SH}_{\mathrm{cell}}(\R)$. They showed, in particular, that after completing at a prime, the functor
\[\mathrm{Sing}:\Sp^{C_2}\to \mathcal{SH}_{\mathrm{cell}}(\R)\]
is a fully faithful right adjoint to $\beta$, and we may therefore regard $\Sp^{C_2}$ as a reflective subcategory of $\mathcal{SH}_{\mathrm{cell}}(\R)$ with reflection functor given by $\beta$ \cite[Theorem 1.12]{BehrensShah}.

On the other hand, $\Sp^{C_2}$ may be described in terms of pullback squares. For $E\in \Sp^{C_2}$, there is a commutative square
\[
\begin{tikzcd}
    E\arrow[r]\arrow[d]&\tilde{E}C_2\otimes E\arrow[d]\\
    F(E{C_2}_+,E)\arrow[r]&\tilde{E}C_2\otimes F(E{C_2}_+,E)
\end{tikzcd}
\]
known as the \textit{Tate square} of $E$, which is a pullback. Behrens--Shah \cite[Theorem 1.10]{BehrensShah} provided an extension of this construction to all $E\in \mathcal{SH}_{\mathrm{cell}}(\R)$, forming a commutative square
\[
\begin{tikzcd}
    E\arrow[r]\arrow[d]& E[\rho^{-1}]\arrow[d]\\
    E^{\complete}_\rho[\tau^{-1}]\arrow[r]&E^{\complete}_\rho[\tau^{-1}][\rho^{-1}]
\end{tikzcd}
\]
which we call the \textit{arithmetic square} of $E$. This square is \textit{not} in general a pullback, but the results of Behrens--Shah imply the following.

\begin{proposition}\label{proparithmeticsquare}
    $E\in \mathcal{SH}_{\mathrm{cell}}(\R)$ is in the essential image of the fully faithful right adjoint
    \[\mathrm{Sing}:\Sp^{C_2}\to \mathcal{SH}_{\mathrm{cell}}(\R)\]
    if and only if the arithmetic square of $E$ is a pullback. In particular, the functor $\beta(-)$ may be regarded as the endofunctor of $\mathcal{SH}_{\mathrm{cell}}(\R)$ given by the pullback
    \[
\begin{tikzcd}
    \beta(E)\arrow[r]\arrow[d]& E[\rho^{-1}]\arrow[d]\\
    E^{\complete}_\rho[\tau^{-1}]\arrow[r]&E^{\complete}_\rho[\tau^{-1},\rho^{-1}]
\end{tikzcd}
\]

\end{proposition}

This, in turn, describes the genuine fixed points functor $(-)^{C_2}:\Sp^{C_2}\to \Sp$ as a special case of the global sections functor of Section  \ref{subsubsecglobalsec}.

\begin{proposition}\label{propfixedpointsglobalsec}
    The genuine fixed points functor $(-)^{C_2}$ is equivalent to the composite $\Gamma\circ \mathrm{Sing}$.
\end{proposition}

The unit map for the reflection functor $\beta$ gives a natural transformation
\[\eta_E:\Gamma(E)\to \Gamma(\beta(E))\]
We will show that $\eta_E$ is an equivalence in many cases of interest; that is, the arithmetic square of $E$ becomes a pullback after applying $\Gamma(-)$ in these cases. This follows by induction up the slice tower of $E$, given the following base case.

\begin{proposition}\label{propbetaMZ}
The map
\[\Gamma(\Sigma^{s,w}M\Z)\to\Gamma(\Sigma^{s,w}\beta(M\Z))\]
is an equivalence if $w>-2$. Equivalently, the map
\[\pi_{*,*}^{\R}M\Z\to\pi_{*,*}^{\R}\beta(M\Z)\]
is an iso in weights $<2$.
\end{proposition}
\begin{proof}
On one hand, this is immediate from Lemma \ref{lemma:homotopyofglobalsections} and the fact that $\beta$ induces an isomorphism
\[\pi_{*,*}^\R M\Z\to\pi_{*,*}^{C_2}H\mZ\]
in weights $<2$ (see \cite{DI}).

However, we may argue directly. From Voevodsky's computation \cite{voe}
\[\pi_{*,*}^\R M\Z=\Z[\tau^2,\rho]/(2\rho)\]
where $|\tau|=(0,-2)$ and $|\rho|=(-1,-1)$, one sees that applying bigraded homotopy groups to the arithmetic square of $M\Z$ gives the square
\[
\begin{tikzcd}
    \Z[\tau^2,\rho]/(2\rho)\arrow[d]\arrow[r]&\F_2[\tau^2,\rho^{\pm}]\arrow[d]\\
    \Z[\tau^{\pm 2},\rho]/(2\rho)\arrow[r]&\F_2[\tau^{\pm 2},\rho^{\pm}]
\end{tikzcd}
\]
The resulting Mayer--Vietoris sequence is exact in weights $w<2$, and the result now follows from Lemma \ref{lemma:homotopyofglobalsections} and Proposition \ref{proparithmeticsquare}.
\end{proof}

In particular, we also have that $\eta_E$ is an equivalence when $E$ is of the form $\Sigma^{2t,t}M\Z$ for $t\ge0$. If $E$ is slice connective with slices of this form, a similar argument to that of Proposition \ref{propsliceconvergence} gives the following.

\begin{corollary}\label{cor:betaE}
    Suppose $E\in \mathcal{SH}_{\mathrm{cell}}(\R)$ has the property that, for all $t\ge0$,
    \[P^{2t}_{2t}E=\bigoplus\limits_{I_t}\Sigma^{2t,t}M\Z\]
    and $P_{2t}^{2t}E$ vanishes for $t<0$. Then 
    \[\Gamma(\Sigma^{s,w}E)\to \Gamma(\Sigma^{s,w}\beta(E))\]
    is an equivalence for $w>-2$. Equivalently, the induced map
    \[\pi_{*,*}^\R E\to\pi_{*,*}^\R\beta(E)\]
    is an isomorphism in weights $w<2$.
\end{corollary}

\begin{corollary}\label{corhomologygammaE}
Suppose $E\in \mathcal{SH}_{\mathrm{cell}}(\R)$ is as in Corollary \ref{cor:betaE}. Then
\[H_*\Gamma(\Sigma^{s,w}E)\to H_*\Gamma(\Sigma^{s,w}\beta(E))\]
is an isomorphism for $w>-2$. Equivalently, the map
    \[\pi_{*,*}^\R (i_*H\F_2\otimes E)\to\pi_{*,*}^\R (i_*H\F_2\otimes\beta(E))=\pi_{*,*}^{\R}(\beta(i_*H\F_2\otimes E))\]
    is an isomorphism in weights $w<2$.
\end{corollary}

In particular, combining Corollary \ref{corhomologygammaE} with Proposition \ref{proparithmeticsquare}, one may compute the bigraded homology of $\Gamma(E)$ in weights $w<2$ in terms of that of the motivic spectra
\begin{align*}
&\Gamma(E[\rho^{-1}])
&&\Gamma(E^{\complete}_\rho[\tau^{-1}]) &&\Gamma(E^{\complete}_\rho[\tau^{-1},\rho^{-1}])
\end{align*}
by applying the Mayer--Vietoris sequence. We will apply this in Section \ref{secarithmeticsquare} for $E=BPGL$, the $2$-local summand of the algebraic cobordism spectrum. 

Recall there is a classifying map $L\to MGL_{*,*}^{\R}$ that is an isomorphism in bidegrees of the form $(2k,k)$, where $L\cong \Z[a_1,a_2,\ldots]$ is the Lazard ring, and $|a_i|=(2i,i)$ (see \cite{Heard} for this isomorphism and \cite{ravgreen} for the description of the Lazard ring). We say an $MGL$-module is a \textit{standard quotient of }$MGL$ if it is equivalent to the quotient of $MGL$ by some collection of polynomials that are $a_i$ modulo decomposables. Standard quotients of $MGL$ satisfy the hypotheses of Corollary \ref{cor:betaE}, by the following theorem of Hopkins--Morel and Levine--Tripathi \cite{levine}.

\begin{proposition}\label{propslicebp}
    Let $I\subset L$ be the ideal generated by some collection of the $a_i$'s. Then the slice associated graded of $MGL/I$ is
    \[M\Z[a_1,a_2,\ldots]/I:=M\Z[a_i|i\notin I]\]
\end{proposition}

Working $2$-locally, $MGL$ splits as a sum of shifts of the spectrum
\[BPGL:=MGL/(a_i|i\neq 2^k-1)\]
As is standard, we let $\ov_i:=a_{2^i-1}$, and our computations with the HSSS will focus on the truncated versions
\[BPGL\langle m\rangle=BPGL/(\ov_{m+1},\ov_{m+2},\ldots)\]

\subsection{The edge homomorphism}\label{subsecedgehom}

For any \(E\in \mathcal{SH}_{\mathrm{cell}}(\R)\), we have a natural map of motivic spectra
\[
    E\to P^k E,
\]
and so we can consider the induced map of slice towers. Since we have a map of filtered spectra, we have an induced map of spectral sequences. The slice spectral sequence for \(P^kE\) is exactly the truncation of the slice spectral sequence for \(E\) where we throw away all classes in degrees \((t-s,s)\) with \(t>k\). 

There is a very important case of the truncated slice tower: the case \(k=0\). For standard quotients $E$ of $MGL$, Proposition \ref{propslicebp} implies that $P^0 E\simeq M\Z$. We therefore have a map of HSSS's
\[
E_\ast^{s,w,t}\big(E;i_*H\F_2\big)\to E_\ast^{s,w,t}\big(M\Z;i_*H\F_2\big).
\]
The latter is concentrated on the line \(t=0\), but recognizing that this comes from a map of filtered spectra, we see that this gives the edge homomorphism
\[
\pi_{*,*}^{\R}(i_*H\F_2\otimes E)\to H_{*,*}^{\R}\big(i_*H\F_2;\Z\big).
\]
Again, this map is a map of \(\mathcal A^{\R}_{*,*}\)-comodule (algebras). By definition, every class in the image of the edge homomorphism admits a lift along this map, so we have the following. 

\begin{proposition}\label{propedgeispermanent}
    The image of the edge homomorphism in the HSSS consists of permanent cycles.
\end{proposition}

\section{Homology of $\Gamma(BPGL\langle m\rangle)$ via the arithmetic square}\label{secarithmeticsquare}
We have seen in the previous section that computing the homology groups $H_*\Gamma(E)$ amounts to computing the motivic homotopy groups
\[\pi_{*,0}^{\R}(i_*H\F_2\otimes E)\]
and that, when $E$ is as in Proposition \ref{propslicebp}, this $\mathcal{A}_*$-comodule may be recovered from the arithmetic square of $i_*H\F_2\otimes E$. In this section, we explore this approach in the case $E=BPGL\langle m\rangle$. 
%
%
%

\subsection{The arithmetic square of $i_*H\F_2\otimes BPGL\langle m\rangle$}

\subsubsection{The $\tau$-local part}\label{subsecBredon}
In this section, we determine bottom row of the arithmetic square of $i_*H\F_2\otimes BPGL\langle m\rangle$, namely the map
\[\pi_{*,*}^{\R}(i_*H\F_2\otimes BPGL\langle m\rangle)^{\complete}_\rho[\tau^{-1}]\to \pi_{*,*}^{\R}(i_*H\F_2\otimes BPGL\langle m\rangle)^{\complete}_\rho[\tau^{-1},\rho^{-1}]\]
The main advantage of applying the functor $(-)^{\complete}_\rho[\tau^{-1}]$ is that it converts $i_*H\F_2$ homology to $M\F_2$-homology.

\begin{proposition}
\label{borelhomology}
The canonical map $i_*H\F_2\to M\F_2$ induces an equivalence
\[(i_*H\F_2\otimes E)^{\complete}_\rho[\tau^{-1}]\to (M\F_2\otimes E)^{\complete}_\rho[\tau^{-1}]\]
for any $E\in\mathcal{\mathcal{SH}}_{\mathrm{cell}}(\R)$.
\end{proposition}
\begin{proof}
The map in question is obtained by taking the limit of the maps
\[i_*H\F_2\otimes E\otimes C(\rho^i)[\tau^{-1}]\to M\F_2\otimes E\otimes C(\rho^i)[\tau^{-1}]\]
(see \cite[Theorem 1.10]{BehrensShah}). By filtering $C(\rho^i)$, it suffices to show this map is an equivalence when $i=1$. By \cite[Proposition 8.3]{BehrensShah}, the map
\[C(\rho)\to \Spec(\mathbb C_+)\]
becomes an equivalence after applying $2$-completion and cellularization. It therefore suffices to show that the map
\[i_*H\F_2\otimes \Spec(\mathbb C_+)[\tau^{-1}]\to M\F_2\otimes \Spec(\mathbb C_+)[\tau^{-1}]\]
is an equivalence, for which it suffices to base change to $\mathbb C$ and show that 
\[i_*H\F_2[\tau^{-1}]\to M\F_2[\tau^{-1}]\]
is an equivalence in $\mathcal{SH}(\mathbb C)$, by \cite[Theorem 1.7]{BehrensShah}. Finally, by \cite[Theorem 1.1]{BehrensShah}, it suffices to apply $\mathbb C$-Betti realization, which gives the identity map of $H\F_2$.
\end{proof}

The motivic homology of $BPGL\langle m\rangle$ for $-1\le m\le \infty$ is described similarly to the classical case. We have the following computation of Ormsby \cite[Theorem 3.8]{ormsby}.

\begin{proposition}\label{prop:motivichomologybprm}
    The motivic homology of $BPGL\langle m\rangle$ is
    \[H_{*,*}^{\R}BPGL\langle m\rangle\cong \mathcal{A}_{*,*}^{\R}\square_{\mathcal{E}_{*,*}^{\R}(m)}{M\F_2}_{*,*}=\frac{\F_2[\rho,\tau,\overline{t}_i,c(\tau_j)|i\ge1,j\ge m+1]}{(c(\tau_j)^2=c(\tau_{j+1})\rho+c(\overline{t}_{j+1})\tau)}\]
\end{proposition}

\begin{corollary}\label{cor:bottomrow}
The map
\[
\pi_{*,*}^{\R}(i_*H\F_2\otimes BPGL)^{\complete}_\rho[\tau^{-1}]\to\pi_{*,*}^{\R}(i_*H\F_2\otimes BPGL)^{\complete}_\rho[\tau^{-1},\rho^{-1}]
\]
is given by the inclusion
\[
\frac{\F_2[\rho,\tau^{\pm},\overline{t}_i,c(\tau_j)|i\ge1,j\ge m+1]}{(c(\tau_j)^2=c(\tau_{j+1})\rho+c(\overline{t}_{j+1})\tau)}^{\complete}_{\rho} \hookrightarrow \frac{\F_2[\rho,\tau^{\pm},\overline{t}_i,c(\tau_j)|i\ge1,j\ge m+1]}{(c(\tau_j)^2=c(\tau_{j+1})\rho+c(\overline{t}_{j+1})\tau)}^{\complete}_{\rho}[\rho^{-1}] 
\]
\end{corollary}
\begin{proof}
By Proposition \ref{borelhomology}, we may prove this claim instead for the map
\[
\pi_{*,*}^{\R}(M\F_2\otimes BPGL)^{\complete}_\rho[\tau^{-1}]\to\pi_{*,*}^{\R}(M\F_2\otimes BPGL)^{\complete}_\rho[\tau^{-1},\rho^{-1}]
\]
which follows from Proposition \ref{prop:motivichomologybprm}.
\end{proof}


\begin{remark}
    Ormsby's computation may be recovered using the spectral sequence of Section \ref{subsecmotivicslice} with $K=M\F_2$ and $E=BPGL\langle m\rangle$. This computation mirrors exactly the classical computation of $H_*BP\langle m\rangle$ using the Atiyah--Hirzebruch spectral sequence.
\end{remark}

\subsubsection{The $\rho$-local part}\label{subsecTate}
Corollary \ref{cor:bottomrow} determines the bottom row of the arithmetic square for $i_*H\F_2\otimes BPGL\langle m\rangle$. The key input for this calculation was the fact that, working $\tau$-locally, $i_*H\F_2$ homology agrees with $M\F_2$-homology. On the other hand, working $\rho$-locally, we have the following \cite[Theorem 1.2]{BehrensShah}.

\begin{proposition}
   The functor
   \[\Gamma:\mathcal{SH}(\R)_{\mathrm{cell}}[\rho^{-1}]\to \Sp\]
   is an equivalence of categories with inverse given by $i_*(-)[\rho^{-1}]$.
\end{proposition}

The following is easy to obtain by appealing to $C_2$-equivariant homotopy and using \cite[Theorem 1.5]{BehrensShah} and \cite[Proposition 4.9]{GreenleesMeier}, or directly by simply adapting the arguments of \cite[Proposition 4.9]{GreenleesMeier} to the motivic setting.

\begin{lemma}\label{lemmarholocalbpgl}
    Under this equivalence, $\Gamma$ sends $BPGL[\rho^{-1}]$ to $H\F_2$. Coning off $(\ov_{m+1},\ov_{m+2},\ldots)$, we have $\Gamma(BPGL\langle m\rangle[\rho^{-1}])\simeq H\F_2[y^{2^{m+1}}]$, as an $H\F_2$-module, where $|y|=1$.
\end{lemma}
 This tells us that the gluing map in the arithmetic square for $i_*H\F_2\otimes BPGL$
\[i_*H\F_2\otimes BPGL[\rho^{-1}]\to (i_*H\F_2\otimes BPGL)^{\complete}_\rho[\tau^{-1},\rho^{-1}]\]
is given on bigraded homotopy by some ring map of the form
\[\mathcal A_*[\rho^{\pm}]\to\F_2[\rho,\tau^{\pm},\overline{t}_i]^{\complete}_{\rho}[\rho^{-1}]\]
To determine this gluing map, we can embed both groups into
\[\pi_{*,*}^\R(M\F_2\otimes M\F_2)^{\complete}_\rho[\tau^{-1},\rho^{-1}]\]
via the Thom reduction map $BPGL\to M\F_2$ and the canonical map $i_*H\F_2\to M\F_2$. 

Understanding these embeddings reduces to understanding the composite 
\[i_*H\F_2[\rho^{-1}]\simeq BPGL[\rho^{-1}]\to M\F_2[\rho^{-1}]\]
Much of the subtlety in our computations comes from the observation that this map is not homotopic to the canonical map of \ref{subsubseccanonicalmap}
\[i_*H\F_2[\rho^{-1}]\to M\F_2[\rho^{-1}]\]
The two maps of course have the same effect in homotopy; we need to pass to homology to tell them apart. The following allows us to describe the effect of the canonical map in homology.

\begin{proposition}\label{propcounitimage}
    Regarding the classical Milnor generators $\xi_i$ as elements of $\mathcal{A}_{*,*}^{\R}$ via the composite
    \[\mathcal{A}_*=\pi_*(H\F_2\otimes H\F_2)\xrightarrow{i_*} \pi_{*,*}^{\R}(i_*H\F_2\otimes i_*H\F_2)\to \pi_{*,*}^{\R}(M\F_2\otimes M\F_2)=\mathcal{A}_{*,*}^{\R}\]
    one has the recursion formulas
    \aln{
    \overline{t}_{0}&:=1\\
    \rho^{2^i}\overline{t}_i&=\xi_{i-1}^2(\tau_0\rho+\tau)+\xi_i\rho+\overline{t}_{{i-1}}\tau^{2^{i-1}}
    }
\end{proposition}
\begin{proof}
    We appeal to $C_2$-equivariant homotopy, where the corresponding formulas were computed by Hu--Kriz \cite[Theorem 6.18]{HK}. The result follows $\R$-motivically by the factorization
    \[
    \begin{tikzcd}
        \mathcal A_*\arrow[r]\arrow[dr]&\mathcal{A}_{*,*}^{\R}\arrow[d,hookrightarrow]\\
        &\mathcal{A}_{*,*}^{C_2}
    \end{tikzcd}
    \]
\end{proof}

Let can denote the equivalence
\[\mathrm{can}:H\F_2[y]\xrightarrow{\simeq} \Gamma(M\F_2[\rho^{-1}])\]
where $H\F_2[y]$ is the free $E_1$-$H\F_2$-algebra on a class $y$ in degree 1, and can is defined by giving $\Gamma(M\F_2[\rho^{-1}])$ an $H\F_2$-algebra structure via the canonical map of \ref{subsubseccanonicalmap} and by sending $y$ to $\tau/\rho\in\pi_1\Gamma(M\F_2[\rho^{-1}])$.

\begin{proposition}\label{propnotunit}
The map
\[H\F_2\simeq\Gamma(BPGL[\rho^{-1}])\to\Gamma(M\F_2[\rho^{-1}])\xrightarrow{\mathrm{can}^{-1}}H\F_2[y]\]
is not homotopic to the unit map $\eta:H\F_2\to H\F_2[y]$. 
\end{proposition}

\begin{proof}
If it were, the composite 
\[i_*H\F_2\otimes i_*H\F_2[\rho^{-1}]\simeq i_*H\F_2\otimes BPGL[\rho^{-1}]\to M\F_2\otimes M\F_2[\rho^{-1}]\]
would be homotopic to the canonical map
\[i_*H\F_2\otimes i_*H\F_2[\rho^{-1}]\to M\F_2\otimes M\F_2[\rho^{-1}]\]
By Proposition \ref{propcounitimage}, this map sends the class $\xi_1$ to $\rho\overline{t}_1+\tau_0$. This would imply that $\tau_0$ is in the image of the (injective) map 
\[\pi_{*,*}^{\R}(M\F_2\otimes M\Z)^{\complete}_\rho[\tau^{-1},\rho^{-1}]\to\pi_{*,*}^{\R}(M\F_2\otimes M\F_2)^{\complete}_\rho[\tau^{-1},\rho^{-1}]\]
since the reduction map $BPGL\to M\F_2$ factors through $M\Z$.
Voevodsky's relations 
\[\tau_i^2=\tau \overline{t}_{i+1}+\rho\tau_{i+1}+\rho\tau_0\overline{t}_{i+1}\]
would then imply that all $\tau_i$'s are in the image of this map, so that it is an isomorphism, a contradiction.
\end{proof}


\begin{remark}\label{hukrizerror}
In \cite{BPO}, Hu--Kriz compute the homology groups $H_*BP_\R^{C_2}$ using the isotropy separation sequence. By Corollary \ref{cor:betaE}, this is the same as computing 
\[\pi_{*,0}^{\R}(i_*H\F_2\otimes BPGL)\]
Our computations differ from theirs, and this is related to what seems to be an error on page 114, where they claim that $\eta_R(u_\sigma^{-1})$ is in the image of the reduction map
\[f:\pi_\star(\tilde{E}C_2\otimes F(E{C_2}_+,H\underline{\F_2}\otimes BP_\R))\to \pi_\star(\tilde{E}C_2\otimes F(E{C_2}_+,H\underline{\F_2}\otimes H\underline{\F_2}))\]
In their notation, the claim is that $f(\rho)=\eta_R(\sigma)$. As in the proof of the previous proposition, this would imply that $\tau_0$ is in the image of this map, which leads to contradictions.
\end{remark}

The composite in Proposition \ref{propnotunit} gives \textit{some} ring map 
\[H\F_2\to H\F_2[y]\]
and we identify it by showing it is structured enough to apply Mahowald's theorem on $H\F_2$ \cite{mahow}.

\begin{theorem}\label{thmtatefrob}
The composite
\[H\F_2\simeq\Gamma(BPGL[\rho^{-1}])\to \Gamma(M\F_2[\rho^{-1}])\xrightarrow{\mathrm{can}^{-1}}H\F_2[y]\simeq\prod\limits_{i\ge0}\Sigma^i H\F_2\]
has components $Sq^i:H\F_2\to \Sigma^i H\F_2$.
\end{theorem}
\begin{proof}
We note first that the map
\[H\F_2=\Gamma(i_*H\F_2[\rho^{-1}])\simeq\Gamma(BPGL[\rho^{-1}])\to\Gamma(M\F_2[\rho^{-1}])\]
is an $E_\infty$ map since it may be factored as
\[\Gamma(i_*H\F_2[\rho^{-1}])\simeq\Gamma(BPGL[\rho^{-1}])\to\Gamma(MGL_{(2)}[\rho^{-1}])\to \Gamma(M\F_2[\rho^{-1}])\]
The first map is the unit map $H\F_2\to MO$, which is $E_\infty$, and the map $MGL_{(2)}\to M\F_2$ is $E_\infty$.

Mahowald's theorem on $H\F_2$ implies that there is a pushout square in $E_2$-algebras in $\Sp$
\[
\begin{tikzcd}
\text{Free}_{E_2}(\mathbb S)\arrow[r,"\overline{0}"]\arrow[d,"\overline{2}"']&\mathbb S\arrow[d]\\
\mathbb S\arrow[r]&H\F_2
\end{tikzcd}
\]
which gives
\[\pi_0\text{Map}_{E_2}(H\F_2,\Gamma(M\F_2[\rho^{-1}]))\cong\pi_1\Gamma(M\F_2[\rho^{-1}])=\F_2\]
There is thus a unique such $E_2$-map not homotopic to the unit map $\eta$. We conclude by observing that the Tate-valued Frobenius 
\[\varphi_2:H\F_2\to H\F_2^{tC_2}\]
of Nikolaus--Scholze is an $E_\infty$ map that factors thru the connective cover 
\[H\F_2[y]\simeq\Gamma(M\F_2[\rho^{-1}])\simeq\Phi^{C_2}H\underline{\F_2}\]
and admits the above description by \cite[IV.1.5]{NS}.
\end{proof}



\subsection{The Mayer--Vietoris sequence for $i_*H\F_2\otimes BPGL$}\label{subsecMVseq}
Theorem \ref{thmtatefrob} may be used to describe explicitly the gluing map in the arithmetic square for $i_*H\F_2\otimes BPGL$. To determine its effect on homotopy, we need the following lemma.
\begin{lemma}\label{lemmacapproduct}
The map induced on homotopy by 
\[H\F_2\otimes H\F_2\xrightarrow{1\otimes Sq^i}H\F_2\otimes \Sigma^i H\F_2\] 
is the cap product
\[-\cap Sq^i:\mathcal A_*\xrightarrow{\Delta}\mathcal A_*\otimes\mathcal A_*\xrightarrow{1\otimes\langle-,Sq^i\rangle}\mathcal A_{*-i}\]
\end{lemma}
\begin{proof}
The map induced on homotopy by $1\otimes Sq^i$ is an $\mathcal A_*$-comodule map, and $\mathcal A_{*-i}$ is a cofree $\mathcal A_*$-comodule. It suffices to show the composite 
\[\mathcal A_*\xrightarrow{(1\otimes Sq^i)_*}\mathcal A_{*-i}\xrightarrow{\epsilon}\F_2[i]\]
coincides with the composite
\[\mathcal A_*\xrightarrow{\Delta}\mathcal A_*\otimes\mathcal A_*\xrightarrow{1\otimes\langle-,Sq^i\rangle}\mathcal A_{*-i}\xrightarrow{\epsilon}\F_2[i]\]
where $\epsilon:\mathcal A_*\to \F_2$ is the coidentity map in the Hopf algebra $\mathcal A_*$. Under the isomorphism $\mathcal A\cong\Hom_{\F_2}(\mathcal A_*,\F_2)$, the latter map corresponds to $Sq^i$. The same is true of the former map as the pairing $\mathcal A_*\otimes_{\F_2}\mathcal A\to \F_2$
is induced by the map in $\Sp$
\[(H\F_2\otimes H\F_2)\otimes_{H\F_2}F(H\F_2,H\F_2)\to H\F_2\otimes H\F_2\xrightarrow{\mu}H\F_2\qedhere\]
\end{proof}

\begin{proposition}\label{proprightmap}
The map
\[\pi_{*,*}^{\R}(i_*H\F_2\otimes BPGL)[\rho^{-1}]\to \pi_{*,*}^{\R}(M\F_2\otimes BPGL)^{\complete}_\rho[\tau^{-1},\rho^{-1}]\]
is given by the map
\[\varphi:\mathcal A_*[\rho^{\pm}]\to \F_2[\tau^{\pm},\rho,\overline{t}_i]^{\complete}_{\rho}[\rho^{-1}]\]
of $\F_2[\rho^{\pm}]$-algebras sending
\[\varphi(\xi_i)=\rho^{2^i-1}\overline{t}_i+\Big(\frac{\tau^{2^{i-1}}}{\rho}\Big)\overline{t}_{{i-1}}\] 
\end{proposition}
\begin{proof}
The diagram
\[
\begin{tikzcd}
(i_*H\F_2\otimes BPGL)[\rho^{-1}]\arrow[d]\arrow[r]&(M\F_2\otimes M\F_2)[\rho^{-1}]\arrow[d]\\
(M\F_2\otimes BPGL)^{\complete}_\rho[\tau^{-1},\rho^{-1}]\arrow[r]&(M\F_2\otimes M\F_2)^{\complete}_\rho[\tau^{-1},\rho^{-1}]
\end{tikzcd}
\]
is given on bigraded homotopy by
\[
\begin{tikzcd}
\mathcal A_*[\rho^{\pm}]\arrow[d,"\varphi"]\arrow[r]&\mathcal A_*[\tau,\eta_R(\tau),\rho^{\pm}]\arrow[d]\\
\F_2[\rho,\tau^{\pm},\overline{t}_i]^{\complete}_{\rho}[\rho^{-1}]\arrow[r]&\mathcal A_*[\tau^{\pm},\eta_R(\tau)^{\pm},\rho]^{\complete}_{\rho}[\rho^{-1}]
\end{tikzcd}
\]
Note that the top map is \textit{not} a map of $\mathcal A_*$-algebras; this is the content of Proposition \ref{propnotunit}. The bottom and righthand maps are injective, so it suffices to determine the top map. Combining Theorem \ref{thmtatefrob} and Lemma \ref{lemmacapproduct}, this is the map of $\F_2[\rho^{\pm}]$-algebras sending
\[\xi_i\mapsto\sum\limits_{j\ge0}(\xi_i\cap Sq^j)\cdot\bigg(\frac{\eta_R(\tau)}{\rho}\bigg)^j\]
Using the coproduct formula on $\mathcal A_*$ and the fact that $Sq^j$ is dual to $\xi_1^j$, we have
\[\xi_i\mapsto \xi_i+\xi_{i-1}^2\frac{\eta_R(\tau)}{\rho}\]
We conclude by applying the relations of Proposition \ref{propcounitimage}
\end{proof}

We move now to weight zero and set $\chi_i:=\tau^{2^i-1}\overline{t}_i$ and $z:=\rho/\tau=y^{-1}$. We may summarize the results of Section \ref{secarithmeticsquare} thus far as follows.

\begin{proposition}\label{proparithmeticsquarecomputation}
The arithmetic square of $i_*H\F_2\otimes BPGL$ is given by
\[
\begin{tikzcd}
H_*\Gamma(BPGL)\arrow[r]\arrow[d]&\mathcal A_*\arrow[d,"\varphi"]&\xi_i\arrow[d,mapsto]\\
\F_2[z,\chi_i|i\ge1]^{\complete}_{z}\arrow[r,hookrightarrow]&\F_2[z,\chi_i|i\ge1]^{\complete}_{z}[z^{-1}]&\chi_iz^{2^i-1}+\chi_{i-1}z^{-1}
\end{tikzcd}
\]
with $|z|=-1$ and $|\chi_i|=2(2^i-1)$. 
\end{proposition}

To compute the associated Mayer--Vietoris sequence, we establish the following.

\begin{proposition}\label{propbndmap}
The arithmetic square of Proposition \ref{proparithmeticsquarecomputation} satisfies the following properties:
\begin{enumerate}
\item The map $\varphi$ is injective.
\item The map $\F_2[z,\chi_i|i\ge 1]^{\complete}_{z}\oplus\mathcal A_*\to \F_2[z,\chi_i|i\ge 1]^{\complete}_{z}[z^{-1}]$ is injective in nonzero degrees.
\item The maps 
\[H_*\Gamma(BPGL)\to H_*\Gamma(BPGL[\rho^{-1}])=\mathcal A_*\]
and
\[H_*\Gamma(BPGL)\to \pi_{*,0}^{\R}(i_*H\F_2\otimes BPGL)^{\complete}_\rho[\tau^{-1}]\]
are both zero in nonzero degrees
\item In positive degrees $j$, the following sequence is short exact 
\end{enumerate}
\adjustbox{scale=1,center}{
$0\to (\F_2[z,\chi_i|i\ge 1]^{\complete}_{z})_{j+1}\oplus\mathcal A_{j+1}\to (\F_2[z,\chi_i|i\ge 1]^{\complete}_{z}[z^{-1}])_{j+1}\xrightarrow{\partial}H_{j}\Gamma(BPGL)\to0$
}

\end{proposition}
\begin{proof}
For (1), we have a diagram
\[
\begin{tikzcd}
\pi_{*,*}^{\R}(i_*H\F_2\otimes BPGL)[\rho^{-1}]\arrow[r]\arrow[d,"\varphi"]&\pi_{*,*}^{\R}(i_*H\F_2\otimes M\F_2)[\rho^{-1}]\arrow[d]\\
\pi_{*,*}^{\R}(M\F_2\otimes BPGL)^{\complete}_\rho[\tau^{-1},\rho^{-1}]\arrow[r]&\pi_{*,*}^{\R}(M\F_2\otimes M\F_2)^{\complete}_\rho[\tau^{-1},\rho^{-1}]
\end{tikzcd}
\]
and the righthand and bottom maps are injective. We have seen the top map is induced by the map
\[i_*H\F_2\otimes i_*H\F_2\xrightarrow{1\otimes 1\otimes i_*(\varphi_2)}i_*H\F_2\otimes i_*H\F_2[y]\]
$\varphi_2$ differs from the unit map $H\F_2\xrightarrow{\eta}H\F_2[y]$ by an automorphism of $H\F_2[y]$ and hence is the inclusion of a retract.

For (2), since $\varphi$ is injective, the map 
\[\F_2[z,\chi_i|i\ge 1]^{\complete}_{z}\oplus\mathcal A_*\to \F_2[z,\chi_i|i\ge 1]^{\complete}_{z}[z^{-1}]\]
fails to be injective in degree $j\neq 0$ if and only if there exists a class $0\neq x\in\mathcal A_j$ such that 
\[\varphi(x)\in\mathrm{image}\bigg(\F_2[z,\chi_i|i\ge 1]^{\complete}_{z}\hookrightarrow \F_2[z,\chi_i|i\ge 1]^{\complete}_{z}[z^{-1}]\bigg)\]
Note, however, that since $\varphi(\xi_i)=z^{-1}(\chi_iz^{2^i}+\chi_{i-1})$, a monomial $\xi_{i_1}^{j_1}\cdots\xi_{i_n}^{j_n}$ is sent to
\aln{
\varphi(\xi_{i_1}^{j_1}\cdots\xi_{i_n}^{j_n})&=z^{-j_1}(\chi_{i_1}z^{2^{i_1}}+\chi_{i_1-1})^{j_1}\cdots z^{-j_n}(\chi_{i_n}z^{2^{i_n}}+\chi_{i_n-1})^{j_n}\\
&=z^{-(j_1+\cdots +j_n)}\chi_{i_1-1}^{j_1}\cdots\chi_{i_n-1}^{j_n}+\cdots
}
where the omitted terms are of the form $z^mr(\chi_i)$ with $m>-(j_1+\cdots +j_n)$, where $r(\chi_i)$ is some polynomial in the $\chi_i$'s. In particular, for a polynomial 
\[p(\xi_i)=\sum\limits_{k=1}^N\xi_{i_{k,1}}^{j_{k,1}}\cdots\xi_{i_{k,n_k}}^{j_{k,n_k}}\]
in $\mathcal A_*$, let 
\[M:=\max\{j_{k,1}+\cdots+j_{k,n_k}\text{ : } 1\le k\le N\}\] 
\[q(\xi_i):=\sum\limits_{k\text{ : } j_{k,1}+\cdots+j_{k,n_k}=M}\xi_{i_{k,1}}^{j_{k,1}}\cdots\xi_{i_{k,n_k}}^{j_{k,n_k}}\]
so that $p(\xi_i)=q(\xi_i)+r(\xi_i)$ where 
\[r(\xi_i):=\sum\limits_{k\text{ : } j_{k,1}+\cdots+j_{k,n_k}<M}\xi_{i_{k,1}}^{j_{k,1}}\cdots\xi_{i_{k,n_k}}^{j_{k,n_k}}\]
Then 
\[\varphi(p(\xi_i))=z^{-M}q(\chi_{i-1})+\cdots\]
where the omitted terms are of the form $z^ms(\chi_i)$ with $m>-M$,  where $s(\chi_i)$ is some polynomial in the $\chi_i$'s. Note that $q(\chi_{i-1})\neq0$, so this sum cannot be in the image of the bottom map in the arithmetic square, as $M>0$. (3) and (4) are immediate from the Mayer--Vietoris sequence.
\end{proof}

This is already enough to describe $H_*\Gamma(BPGL)$ as an $\mathcal{A}_*$-comodule. For the following discussion, cf. \cite{BPO}. Let $K(\rho)$ be defined by the cofiber sequence
\[K(\rho)\to S^{0,0}\to S^{0,0}[\rho^{-1}]\]
We have a long exact sequence of left $\mathcal A_*$-comodules
\[\cdots\xrightarrow{\partial}H_*\Gamma(BPGL\otimes K(\rho))\to H_*\Gamma(BPGL)\to\mathcal A_*\xrightarrow{\partial}H_{*-1}\Gamma(BPGL\otimes K(\rho))\to\cdots\]
Proposition \ref{propbndmap} gives a splitting of left $\mathcal A_*$-comodules
\[H_*\Gamma(BPGL)\cong \F_2\{1\}\oplus\coker(\partial)\]
To describe $\coker(\partial)$, since our arithmetic square is a pullback in weight zero by Proposition \ref{proparithmeticsquare} and Corollary \ref{cor:betaE}, we have a diagram
\[
\begin{tikzcd}
\mathcal A_*\arrow[r,"\partial"]\arrow[d,"\varphi"]&H_{*-1}\Gamma(BPGL\otimes K(\rho))\arrow[d,"="]\\
\F_2[z,\chi_i]^{\complete}_{z}[z^{-1}]\arrow[r]&H_{*-1}\Gamma(BPGL\otimes K(\rho))
\end{tikzcd}
\]
Since $\F_2[z,\chi_i]^{\complete}_{z}\to \F_2[z,\chi_i]^{\complete}_{z}[z^{-1}]$ is an injection, we find that
\aln{
\Sigma H_*\Gamma(BPGL\otimes K(\rho))&=\coker(\F_2[z,\chi_i]^{\complete}_{z}\to \F_2[z,\chi_i]^{\complete}_{z}[z^{-1}])\\
&=\F_2[z^{-1},\chi_i]\{z^{-1}\}
}
Together with our explicit description of $\varphi$, we conclude:
\begin{theorem}\label{thmhomologybpr}
$H_*\Gamma(BPGL)$ splits as a left $\mathcal A_*$-comodule
\[H_*\Gamma(BPGL)=\F_2\{1\}\oplus\mathrm{coker}(\partial)\]
where $\partial$ may be described as the map
\aln{
\partial:\mathcal A_*&\to \Sigma^{-1}\F_2[z^{-1},\chi_i]\{z^{-1}\}\\
\xi_i&\mapsto \chi_iz^{2^i-1}+\zeta_{i-1}z^{-1}
}
which is to be understood as multiplicative on the generators $\xi_i$, while in the result, all nonnegative powers of $y$ are set to 0. The map $\partial$ is injective in positive degrees. 
\end{theorem}

\begin{remark}
    It is possible to describe the $\mathcal A_*$-comodule structure on $\coker(\partial)$ with explicit formulas. We omit these formulas, however, as we believe they are too complicated to be useful. By contrast, as we will see, the HSSS of $BPGL$ describes $H_*\Gamma(BPGL)$ as a subquotient of an explicit and straightforward $\mathcal A_*$-comodule algebra.

    We remark that it is also possible to describe $H_*\Gamma(BPGL\langle m\rangle)$ as an $\mathcal{A}_*$-comodule somewhat explicitly in terms of the boundary map in the isotropy separation sequence, as in Theorem \ref{thmhomologybpr}. Due to the complicated nature of these formulas, however, we again prefer to use the HSSS to describe the global structure of homology.
\end{remark}

\subsection{The image of the edge homomorphism for $BPGL\langle m\rangle$}\label{subsecimageedgehom}
We will see in Section \ref{seccomputations} that in the HSSS for $BPGL\langle m\rangle$, the map $E_2\to E_2[\rho^{-1}]$ to the $\rho$-localized spectral sequence is an injection on $E_2$, and this allows us to determine the image of the edge homomorphism from that of the localized spectral sequence, using our results on the arithmetic square.

The edge homomorphism is natural, giving a commutative diagram
\[
\begin{tikzcd}
    \pi_{*,*}^{\R}(i_*H\F_2\otimes BPGL\langle m\rangle)\arrow[d]\arrow[r]&\pi_{*,*}^{\R}(i_*H\F_2\otimes BPGL\langle m\rangle)[\rho^{-1}]\arrow[d]\\
    \pi_{*,*}^{\R}(i_*H\F_2\otimes M\Z)\arrow[r]&\pi_{*,*}^{\R}(i_*H\F_2\otimes M\Z)[\rho^{-1}]
\end{tikzcd}
\]
In the lefthand side of this square, we have an isomorphism in weight zero
\[\mathcal A_*\square_{\mathcal A(0)_*}\F_2\cong H_*H\Z\cong \pi_{*,0}^{\R}(i_*H\F_2\otimes M\Z)\]
coming from the canonical map $i_*H\Z\to M\Z$. On the righthand side, using the equivalence 
\[i_*H\F_2[y^{2^{m+1}}][\rho^{-1}]\simeq BPGL\langle m\rangle[\rho^{-1}]\] 
of Lemma \ref{lemmarholocalbpgl}, the righthand map becomes the canonical inclusion
\[\mathcal A_*[y^{2^{m+1}},\rho^{\pm}]\hookrightarrow\mathcal A_*[y^{2},\rho^{\pm}]\]

Describing the bottom map in terms of these identifications therefore places us in the context of Theorem \ref{thmtatefrob}, which determines two identifications
\[H_*\Gamma(M\F_2[\rho^{-1}])\cong\mathcal A_*[y]\]
To be explicit, the canonical map $i_*H\F_2[\rho^{-1}]\to M\F_2[\rho^{-1}]$ gives an equivalence $H\F_2[y]\simeq \Gamma(M\F_2[\rho^{-1}])$, which determines an isomorphism of the above form denoted $\cong_{\mathrm{can}}$. On the other hand, the map $i_*H\F_2[\rho^{-1}]\simeq BPGL[\rho^{-1}]\to M\F_2[\rho^{-1}]$ determines an isomorphism of the same form, which we denote $\cong_{BP}$. Theorem \ref{thmtatefrob} and the proof of Proposition \ref{proprightmap} tell us how to translate between these identifications.

\begin{proposition}\label{prophomologytwist}
    The composition
    \[\mathcal A_*[y]\cong_{\mathrm{can}} H_*\Gamma(M\F_2[\rho^{-1}])\cong_{BP}\mathcal A_*[y]\]
    sends
    \aln{
    y&\mapsto y\\
    \zeta_i&\mapsto \zeta_i+\zeta_{i-1}y^{2^{i-1}}
    }
\end{proposition}
\begin{proof}
Theorem \ref{thmtatefrob} and Lemma \ref{lemmacapproduct} describe the inverse of this composite as a sum of cap products. Using the congruence
\[\zeta_m\equiv \xi_1^{2^m-1}\mod (\xi_2,\xi_3,\ldots)\]
one has that the inverse of this composite sends 
\[\zeta_m\mapsto \sum\limits_{i=0}^m\zeta_iy^{2^m-2^i}\]
which implies the claimed description.
\end{proof}

\begin{corollary}\label{coredgeA(m)}
The image of the edge homomorphism in the HSSS of $BPGL\langle m\rangle$ contains $\mathcal A_*\square_{\mathcal A(m)_*}\F_2\subset H_*H\Z$.
\end{corollary}
\begin{proof}
We claim that the image of each of the generators
\[\zeta_1^{2^{m+1}},\zeta_2^{2^m},\ldots,\zeta_{m+1}^2,\zeta_{m+2},\ldots,\]
of $\mathcal A_*\square_{\mathcal A(m)_*}\F_2$ along the map $E_2\to E_2[\rho^{-1}]$ is in the image of the edge homomorphism for the $\rho$-localized \sseq. Indeed, Proposition \ref{prophomologytwist} implies that the localization map sends
\aln{
\zeta_{i}^{2^{m+2-i}}\mapsto \zeta_i^{2^{m+2-i}}+\zeta_{i-1}^{2^{m+2-i}}y^{2^{m+1}}\hspace{0.5cm}&\mathrm{ for}\hspace{0.1cm}i\le m+2\\
\zeta_i\mapsto \zeta_i+\zeta_{i-1}y^{2^{i-1}}\hspace{0.5cm}&\mathrm{ for}\hspace{0.1cm}i>m+2
}
so that each of these generators lands in the subalgebra $\mathcal{A}_*[y^{2^{m+1}}]\subset E_2[\rho^{-1}]$.

We show in Corollary \ref{corE_2description} that $E_2\to E_2[\rho^{-1}]$ is an injection, and it therefore suffices to show that each of the classes
\aln{
\zeta_i^{2^{m+2-i}}+\zeta_{i-1}^{2^{m+2-i}}y^{2^{m+1}}\hspace{0.5cm}&\mathrm{ for}\hspace{0.1cm}i\le m+2\\
\zeta_i+\zeta_{i-1}y^{2^{i-1}}\hspace{0.5cm}&\mathrm{ for }\hspace{0.1cm}i>m+2
}
admits a lift along the map
\[\pi_{*,*}^{\R}(i_*H\F_2\otimes BPGL\langle m\rangle)\to\pi_{*,*}^{\R}(i_*H\F_2\otimes BPGL\langle m\rangle)[\rho^{-1}]\]
The Mayer--Vietoris sequence for the arithmetic square of $i_*H\F_2\otimes BPGL\langle m\rangle$ implies that it suffices to show that the image of each of these classes along the map 
\[\varphi_m:\pi_{*,*}^{\R}(i_*H\F_2\otimes BPGL\langle m\rangle)[\rho^{-1}]\to\pi_{*,*}^{\R}(i_*H\F_2\otimes BPGL\langle m\rangle)^{\complete}_\rho[\tau^{-1},\rho^{-1}]\]
is in the image of the map
\[\pi_{*,*}^{\R}(i_*H\F_2\otimes BPGL\langle m\rangle)^{\complete}_\rho[\tau^{-1}]\to\pi_{*,*}^{\R}(i_*H\F_2\otimes BPGL\langle m\rangle)^{\complete}_\rho[\tau^{-1},\rho^{-1}]\]
The map $\varphi_m$ is described by Proposition \ref{proprightmap} along with the fact that it sends $y^{2^{m+1}}$ to $(\eta_R(\tau)/\rho)^{2^{m+1}}$. 

Embedding into the case $m=-1$, the map $\varphi$ sends these classes back along the inverse of the composition in Proposition \ref{prophomologytwist}, i.e. to the image of each of the generators above along the map
\[i_*:\mathcal{A}_*\to \pi_{*,*}^{\R}(M \F_2\otimes M\F_2)^{\complete}_\rho[\tau^{-1},\rho^{-1}]\]
The image of the map
\[\pi_{*,*}^{\R}(i_*H\F_2\otimes BPGL\langle m\rangle)^{\complete}_\rho[\tau^{-1}]\to \pi_{*,*}^{\R}(M \F_2\otimes M\F_2)^{\complete}_\rho[\tau^{-1},\rho^{-1}]\]
is a sub-$\mathcal{A}_*$-comodule, so it suffices to show that its image contains 
\[i_*(\zeta_{m+2}),i_*(\zeta_{m+3}),\ldots\]
Conjugating the relations \cite[Theorem 2.12]{LSWX}, one has
\[i_*(\zeta_i)=\frac{1}{\tau}\big(c(\tau_{i-1})\eta_R(\tau)^{2^{i-1}}+\rho^{2^i}c(\tau_i)\big)\]
The result now follows from Proposition \ref{prop:motivichomologybprm}.
\end{proof}

Arguing similarly, we find the following classes in the image of the edge homomorphism in weight $-1$.
\begin{corollary}\label{corx_mpermanent}
For all $m\ge0$, the class $\rho\xi_m$ admits a lift along the map
\[\pi_{*,*}^{\R}(i_*H\F_2\otimes BPGL)\to\pi_{*,*}^{\R}(i_*H\F_2\otimes BPGL)[\rho^{-1}]\]
Therefore any lift of $\rho\xi_m\in E_2[\rho^{-1}]$ along $E_2\to E_2[\rho^{-1}]$ is in the image of the edge homomorphism.
\end{corollary}
\begin{proof}
By Proposition \ref{proprightmap}, we have
\[\varphi(\rho\xi_m)=\rho^{2^m}\overline{t}_m+\tau^{2^{m-1}}\overline{t}_{m-1}\]
\end{proof}

\section{The HSSS for $BPGL\langle m\rangle$}\label{secdifferentials}
%

\subsection{The slice \texorpdfstring{\(E_2\)}{E Two} page}\label{subsecE_2}

By Proposition \ref{propslicebp}, the HSSS of $BPGL\langle m\rangle$ has $E_2$ page
\[\pi_{*,*}^{\R}(i_*H\F_2\otimes M\Z[\overline{v}_{1},\ldots,\overline{v}_{m}])=\pi_{*,*}^{\R}(i_*H\F_2\otimes M\Z)[\overline{v}_{1},\ldots,\overline{v}_{m}]\]
We therefore proceed to calculate the bigraded homotopy ring
\[\pi_{*,*}^{\R}(i_*H\F_2\otimes M\Z)\]
This has some unexpected multiplicative features; we first determine the additive structure.

\begin{proposition}\label{propE_2splitting}
The map $M\Z\to i_*H\mathbb{F}_2\otimes M\Z$ determines an isomorphism of left $\mathcal{A}_*$-comodules
\[\pi_{*,*}^{\R}(i_*H\F_2\otimes M\Z)\cong(\mathcal A_*\square_{\mathcal A(0)_*}\F_2)[\tau^2]\oplus\bigoplus\limits_{i\ge0,j>0}\mathcal A_*\{\tau^{2i}\rho^j\}\]
\end{proposition}
\begin{proof}
Voevodsky's computation \cite{voe} of $\pi_{*,*}^{\R}M\Z$ (see also \cite{GHIR}) implies that the $H\Z$-module
\[\Gamma\bigg(\bigoplus\limits_{b\in\Z}\Sigma^{0,b}M\Z\bigg)\]
splits as 
\[\bigoplus\limits_{i\ge0}H\Z\{\tau^{2i}\}\oplus\bigoplus\limits_{i\ge0,j>0}H\F_2\{\tau^{2i}\rho^j\}\]
and the result follows upon taking homology.
\end{proof}
There are subtleties in computing the products in this ring, and these arise from the following class.

\begin{definition}\label{defx_1}
Let $x_1$ denote the class $[\rho\zeta_1]$ in bidegree $(0,-1)$. We use brackets to emphasize that $x_1$ is indecomposable in 
\[\pi_{*,*}^{\R}(i_*H\F_2\otimes M\Z)\]
In particular, there is no class $\zeta_1$, as $\zeta_1\notin\mathcal A_*\square_{\mathcal A(0)_*}\F_2$, and $x_1$ is not divisible by $\rho$.
\end{definition}

Most of the ring structure in $\pi_{*,*}^{\R}(i_*H\mathbb{F}_2\otimes M\Z)$ is determined by the ring maps
\[\pi_*(H\F_2\otimes H\Z)\to\pi_{*,*}^{\R}(i_*H\F_2\otimes M\Z)\]
and 
\[\pi_{*,*}^{\R}(M\Z)\to\pi_{*,*}^{\R}(i_*H\F_2\otimes M\Z)\]
However, the class $x_1$ is not in the subalgebra generated by the images of these maps. To determine products involving $x_1$, we instead work $\rho$-locally.

\begin{proposition}\label{propE_2localized}
There is an isomorphism of bigraded rings
\[\mathcal A_*[y_2,\rho^{\pm}]\to \pi_{*,*}^{\R}(i_*H\F_2\otimes M\Z)[\rho^{-1}]\]
where $|y_2|=2$, and the localization map
\[\phi:\pi_{*,*}^{\R}(i_*H\F_2\otimes M\Z)\to \pi_{*,*}^{\R}(i_*H\F_2\otimes M\Z)[\rho^{-1}]\]
sends
\begin{align*}
\tau^2&\mapsto \rho^2y_2
&x_1&\mapsto \zeta_1\rho
&\zeta_1^2&\mapsto\zeta_1^2+y_2
&\zeta_i&\mapsto\zeta_i+\zeta_{i-1}y_2^{2^{i-2}}
\end{align*}
where $\zeta_1^2,\zeta_i$ are regarded as the elements in the 
\[(\mathcal A_*\square_{\mathcal A(0)_*}\F_2)[\tau^2]\]
summand described in the previous proposition.
\end{proposition}
\begin{proof}
We define the map
\[\mathcal A_*[y_2,\rho^{\pm}]\to \pi_{*,*}^{\R}(i_*H\F_2\otimes M\Z)[\rho^{-1}]\]
via the map
\[\mathcal A_*[\rho^{\pm}]\cong\pi_{*,*}^{\R}(i_*H\F_2\otimes BPGL)[\rho^{-1}]\to\pi_{*,*}^{\R}(i_*H\F_2\otimes M\Z)[\rho^{-1}]\]
by sending $y_2$ to the image of $\tau^2/\rho^2$ along the right unit map
\[\pi_{*,*}^{\R}(M\Z)[\rho^{-1}]\to \pi_{*,*}^{\R}(i_*H\F_2\otimes M\Z)[\rho^{-1}]\]
This is an isomorphism as $M\Z[\rho^{-1}]\simeq i_*H\F_2[y_2][\rho^{-1}]$.

For the claims about the map $\phi$, note that Proposition \ref{propE_2splitting} implies that 
\[\pi_{*,*}^{\R}(i_*H\F_2\otimes M\Z)\]
is $\rho$-torsion free, so $\phi$ is an injection. The map
\[\mathcal A_*\square_{\mathcal A(0)_*}\F_2\cong\pi_{*,0}^{\R}(i_*H\F_2\otimes M\Z)\to \pi_{*,0}^{\R}(i_*H\F_2\otimes M\Z)[\rho^{-1}]\cong \mathcal A_*[y_2]\]
is described by Proposition \ref{prophomologytwist} since the left and right presentations come from the canonical and $BPGL$ $H\F_2$-algebra structures, respectively. Now $x_1\mapsto \rho\xi_1$ since $\rho\xi_1$ is the only nonzero class in this bidegree. 
\end{proof}

\begin{corollary}\label{corx_1rel}
In the ring $\pi_{*,*}^{\R}(i_*H\F_2\otimes M\Z)$, we have the relation
\[x_1^2=\zeta_1^2\rho^2+\tau^2\]
\end{corollary}
\begin{proof}
The map $\phi$ of the previous proposition is an injective ring map, and one has
\[\phi(\zeta_1^2\rho^2+\tau^2)=\zeta_1^2\rho^2+y_2\rho^2+y_2\rho^2=\phi(x_1^2)\qedhere\]
\end{proof}

\begin{corollary}\label{corE_2description}
The map
\[(\mathcal A_*\square_{\mathcal A(0)_*}\F_2)[\rho,x_1]\to\pi_{*,*}^{\R}(i_*H\F_2\otimes M\Z)\]
is an isomorphism of left $\mathcal A_*$-comodule algebras, where $\rho$ is primitive, and
\[\psi(x_1)=1\otimes x_1+\xi_1\otimes \rho\]
\end{corollary}
\begin{proof}
The source and target of the map are bigraded $\F_2$-vector spaces of the same (finite) graded dimension. It therefore suffices to show the map is surjective, and this follows from Corollary \ref{corx_1rel}, which implies that $\tau^2$ is in the image.
\end{proof}

\subsection{Differentials on subalgebras of the Steenrod algebra}\label{subsecdifferentials}
Corollary \ref{corE_2description} gives a description of our $E_2$-page \[E_2^{*,*,*}(BPGL\langle m\rangle;i_*H\mathbb F_2)\cong(\mathcal A_*\square_{\mathcal A(0)_*}\F_2)[\rho,x_1,\overline{v}_{1},\ldots,\overline{v}_{m}]\]
The classes $\rho,x_1,\overline{v}_{1},\ldots,\overline{v}_{m}$ are all permanent cycles, and the subalgebra
\[\mathcal A_*\square_{\mathcal A(m)_*}\F_2\subset \mathcal A_*\square_{\mathcal A(0)_*}\F_2\]
consists of permanent cycles by Corollary \ref{coredgeA(m)}. In this section, we describe a family of differentials in this spectral sequence on $\mathcal A_*\square_{\mathcal A(0)_*}\F_2$ that interpolate between $\mathcal A_*\square_{\mathcal A(0)_*}\F_2$ and $\mathcal A_*\square_{\mathcal A(m)_*}\F_2$.

For all $i$, $\mathcal A_*\square_{\mathcal A(i-1)_*}\F_2$ is a free $\mathcal A_*\square_{\mathcal A(i)_*}\F_2$-module given by
\[\mathcal A_*\square_{\mathcal A(i-1)_*}\F_2\cong (\mathcal A_*\square_{\mathcal A(i)_*}\F_2)\otimes E(\zeta_1^{2^{i}},\zeta_2^{2^{i-1}},\ldots,\zeta_{i+1})\]
We show that, for $i\le m$ each of these generators $\zeta_{j}^{2^{i+1-j}}$ supports a $d_{2^{i+1}-1}$ (Theorem \ref{differentialsthm}). Each of the squares of these generators is then a cycle and we are left with $\mathcal A_*\square_{\mathcal A(i)_*}\F_2$ on $E_{2^{i+1}}$ along with various Massey products (see Definition \ref{defx_m}) due to the fact that the values of $d_{2^{i+1}-1}$ on these generators are not algebraically independent. We deduce from these differentials that the image of the edge homomorphism is \textit{precisely} $\mathcal A_*\square_{\mathcal A(m)_*}\F_2$ (Theorem \ref{edgethm}).

To determine these differentials, we use an inductive argument on $m$ and Corollary \ref{coredgeA(m)} to deduce that the differential $d_{2^{m+1}-1}$ must take classes in $\mathcal A_*\square_{\mathcal A(m-1)_*}\F_2$ to a $\rho$-torsion free part of $E_{2^{m+1}-1}$. The differential is then determined by the structure of the $\rho$-localized spectral sequence, which we determine completely (Proposition \ref{proplocalizedSS}).

In this section, we use the following notation for our HSSS's.
\aln{
E_r^{*,*,*}\langle m\rangle&:=E_r^{*,*,*}(BPGL\langle m\rangle;i_*H\F_2)\\
E_r^{*,*,*}\langle m\rangle[\rho^{-1}]&:=E_r^{*,*,*}(BPGL\langle m\rangle;i_*H\F_2[\rho^{-1}])
}

\subsubsection{The localized spectral sequence}\label{subsubseclocalized} 
We begin by determining the structure of the localized spectral sequence
\[E_r[\rho^{-1}]=E_r^{*,*,*}(BPGL;i_*H\F_2[\rho^{-1}])\]
The spectral sequence $E_r\langle m\rangle[\rho^{-1}]$ for $m$ finite
can be read off from this by setting $f_i=0$ for $i>m$.

\begin{proposition}\label{proplocalizedSS}
In the spectral sequence $E_r[\rho^{-1}]$, we have 
\begin{enumerate}
\item $E_2[\rho^{-1}]\cong\mathcal A_*[y_2,f_1,f_2,\ldots][\rho^{\pm}]$,
where

\begin{itemize}
\item $f_m=\rho^{2^m-1}\overline{v}_{m}$ so that $|f_m|=(2^m-1,0,2^m-1)$
\item $\mathcal A_*$ is the image of the edge homomorphism
\[\mathcal A_*\cong\pi_{*,*}^{\R}(i_*H\F_2\otimes BPGL)[\rho^{-1}]\to\pi_{*,*}^{\R}(i_*H\F_2\otimes M\Z)[\rho^{-1}]=\mathcal A_*[y_2]\]
and $y_2$ is the image of the class $\tau^2/\rho^2$ along the map
\[\pi_{*,*}^{\R}(M\Z)[\rho^{-1}]\to \pi_{*,*}^{\R}(i_*H\F_2\otimes M\Z)[\rho^{-1}]\]
\item If $q\in\mathcal A_i$, the tridegree of $q$ is $(i,0,-i)$, and $|y_2|=(2,0,-2)$\\

\end{itemize}
\item The differential $d_k[\rho^{-1}]=0$ unless $k=2^{m+1}-1$ for $m\ge 1$, and
\[E_{2^{m+1}-1}[\rho^{-1}]=\mathcal A_*[y_2^{2^{m-1}},f_m,f_{m+1},\ldots][\rho^{\pm}]\]
\item The differential $d_{2^{m+1}-1}[\rho^{-1}]$ is determined by
\[d_{2^{m+1}-1}[\rho^{-1}](y_2^{2^{m-1}})=f_m\]
\end{enumerate}
\end{proposition}
\begin{proof}
The description of $E_2[\rho^{-1}]$ follows from Proposition \ref{propE_2localized}. The image of the edge homomorphism consists of permanent cycles, so the only algebra generator that is not a permanent cycle is $y_2$, and since the spectral sequence must converge to
\[\pi_{*,*}^{\R}(i_*H\F_2\otimes BPGL)[\rho^{-1}]\cong\mathcal A_*[\rho^{\pm}]\]
each $f_i$ must be hit by a differential. For degree reasons, the claimed pattern of differentials is the only possibility.
\end{proof}

Proposition \ref{propE_2localized} also gives us a description of the map $E_2\to E_2[\rho^{-1}]$, and there are important permanent cycles in $E_2^{*,*,*}(BPGL;i_*H\F_2)$ that lift the classes $\rho\xi_m$ along this map.

\begin{definition}\label{defx_m}
Define $x_0=\rho$, $x_1$ as in Definition \ref{defx_1}, and inductively
\[x_m=\sum\limits_{i=0}^{m-1}x_i\zeta_{m-i}^{2^i}\]
for $m>1$. The tridegree of $x_m$ is $(2^m-2,-1,2-2^m)$ in $E_r^{*,*,*}(BPGL;i_*H\F_2)$.
\end{definition}

\begin{proposition}\label{propmaptolocalized}
The map $E_2\to E_2[\rho^{-1}]$ sends
\begin{align*}
\zeta_1^2&\mapsto \zeta_1^2+y_2
&\zeta_i &\mapsto \zeta_i+\zeta_{i-1}y_2^{2^{i-2}}
&x_m&\mapsto \rho\xi_m
\end{align*}
\end{proposition}
\begin{proof}
The first two follow immediately from Proposition \ref{propE_2localized}, and the claim
\[x_m\mapsto \rho\xi_m\]
follows by induction on $m$.
\end{proof}

The fact that $E_2\to E_2[\rho^{-1}]$ is an injection also gives us the coactions on these classes.

\begin{corollary}\label{corcoactionx_m}
On $E_2$, the classes $x_m$ satisfy the following coaction formula
\aln{
\psi(x_0)&=1\otimes x_0\\
\psi(x_m)&=\sum\limits_{i=0}^m\xi_i^{2^{m-i}}\otimes x_{m-i}
}
\end{corollary}

\subsubsection{Differentials on \texorpdfstring{\((\cA(m)/\!/\cA(m-1))^\ast\)}{A of m mod A of m-1}}

The cofiber sequences
\[\Sigma^{2(2^m-1),2^m-1}BPGL\langle m\rangle\xrightarrow{\cdot\overline{v}_{m}} BPGL\langle m\rangle\to BPGL\langle m-1\rangle\]
allow us to set up an inductive argument to determine a family of differentials in $E_r^{*,*,*}\langle m\rangle$. We first state a standard lemma on morphisms of spectral sequences.

\begin{lemma}\label{filteredlemma}
Suppose $E_r^{s,t}(1)\to E_r^{s,t}(2)$ is a morphism of spectral sequences with the property that the map is an isomorphism for $t<N$ when $r=2$. Then
\begin{enumerate}
\item The map $E_r^{s,t}(1)\to E_r^{s,t}(2)$ is an isomorphism whenever $t+r<N+2$.
\item The map $E_r^{s,t}(1)\to E_r^{s,t}(2)$ is an injection whenever $t<N$.
\end{enumerate}
\end{lemma}

In our case, we have the following.

\begin{proposition}\label{prop:EquivOfSlices}
For \(t< 2^{m+1}-2\), the map
\[
P^tBPGL\langle m\rangle\to P^tBPGL\langle m-1\rangle
\]
is an equivalence.
\end{proposition}
\begin{proof}
We have a cofiber sequence of \(BPGL\)-modules
\[
\Sigma^{2(2^m-1),2^m-1}BPGL\langle m\rangle\xrightarrow{\bar{v}_{m}}BPGL\langle m\rangle\to BPGL\langle m-1\rangle
\]
The source of \(\bar{v}_{m}\)-multiplication is slice \(2(2^{m}-1)\)-connective, and the result follows.
\end{proof}




We begin by showing that each of the classes $\zeta_j^{2^{i+1-j}}$ survives to $E_{2^{i+1}-1}$.

\begin{proposition}\label{dlower}
For $m\ge 1$, in $E_r^{*,*,*}\langle m\rangle$,
\[d_r(\zeta_j^{2^{i+1-j}})=0\]
for all $1\le i\le m$, $1\le j\le i+1$, and $r<2^{i+1}-1$.
\end{proposition}
\begin{proof}
We proceed by induction on $m$. When $m=1$, this follows from the fact that $d_2=0$; in fact $d_{2k}=0$ for all $k$ and $m$ as the odd slices of $BPGL\langle m\rangle$ vanish. By Proposition  \ref{prop:EquivOfSlices} and Lemma \ref{filteredlemma}, the map of spectral sequences
\[E_r^{s,0,t}\langle m\rangle\to E_r^{s,0,t}\langle m-1\rangle\]
is an injection whenever $t<2^{m+1}-2$. The classes $\zeta_j^{2^{i+1-j}}$ have $t=0$, and thus the $i<m$ case of the proposition follows by the inductive hypothesis. The $i=m$ case then follows in the same way from the fact that the classes $\zeta_j^{2^{m+1-j}}$ are permanent cycles in $E_r^{*,*,*}\langle m-1\rangle$ by Corollary \ref{coredgeA(m)}.  
\end{proof}
We know, therefore, that the first possible nonzero differential on $\zeta_j^{2^{i+1-j}}$ is a $d_{2^{i+1}-1}$. We show that this differential is nonzero and can be read off from the corresponding differential in the localized spectral sequence $E_r^{*,*,*}\langle m\rangle[\rho^{-1}]$. For this, we need to know that the differential lands in an $\rho$-torsion free part of the spectral sequence, for which we need the following lemma.

%

\begin{lemma}\label{subspacelemma}
In $E_k^{*,*,*}\langle m\rangle$, for $k<2^{m+1}-1$, there are no differentials of the form
\[d_k(x)=\ov_m p(\rho,x_1)q(\zeta_i)+r(\overline{v}_{i},\rho,x_1)s(\zeta_i)\]
where 
\aln{
q(\zeta_i),s(\zeta_i)&\in\mathcal A_*\square_{\mathcal A(0)_*}\F_2\\
p(\rho,x_1)&\in\F_2[\rho,x_1]\\
r(\overline{v}_{i},\rho,x_1)&\in\F_2[\rho,x_1,\overline{v}_{1},\ldots,\overline{v}_{{m-1}}]
}
such that $pq$ is nonzero.
\end{lemma}
\begin{proof}
Suppose to the contrary that, for some $k<2^{m+1}-1$, there is some nonzero differential $d_k(x)=y$ for $y$ of the above form, and let $k$ be minimal with respect to this property. It follows that $y$ is $\rho$-torsion free; since $y$ is $\rho$-torsion free on $E_2$, if we had $\rho^Ny=0\in E_k$, we must have a nonzero differential $d_{k'}(x')=\rho^Ny$ for some $k'<k$, contradicting minimality of $k$.

Since $y\in E_k$ is $\rho$-torsion free, its image in the localized spectral sequence $E_k^{*,*,*}\langle m\rangle[\rho^{-1}]$ is nonzero. By naturality, this determines a nonzero differential $d_k$ in the localized spectral sequence. This contradicts the description of the differentials given in Proposition \ref{proplocalizedSS}.
\end{proof}

\begin{remark}
In particular, letting $rs=0$, it follows that the subspace
\[B_m:=\overline{v}_{m}\cdot\mathrm{Sym}^{2^m-1}(\F_2\{\rho,x_1\})\cdot (\mathcal A_*\square_{\mathcal A(0)_*}\F_2)\]
of $E_2$ receives no differentials $d_r$ for $r<2^{m+1}-1$. In particular, the subspace of $B_m$ of $d_1,\ldots,d_r$ cycles for $r<2^{m+1}-1$ is a $\rho$-torsion free subspace of $E_{2^{m+1}-1}$, since $B_m$ is $\rho$-torsion free on $E_2$.
\end{remark}

\begin{proposition}\label{v_npure}
In $E_{2^{m+1}-1}\langle m\rangle$, the composition
\[\mathcal A_*\square_{\mathcal A(m-1)_*}\F_2\hookrightarrow E_{2^{m+1}-1}\xrightarrow{d_{2^{m+1}-1}}E_{2^{m+1}-1}\]
lands in the subspace of
\[\overline{v}_{m}\cdot\mathrm{Sym}^{2^m-1}(\F_2\{\rho,x_1\})\cdot (\mathcal A_*\square_{\mathcal A(0)_*}\F_2)\]
that survives to $E_{2^{m+1}-1}$.
\end{proposition}
\begin{proof}
The class $z:=d_{2^{m+1}-1}(\zeta_i^{2^{m+1-i}})$ may be written as a sum of monomials of the form
\[y:=\rho^ax_1^bm(\overline{v}_{i})q(\zeta_i)\]
written as an element in a subquotient of $E_2$, where $m(\overline{v}_{i})$ is a monomial in $\F_2[\overline{v}_{1},\ldots,\overline{v}_{m}]$ and $q(\zeta_i)$ is a monomial in $\mathcal A_*\square_{\mathcal A(0)_*}\F_2$. Note that the class $\zeta_i^{2^{m+1-i}}$ has $t=0$, hence $y$ has $t=2^{m+1}-2$. It follows that  $\overline{v}_{m}\cdot\overline{v}_{j}$ does not divide $m(\overline{v}_{i})$ for any $j\le m$.

We therefore have that $z$ can be written as 
\[\overline{v}_{m}p(\rho,x_1)q(\zeta_i)+r(\rho,x_1,\overline{v}_{i})s(\zeta_i)\]
as in the lemma. The class $\overline{v}_{m}p(\rho,x_1)q(\zeta_i)$ lies in
\[\overline{v}_{m}\cdot\mathrm{Sym}^{2^m-1}(\F_2\{\rho,x_1\})\cdot (\mathcal A_*\square_{\mathcal A(0)_*}\F_2)\]
since $z$ is in weight zero. Note that $pq$ is nonzero by naturality and the fact that 
\[d_{2^{m+1}-1}[\rho^{-1}](y_2^{2^{m-1}})=f_m\]
by Proposition \ref{proplocalizedSS}. We claim that $r(\overline{v}_{i},\rho,x_1)s(\zeta_i)=0\in E_{2^{m+1}-1}$, completing the proof. 

If $r(\overline{v}_{i},\rho,x_1)s(\zeta_i)\neq 0\in E_{2^{m+1}-1}$, since 
\[\zeta_i^{2^{m+1-i}}\in E_r^{*,*,*}\langle m-1\rangle\] 
is a permanent cycle, we must have that $r(\overline{v}_{i},\rho,x_1)s(\zeta_i)=0$ on $E_{2^{m+1}-1}^{*,*,*}\langle m-1\rangle$. Hence there must be a differential in the latter spectral sequence of the form 
\[d_k(x)=r(\overline{v}_{i},\rho,x_1)s(\zeta_i)\]
for $k<2^{m+1}-1$. The class $x$ lies in 
\[E_k^{2^{m+1-i}-1-k,0,2^{m+1}-1-k}\langle m-1\rangle\]
and by Lemma \ref{filteredlemma} (1), the map
\[E_k^{2^{m+1-i}-1-k,0,2^{m+1}-1-k}\langle m\rangle\to E_k^{2^{m+1-i}-1-k,0,2^{m+1}-1-k}\langle m-1\rangle\]
is an isomorphism. Therefore, the differential $d_k(x)$ lifts to a differential in $E_k^{*,*,*}\langle m\rangle$ of the form
\[d_k(\tilde{x})=\overline{v}_{m}p'(\rho,x_1)q'(\zeta_i)+r'(\rho,x_1,\overline{v}_{i})s'(\zeta_i)\]
But we must have $p'q'=0$ by Lemma \ref{subspacelemma}, and $E_k^{*,*,*}\langle m\rangle\to E_k^{*,*,*}\langle m-1\rangle$ sends
\[r'(\rho,x_1,\overline{v}_{i})s'(\zeta_i)\mapsto r(\rho,x_1,\overline{v}_{i})s(\zeta_i)\] 
An argument similar to that of Lemma \ref{subspacelemma} shows that $r's'=rs\in E_k\langle m\rangle$.
\end{proof}

\begin{theorem}\label{differentialsthm}
In $E_{2^{i+1}-1}^{*,*,*}(BPGL\langle m\rangle;i_*H\F_2)$, we have the differentials
\[d_{2^{i+1}-1}(\zeta_j^{2^{i+1-j}})=\overline{v}_{i}\rho^{2^i-1}\bigg(p_{j-1}\bigg(\frac{x_1}{\rho},\ldots,\frac{x_{j-1}}{\rho}\bigg)\bigg)^{2^{i+1-j}}\]
for all $1\le i\le m$ and $1\le j\le i+1$, where $p_j$ is the polynomial
\[\zeta_{j}=p_j(\xi_1,\ldots,\xi_j)\]
given by the inversion formulas in the Hopf algebra $\mathcal A_*$.
\end{theorem}
\begin{proof}
We know from Lemma \ref{subspacelemma} and Proposition \ref{v_npure} that these differentials land in an $\rho$-torsion free part of the spectral sequence, so we use naturality and Proposition \ref{proplocalizedSS}. In particular, the map $E_2\to E_2[\rho^{-1}]$ sends
\[\zeta_j^{2^{i+1-j}}\mapsto \zeta_j^{2^{i+1-j}}+\zeta_{j-1}^{2^{i+1-j}}y_2^{2^{i-1}}\]
and
\[d_{2^{i+1}-1}[\rho^{-1}](\zeta_j^{2^{i+1-j}}+\zeta_{j-1}^{2^{i+1-j}}y_2^{2^{i-1}})=\zeta_{j-1}^{2^{i+1-j}}f_i\]
\end{proof}

\begin{remark}\label{capremark}
There is another way to view these differentials. The composition
\[\mathcal A_*\square_{\mathcal A(i-1)_*}\F_2\hookrightarrow E_{2^{i+1}-1}\xrightarrow{d_{2^{i+1}-1}}\Sigma E_{2^{i+1}-1}\to \Sigma E_{2^{i+1}-1}[\rho^{-1}]\]
lands in $\Sigma\mathcal A_*\cdot f_i\cong \mathcal A_{*-2^i}$, and the last map in the composition is an injection on 
\[d_{2^{i+1}-1}(\mathcal A_*\square_{\mathcal A(i-1)_*}\F_2)\]
by Lemma \ref{subspacelemma} and Proposition \ref{v_npure}. The formulas given in the theorem tell us this map coincides with the $\mathcal A_*$-comodule map
\[\mathcal A_*\square_{\mathcal A(i-1)_*}\F_2\xrightarrow{-\cap Sq^{2^i}}\mathcal A_{*-2^i}\]
Indeed, we have
\aln{
\Delta(\zeta_j^{2^{i+1-j}})&=\sum\limits_{l+k=j}\zeta_l^{2^{i+1-j}}\otimes\zeta_k^{2^{i+1-j+l}}\\
&\equiv \sum\limits_{l+k=j}\zeta_l^{2^{i+1-j}}\otimes(\xi_1^{2^k-1})^{2^{i+1-j+l}}\in\mathcal A_*\otimes\mathcal A_*/(\xi_2,\xi_3,\ldots)\\
&\equiv \sum\limits_{l+k=i}\zeta_l^{2^{i+1-j}}\otimes\xi_1^{2^{i+1}-2^{i+1-j+l}}\in\mathcal A_*\otimes\mathcal A_*/(\xi_2,\xi_3,\ldots)
}
where we have used the congruence
\[\zeta_k\equiv\xi_1^{2^k-1}\mod (\xi_2,\xi_3,\ldots)\]
The righthand tensor factor is $\xi_1^{2^i}$ if and only if $l=j-1$, so pairing on the right with $Sq^{2^i}$ gives $\zeta_{j-1}^{2^{i+1-j}}$.
\end{remark}

\begin{theorem}\label{edgethm}
The image of the edge homomorphism in weight zero in 
\[E_\infty^{*,*,*}(BPGL\langle m\rangle;i_*H\F_2)\]
is precisely $\mathcal A_*\square_{\mathcal A(m)_*}\F_2$.
\end{theorem}
\begin{proof}
We proceed by induction, where the case $m=0$ is clear. By induction, Proposition \ref{dlower}, and Lemma \ref{filteredlemma} (1), we see that a class $x\in \mathcal A_*\square_{\mathcal A(0)_*}\F_2$ is a $d_2,\ldots,d_{2^{m+1}-2}$-cycle if and only if $x\in \mathcal A_*\square_{\mathcal A(m-1)_*}\F_2$. As in Remark \ref{capremark}, the restriction of $d_{2^{m+1}-1}$ to $\mathcal A_*\square_{\mathcal A(m-1)_*}\F_2$ may be identified with the map
\[\mathcal A_*\square_{\mathcal A(m-1)_*}\F_2\xrightarrow{-\cap Sq^{2^m}}\mathcal A_{*-2^m}\]
The kernel of this map is precisely $\mathcal A_*\square_{\mathcal A(m)_*}\F_2$ because $\mathcal A_*\square_{\mathcal A(m)_*}\F_2$ is dual to the quotient $\mathcal A/\!/\mathcal A(m)$, and $\mathcal{A}(m)$ is the subalgebra of the Steenrod algebra generated by the classes 
\[\{Sq^{2^i}\st 0\le i\le m\}\qedhere\]
\end{proof}

\section{The homology of $\Gamma(BPGL\langle m\rangle)$ for $m\le 3$}\label{seccomputations}
In Section \ref{secdifferentials}, we determined the $E_2$ page of the HSSS for $BPGL\langle m\rangle$, the image of the edge homomorphism, and a family of differentials. Modulo comodule algebra extension problems, this reduces the computation of $H_*\Gamma(BPGL\langle m\rangle)$ to two problems:
\begin{enumerate}
    \item Show that the differentials of Theorem \ref{differentialsthm} are \textit{all} the differentials in the HSSS for $BPGL\langle m\rangle$. In particular, show that this spectral sequence collapses on $E_{2^{m+1}}$.
    \item Compute homology with respect to these differentials.
\end{enumerate}
We know of no way to solve problem (1) except by way of our knowledge of the image of the edge homomorphism from Section \ref{subsecimageedgehom} and sparseness arguments afforded by solving problem (2) for small values of $m$. Problem (2) is purely algebraic; however, the computation quickly becomes very complicated as $m$ grows, and it requires the use of algebraic Bockstein spectral sequences.

Nevertheless we solve both problems for $m\le 3$, completely determining the $E_\infty$ page of the HSSS as a comodule algebra. Moreover, we solve all comodule algebra extension problems for $m\le 2$. We can solve the extension problems for $m=3$ case by case in an ad hoc manner, but the number of such extensions is very large, so we do not attempt to include such a computation here.

When $m=1$, $\mathrm{kgl}$ is a form of $BPGL\langle 1\rangle$ and $\Gamma(\mathrm{kgl})\simeq \mathrm{ko}$, so this gives a quick computation of $H_*\mathrm{ko}$ that does not require knowledge of $\pi_*\mathrm{ko}$ or the Wood cofiber sequence. When $m=2$, we have an equivalence
\[\Gamma(BPGL\langle 2\rangle)\simeq BP_\R\langle 2\rangle^{C_2}\]
by Proposition \ref{propfixedpointsglobalsec}, and it is a theorem of Hill and Meier that $\mathrm{tmf}_1(3)$ is a form of $BP_\R\langle 2\rangle$ \cite{hillmeier}. Since $\mathrm{tmf}_1(3)^{C_2}\simeq \mathrm{tmf}_0(3)$, we give a complete computation of the comodule algebra
\[H_*\mathrm{tmf}_0(3)\]
which is new (Theorem \ref{thmt03}). We use this to deduce a Wood-type splitting of $\mathrm{tmf}$-modules
\[\mathrm{tmf}\wedge X\simeq \mathrm{tmf}_0(3)\]
where $X$ is a certain 10-cell complex that was predicted by Davis and Mahowald \cite{davismahowald}.

When $m=3$, it is not known if $\Gamma(BPGL\langle 3\rangle)$ even admits a ring structure. However, we may use the ring structure of $BPGL$ and the fact that the HSSS of $BPGL\langle 3\rangle$ is a module over that of $BPGL$.

\subsection{The homology of $\Gamma(BPGL\langle 1\rangle)$}\label{subsecn=1computation}
Combining Corollary \ref{corE_2description}, Theorem \ref{differentialsthm}, and Corollary \ref{coredgeA(m)}, we have the following.
\begin{theorem}
There is a spectral sequence of $\mathcal A_*$-comodule algebras with $E_2$-term given by
\[E_2^{*,*,*}=\F_2[\zeta_1^2,\zeta_2,\ldots][\rho,x_1,\ov_1]\]
that converges to $H_*\Gamma(BPGL\langle 1\rangle)$. We have $d_3$-differentials 
\begin{align*}
d_3(\zeta_1^2)&=\rho\ov_1
&d_3(\zeta_2)&=x_1\ov_1
\end{align*}
and the subalgebra generated by
\[\rho,x_1,x_2,\ov_1,\zeta_1^4,\zeta_2^2,\zeta_3,\ldots\]
consists of permanent cycles.
\end{theorem}

We display the $E_3$ page of this spectral sequence in Figure \ref{E3kr}. The black dots represent classes in the dual Steenrod algebra, and brown dots represent classes divisible by $\ov_1$. Brown structure lines represent multiplication by $\rho\ov_1$, which detects $\eta\in\pi_1S^0$ (see \cite{LSWX}), and we have drawn the vanishing lines of Proposition \ref{propvanishinglines} in green. It is straightforward to compute homology with respect to these differentials.

\begin{sseqdata}[ name = E3kr, Adams grading, classes = {fill, show name=below},
grid = go, xrange ={0}{7},yrange={-7}{7},xscale=1,yscale=1,x tick step =2, y tick step =2,run off differentials = {->},struct lines = red!25!brown ]
\class[name = 1](0,0)
\class[red!25!brown,name = \rho\ov_1](1,1)
\DoUntilOutOfBoundsThenNMore{2}{
    \class[red!25!brown](\lastx+1,\lasty+1)
    \structline
}
\structline(0,0)(1,1)
\class[red!25!brown,name = x_1\ov_1](2,0)
\DoUntilOutOfBoundsThenNMore{3}{
    \class[red!25!brown](\lastx+1,\lasty+1)
    \structline
}
\class[name = \zeta_1^2](2,-2)
\DoUntilOutOfBoundsThenNMore{3}{
    \d3
    \class[red!25!brown](\lastx+1,\lasty+1)
    \structline
}
\class[red!25!brown](4,0)
\DoUntilOutOfBoundsThenNMore{3}{
    \class[red!25!brown](\lastx+1,\lasty+1)
    \structline
}
\class[name = \zeta_2](3,-3)
\DoUntilOutOfBoundsThenNMore{3}{
    \d3
    \class[red!25!brown](\lastx+1,\lasty+1)
    \structline
}
\class[red!25!brown](4,-2)
\DoUntilOutOfBoundsThenNMore{3}{
    \d3
    \class[red!25!brown](\lastx+1,\lasty+1)
    \structline
}
\class[name=\zeta_1^4](4,-4)
\DoUntilOutOfBoundsThenNMore{3}{
    \class[red!25!brown](\lastx+1,\lasty+1)
    \structline
}

\class[red!25!brown](5,-3)
\DoUntilOutOfBoundsThenNMore{3}{
    \d3(\lastx,\lasty,2,2)
    \class[red!25!brown](\lastx+1,\lasty+1)
    \structline
}
\class[red!25!brown](6,-2)
\DoUntilOutOfBoundsThenNMore{3}{
    \d3(\lastx,\lasty,3,2)
    \class[red!25!brown](\lastx+1,\lasty+1)
    \structline
}
\class[red!25!brown](6,0)
\DoUntilOutOfBoundsThenNMore{3}{
    \class[red!25!brown](\lastx+1,\lasty+1)
    \structline
}
\class(5,-5)
\DoUntilOutOfBoundsThenNMore{3}{
    \d3
    \class[red!25!brown](\lastx+1,\lasty+1)
    \structline
}
\class[red!25!brown](6,-4)
\DoUntilOutOfBoundsThenNMore{3}{\d3
    \class[red!25!brown](\lastx+1,\lasty+1)
    \structline
}
\class[red!25!brown](7,-3)
\DoUntilOutOfBoundsThenNMore{3}{
    \d3(\lastx,\lasty,3,3)
    \class[red!25!brown](\lastx+1,\lasty+1)
    \structline
}
\class(6,-6)
\DoUntilOutOfBoundsThenNMore{3}{
    \d3
    \class[red!25!brown](\lastx+1,\lasty+1)
    \structline
}
\class[name=\zeta_2^2](6,-6)
\DoUntilOutOfBoundsThenNMore{3}{
    \class[red!25!brown](\lastx+1,\lasty+1)
    \structline
}
\class[red!25!brown](7,-5)
\DoUntilOutOfBoundsThenNMore{3}{\d3(\lastx,\lasty,3,2)
    \class[red!25!brown](\lastx+1,\lasty+1)
    \structline
}
\class(7,-7)
\DoUntilOutOfBoundsThenNMore{3}{
    \d3
    \class[red!25!brown](\lastx+1,\lasty+1)
    \structline
}
\class(8,-8)
\DoUntilOutOfBoundsThenNMore{3}{
    \d3(\lastx,\lasty,1,2)
    \class[red!25!brown](\lastx+1,\lasty+1)
    \structline
}
\class(48,48)
\class(60,-60)
\structline[green!38!black](0,0)(60,-60)
\end{sseqdata}
\begin{figure}[!htbp]
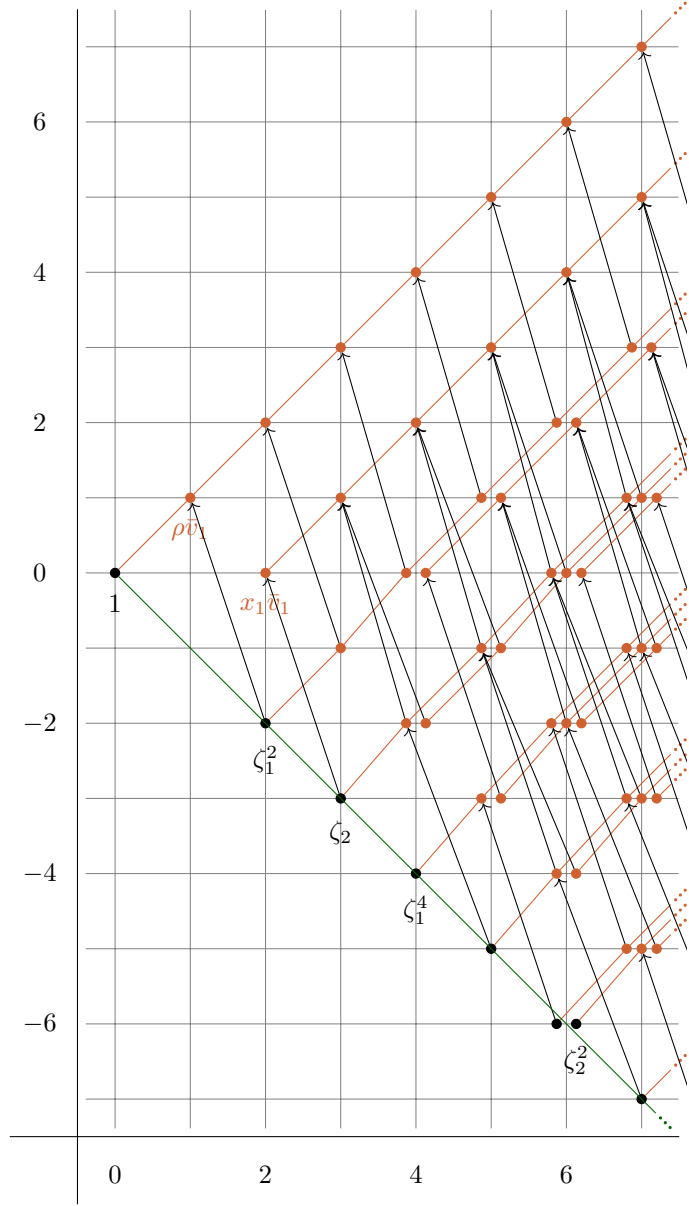

\centering
\printpage[name = E3kr]
\caption{The $E_3$ page of the HSSS for $BPGL\langle 1\rangle$.}
\label{E3kr}
\end{figure}

\begin{corollary}
    In the HSSS of $BPGL\langle 1\rangle$, we have
    \[E_\infty=E_4=\frac{\F_2[\zeta_1^4,\zeta_2^2,\zeta_3,\ldots][\rho,x_1,x_2,\ov_1]}{(\rho\ov_1,x_1\ov_1,x_2\ov_1,x_2^2+\rho^2\zeta_2^2+x_1^2\zeta_1^4)}\]
\end{corollary}

In weight zero, the relations imply that there is no contribution on $E_\infty$ from the generators in nonzero weights, and we recover the classical computation of the homology of $ko$ (displayed in Figure \ref{einftykr}). Each nonzero stem has exactly one nonzero filtration, so there are no nontrivial comodule algebra extensions.

\begin{sseqdata}[ name = Einftykr, Adams grading, classes = {fill, show name=below},
grid = go, xrange ={0}{9},yrange={-9}{1},xscale=0.5,yscale=0.5,x tick step =2, y tick step =2]
\class[name = 1](0,0)
\class[name=\zeta_1^4](4,-4)
\class[name=\zeta_2^2](6,-6)
\class[name=\zeta_3](7,-7)
\class[name=\zeta_1^8](8,-8)
\class(48,48)
\class(60,-60)
\structline[green!50!black](0,0)(48,48)
\structline[green!50!black](0,0)(60,-60)
\end{sseqdata}
\begin{figure}[!htbp]
\centering
\printpage[name = Einftykr]
\caption{The $E_\infty$ page of the HSSS for $BPGL\langle 1\rangle$.}
\label{einftykr}
\end{figure}

\begin{corollary}
    There is an isomorphism of $\mathcal{A}_*$-comodule algebras
    \[H_*(\mathrm{ko})\cong\mathcal{A}_*\square_{\mathcal{A}(1)_*}\F_2=\F_2[\zeta_1^4,\zeta_2^2,\zeta_3,\ldots]\]
\end{corollary}

\subsection{The homology of $\Gamma(BPGL\langle 2\rangle)$}\label{subsecn=2computation}

We move to height 2, and we display the $E_7$ page of the HSSS in Figure \ref{e7bp2}. As before, black dots represent classes in the dual Steenrod algebra, and red dots represent classes divisible by $\ov_1$. Orange structure lines represent multiplication by $\rho^3\ov_2$, which detects $\nu\in\pi_3S^0$, and we have drawn the vanishing lines of Proposition \ref{propvanishinglines} in green.

\begin{theorem}
There is a spectral sequence of \(\cA_{\ast}\)-comodule algebras with
\(E_{2}\)-term given by
\[E_2^{*,*,*}=\F_2[\zeta_1^2,\zeta_2,\ldots][\rho,x_1,\ov_1,\ov_2]\]
that converges to $H_*\Gamma(BPGL\langle 2\rangle)$. We have differentials 
\begin{align*}
d_3(\zeta_1^2)&=\rho\ov_1
&d_3(\zeta_2)&=x_1\ov_1
&d_7(\zeta_1^4)&=\rho^3\ov_2\\
d_7(\zeta_2^2)&=\rho x_1^2\ov_2
&d_7(\zeta_3)&=(x_1^3+\rho^2x_2)\ov_2
\end{align*}
The subalgebra of $E_2$ generated by the classes
\[\rho,x_1,x_2,x_3,\ov_1,\ov_2,\zeta_1^8,\zeta_2^4,\zeta_3^2,\zeta_4,\ldots\]
consists of permanent cycles.
\end{theorem}

The formulas for $d_3$ imply that, as a dga, one has an isomorphism
\[E_3^{*,*,*}(BPGL\langle 2\rangle;i_*H\F_2)\cong E_3^{*,*,*}(BPGL\langle 1\rangle;i_*H\F_2)[\ov_2]\]
and we deduce the following.

\begin{corollary}
The \(E_{4}=E_7\) page is given by
\[
\frac{\F_2[\zeta_1^4,\zeta_2^2,\zeta_3,\ldots]
[\rho,x_1,x_2,\ov_{1},\ov_{2}]}{(\rho\ov_1,x_1\ov_1,x_2\ov_1,x_2^2+\rho^2\zeta_2^2+x_1^2\zeta_1^4)}\]
\end{corollary}

Computing homology with respect to $d_7$ is much trickier as - unlike $E_3$ - $E_7$ is not $\rho$-torsion free, and hence there are a number of Massey products one must take into account. For example, we have the $d_7$-cycle
\[\ov_1\zeta_1^4\in\langle \ov_1,\rho^3,\ov_2\rangle\]
We are interested primarily in the computation in weight zero, however, and none of these classes can contribute to stems of weight zero. Indeed, we observe the following.

\begin{proposition}\label{propwithrelationslist}
The map of dga's
\[E_7\to E_7/(\ov_1)=\frac{\F_2[\zeta_1^4,\zeta_2^2,\zeta_3,\ldots]
[\rho,x_1,x_2,\ov_{2}]}{(x_2^2+\rho^2\zeta_2^2+x_1^2\zeta_1^4)}\]
is an isomorphism in weights $\le0$.
\end{proposition}
\begin{proof}
The relations $\rho\ov_1=x_1\ov_1=x_2\ov_1=0$ imply the map is an isomorphism in weights $\le0$. It is a map of dga's because the ideal $(\ov_1)\subset E_7$ is  contained in $\ker(d_7)$.
\end{proof}

We therefore instead compute the homology of the simpler dga $(E_7/(\ov_1),d_7)$. We claim that $\ker(d_7)$ is the subalgebra of $E_7/(\ov_1)$ generated by the classes
\[\rho,x_1,x_2,x_3,\ov_2,\zeta_1^8,\zeta_2^4,\zeta_3^2,\zeta_4,\ldots\]
It follows that we have the presentation
\[
\ker(d_7)=\frac{\F_2[\zeta_1^8,\zeta_2^4,\zeta_3^2,\zeta_4,\ldots][\rho,x_1,x_2,x_3,\ov_2]}{(x_2^4+\rho^4\zeta_2^4+x_1^4\zeta_1^8,x_3^2+\rho^2\zeta_3^2+x_1^2\zeta_2^4+x_2^2\zeta_1^8)}
\]
Writing our differentials in terms of this description of the kernel, we have
\begin{align*}
d_7(\zeta_1^4)&=\rho^3\ov_2
&d_7(\zeta_2^2)&=\rho x_1^2\ov_2\\
d_7(\zeta_3)&=(x_1^3+\rho^2x_2)\ov_2
&d_7(\zeta_1^4\zeta_2^2)&=\rho x_2^2\ov_2\\
d_7(\zeta_1^4\zeta_3)&=(x_1x_2^2+\rho^2x_3)\ov_2
&d_7(\zeta_2^2\zeta_3)&=(x_1^2x_3+x_2^3)\ov_2\\
d_7(\zeta_1^4\zeta_2^2\zeta_3)&=(x_1^2x_2\zeta_1^8+x_2^2x_3+\rho^2x_1\zeta_2^4)\ov_2
\end{align*}
Letting $I_2$ be the ideal in $\ker(d_7)$ generated by this list of relations, it follows that 
\[H_*(E_7/(\ov_1))=\frac{\F_2[\zeta_1^8,\zeta_2^4,\zeta_3^2,\zeta_4,\ldots][\rho,x_1,x_2,x_3,\ov_2]}{I_2+(x_2^4+\rho^4\zeta_2^4+x_1^4\zeta_1^8,x_3^2+\rho^2\zeta_3^2+x_1^2\zeta_2^4+x_2^2\zeta_1^8)}\]

It remains to verify that the claimed list of elements indeed generate $\ker(d_7)$. We achieve this by running the $\rho$-Bockstein spectral sequence, which converges since $E_7/(\ov_1)$ is (graded) $\rho$-complete. We begin by computing the associated graded of the $\rho$-adic filtration.

\begin{proposition}
    The $E_0$-page of the $\rho$-Bockstein spectral sequence is given by
    \[\frac{\F_2[\zeta_1^4,\zeta_2^2,\zeta_3,\ldots]
[\rho,x_1,x_2,\ov_{2}]}{(x_2^2+x_1^2\zeta_1^4)}\]
This spectral sequence converges to the homology of $E_7/(\ov_1)$. We have differentials
\begin{align*}
\delta_0(\zeta_3)&=x_1^3\ov_2
&\delta_1(\zeta_2^2)&=\rho x_1^2\ov_2
&\delta_3(\zeta_1^4)&=\rho^3\ov_2
\end{align*}
\end{proposition}

The annihilator ideal of $\delta_0(\zeta_3)$ in $E_0$ is trivial, so $\delta_0$ does not create any Massey product cycles. We conclude the following.

\begin{proposition}
    The $E_1$-page of the $\rho$-Bockstein spectral sequence is given by
    \[\frac{\F_2[\zeta_1^4,\zeta_2^2,\zeta_3^2,\zeta_4,\ldots]
[\rho,x_1,x_2,\ov_{2}]}{(x_2^2+x_1^2\zeta_1^4,x_1^3\ov_2)}\]
\end{proposition}

To compute the $E_2$-page, we note that the class $x_3$ is a cycle in $E_7/(\ov_1)$ and therefore it projects to a permanent cycle in the $\rho$-Bockstein spectral sequence. We also denote its projection
\[x_1\zeta_2^2+x_2\zeta_1^4\]
by $x_3$.

\begin{proposition}
    The $E_3$-page of the $\rho$-Bockstein spectral sequence is given by 
    \[
    \frac{\F_2[\zeta_1^4,\zeta_2^4,\zeta_3^2,\zeta_4,\ldots][\rho,x_1,x_2,x_3,\ov_2]}{(x_3^2+x_2^2\zeta_1^8+x_1^2\zeta_2^4,x_2^2+x_1^2\zeta_1^4,x_1^3\ov_2,\rho x_1^2\ov_2,(x_1^2x_3+x_2^3)\ov_2)}
    \]
\end{proposition}

\begin{proof}
    The kernel of $\delta_1$ is generated as an algebra by the classes 
    \[\rho,x_1,x_2,\ov_2,\zeta_1^4,\zeta_2^4,\zeta_3^2,\zeta_4,\ldots\]
    and $\zeta_2^2\cdot\mathrm{ann}_{E_1}(\rho x_1^2\ov_2)$, and we claim $\mathrm{ann}_{E_1}(\rho x_1^2\ov_2)=(x_1)$. To see this, let
    \[R:=\frac{\F_2[\zeta_1^4,\zeta_2^2,\zeta_3^2,\zeta_4,\ldots]
[\rho,x_1,x_2,\ov_{2}]}{(x_2^2+x_1^2\zeta_1^4)}\]
There is a surjection $\pi:R\to R/(x_1^3\ov_2)=E_1$, so if
\[\pi(r)\cdot \rho x_1^2\ov_2=0\]
then $r\rho x_1^2\ov_2\in(x_1^3\ov_2)\subset(x_1^3)$ in $R$. Using that $R$ is of the form $T[\rho,\ov_2]$ for $T$ an $x_1$-torsion free ring, it follows that $r\in(x_1)$.

Therefore, the map
\[S:=\frac{\F_2[\zeta_1^4,\zeta_2^4,\zeta_3^2,\zeta_4,\ldots][\rho,x_1,x_2,x_3,\ov_2]}{(x_3^2+x_2^2\zeta_1^8+x_1^2\zeta_2^4,x_2^2+x_1^2\zeta_1^4,x_1^3\ov_2,(x_1^2x_3+x_2^3)\ov_2)}\to\ker(\delta_1)\]
is surjective. To see that it is injective, note that the map
\[S\to S[y]/(y^2+\zeta_2^4,x_1y+x_3+x_2\zeta_1^4)\]
is injective, and the latter is easily identified with $E_1$ by setting $y=\zeta_2^2$. Now, $\mathrm{im}(\delta_1)$ is the ideal in $S$ generated by $\rho x_1^2\ov_2$, and the result follows.
\end{proof}

\begin{proposition}\label{propEinftyrhobocksteinn=2}
    The $E_4$ page of the $\rho$-Bockstein spectral sequence is given by the quotient of 
    \[
    \F_2[\zeta_1^8,\zeta_2^4,\zeta_3^2,\zeta_4,\ldots][\rho,x_1,x_2,x_3,\ov_2]
    \]
    by the ideal generated by the elements
    \begin{align*}
        &x_3^2+x_2^2\zeta_2^4+x_1^2\zeta_1^8
        &&x_2^4+x_1^4\zeta_1^8
        &&x_1^3\ov_2\\
        &\rho x_1^2\ov_2
        &&\rho^3\ov_2
        &&(x_1^2x_3+x_2^3)\ov_2\\
        &x_2^2x_3\ov_2+\zeta_1^8x_1^2x_2\ov_2
        &&x_2^2\rho\ov_2
        &&x_1x_2^2\ov_2
    \end{align*}
    
\end{proposition}
\begin{proof}
    The kernel of $\delta_3$ is generated by the classes 
    \[\rho,x_1,x_2,x_3,\ov_2,\zeta_1^8,\zeta_2^4,\zeta_3^2,\zeta_4,\ldots\]
    and $\zeta_1^4\cdot\mathrm{ann}_{E_3}(\rho^3\ov_2)$, and we claim $\mathrm{ann}_{E_3}(\rho^3\ov_2)=(x_1^2)$. To see this, set
    \[R:=\frac{\F_2[\zeta_1^4,\zeta_2^4,\zeta_3^2,\zeta_4,\ldots][\rho,x_1,x_2,x_3,\ov_2]}{(x_3^2+x_1^2\zeta_2^4+x_2^2\zeta_1^8,x_2^2+x_1^2\zeta_1^4)}
    \]
    We have a surjection 
    \[\pi:R\to R/(\rho x_1^2\ov_2,x_1^3\ov_2,(x_1^2x_3+x_2^3)\ov_2)=E_3\]
    So  if $\pi(r)\cdot \rho^3\ov_2=0$, then 
    \[r\rho^3\ov_2\in (\rho x_1^2\ov_2,x_1^3\ov_2,(x_1^2x_3+x_2^3)\ov_2)\subset(x_1^2)\]
  in $R$. Using that $R$ is of the form $T[\rho,\ov_2]$ for $T$ an $x_1$-torsion free ring, we see that $r\in (x_1^2)$.
    
    In particular, we do not pick up any new Massey product cycles as $x_1^2\zeta_1^4=x_2^2$, and therefore the map
    \[S:=\F_2[\zeta_1^8,\zeta_2^4,\zeta_3^2,\zeta_4,\ldots][\rho,x_1,x_2,x_3,\ov_2]/J\to \ker(\partial_7)\]
    is surjective, where $J$ is the ideal generated by the list in the statement of the proposition, with the element $\rho^3\ov_2$ removed. To see that this map is injective, note that the map
    \[S\to S[y]/(y^2+\zeta_1^8,x_1^2y+x_2^2)\]
    is injective, and the latter is easily identified with $E_3$ by setting $y=\zeta_1^4$.
\end{proof}

The remaining algebra generators are all permanent cycles, and so the $\rho$-Bockstein spectral sequence collapses on $E_4$. This verifies the discussion following Proposition \ref{propwithrelationslist}.

\subsubsection{The homology of $\mathrm{tmf}_0(3)$} Moving to weight zero, we have the subalgebra $\mathcal A_*\square_{\mathcal A(2)_*}\F_2$ along with the following ten elements
\begin{align*}
1 & \in H_{0} 
    &\rho^{2}x_1\ov_{2}
      &\in H_{4}
        & \rho^2 x_{2}\ov_{2} &\in H_{6}\\
    \rho x_1 x_2\ov_{2}
      &\in H_{7}
        &x_1^{2}x_2\ov_{2}
          &\in H_{8}
&\rho^2 x_{3}\ov_{2}
  &\in H_{10}\\ 
    \rho x_1x_3\ov_{2}
      &\in H_{11}   
        &x_1^{2}x_3\ov_{2}
          &\in H_{12}
&\rho x_{2}x_{3}\ov_{2}
  &\in H_{13} \\
    x_1x_{2}x_{3}\ov_{2}
      &\in H_{14} 
\end{align*}

These form an $\mathcal A(2)_*$-comodule algebra $M_2$ of dimension 10 with coactions determined by Corollary \ref{corcoactionx_m} and the fact that $\rho$ and $\ov_2$ are primitive. We display $M_2$ in Figure \ref{figM2}, where we have omitted the unit class as it generates a trivial comodule summand.

\begin{theorem}\label{thmt03}
    The homology of $\mathrm{\mathrm{tmf}}_0(3)\simeq\Gamma(BPGL\langle 2\rangle)$ is isomorphic as an $\mathcal A_*$-comodule algebra to
    \[\mathcal A_*\square_{\mathcal A(2)_*}M_2\]
    where $\overline{M}_{2}$ is as in Figure \ref{figM2}, and the multiplication in $\overline{M}_{2}$ is square zero.
\end{theorem}
\begin{proof}
   The description follows for the $E_\infty$ page of the HSSS for $BPGL\langle 2\rangle$ by the discussion following Proposition \ref{propwithrelationslist}. This page is displayed in Figure \ref{einftybp2}.
   
   It suffices now to observe that there are no nontrivial comodule algebra extensions in the HSSS. Indeed, each class in $M_2$ is in the highest filtration in its respective stem so the projection $H_*\mathrm{tmf}_0(3)\to M_2$ is a map of $\mathcal{A}(2)_*$-comodule algebras, which is adjoint to an isomorphism of $\mathcal{A}_*$-comodule algebras
   \[H_*\mathrm{tmf}_0(3)\to \mathcal A_*\square_{\mathcal A(2)_*}M_2\qedhere\]
\end{proof}

Prior to the results of Hill--Lawson \cite{hilllawson}, it was not known how to construct a connective model $\mathrm{tmf}_0(3)$ of the periodic spectrum $\mathrm{TMF}_0(3)$ via a derived algebraic geometry approach. Nonetheless, computational aspects of such a spectrum were studied in detail, and Davis-Mahowald proposed several definitions of such a spectrum. In particular, they constructed a certain 10-cell complex $X$ and considered the spectrum $\mathrm{tmf}\wedge X$ as a good connective model. They give an explicit construction of $X$, but we show that it exists by Toda obstruction theory.

\begin{proposition}\label{proptoda}
    Let $N$ be the $\mathcal{A}$-module with one generator in each dimension $3,5,6,7,9,10,11,12,13$, where the following Steenrod squares are nonzero on the generator $g$ of dimension 3
    \[\mathrm{Sq}^2,\mathrm{Sq}^3,\mathrm{Sq}^4,\mathrm{Sq}^4\mathrm{Sq}^2,\mathrm{Sq}^5\mathrm{Sq}^2,\mathrm{Sq}^6\mathrm{Sq}^2=\mathrm{Sq}^8,\mathrm{Sq}^6\mathrm{Sq}^3,\mathrm{Sq}^7\mathrm{Sq}^3\]
    and $\mathrm{Sq}^6(g)=0$. Then there exists a unique 2-complete bounded below spectrum $Y$ with $H^*Y\cong N$. Moreover, there is a map $Y\to S^0$ extending $2\nu$.
\end{proposition}
\begin{proof}
    Toda obstruction theory (see \cite[Theorem 3.2]{BE} or \cite{Toda}) implies that it suffices to show that the $-2$ stem in
    \[\Ext_{\mathcal{A}}^{*,*}(N,N)\]
    is trivial. This is easily checked using Bruner's software, and we display the chart in Figure \ref{bruner4}. The $-1$-stem contains no classes in filtration higher than 1, so uniqueness follows. The existence of the map $Y\to S^0$ may also be checked directly from the Adams spectral sequence, and this is the argument used in \cite[Theorem 2.1(b)]{davismahowald}.
\end{proof}
 
\begin{figure}
\centering
\includegraphics[scale=0.35]{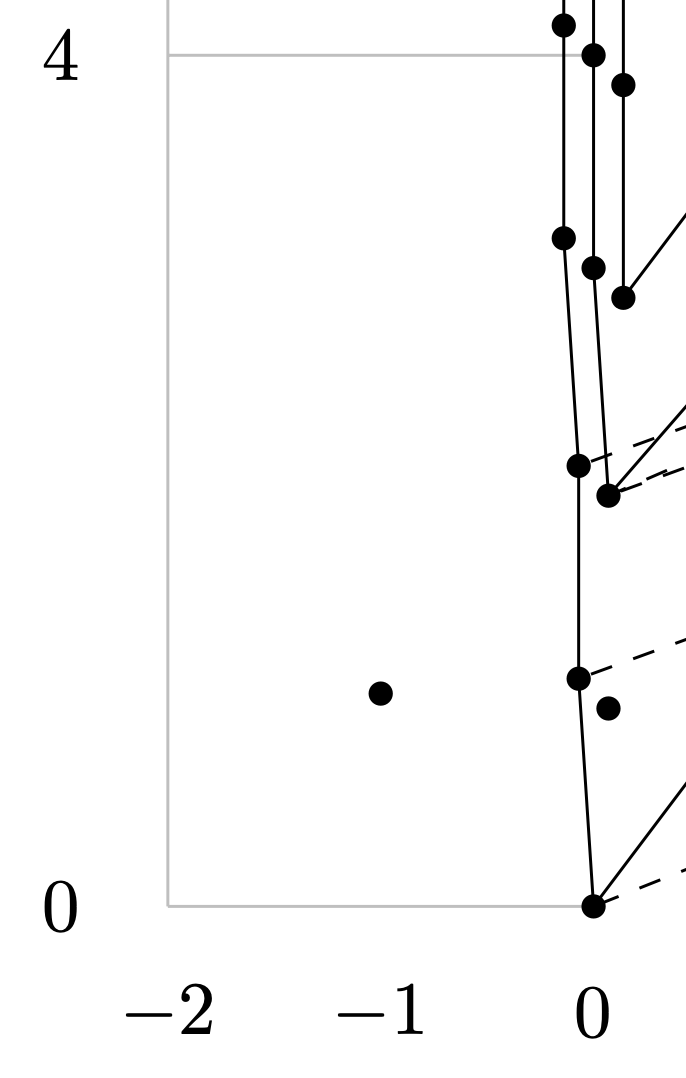}
\caption{$\Ext_{\mathcal{A}}(N,N)$}
\label{bruner4}
\end{figure}

It is easy to check that the $\mathcal{A}$-module $N$ has the property that its restriction to $\mathcal{A}(2)$ is dual to the $\mathcal{A}(2)_*$-comodule $\Sigma^{-1}\overline{M}_{2}$. We conclude the following.

\begin{corollary}\label{cortmfsplitting}
    Let $X$ be the cofiber of the map $Y\to S^0$ constructed in Proposition \ref{proptoda}. There is a 2-local equivalence of $\mathrm{tmf}$-modules
    \[\mathrm{tmf}\wedge X\to \mathrm{tmf}_0(3)\]
\end{corollary}
\begin{proof}
    The homotopy groups of $\mathrm{tmf}$ and $\mathrm{tmf}_0(3)$ are degreewise finitely generated, so it suffices to produce a $2$-complete equivalence, which follows from the existence of any map $X\to \mathrm{tmf}_0(3)$ that induces the map of $\mathcal{A}_*$-comodules $H_*X\to \mathcal A_*\square_{\mathcal A(2)_*}M_2$ adjoint to the identity map of the $\mathcal{A}(2)_*$-comodule $M_2$.
    
    For this, it suffices to show that in the Adams spectral sequence
    \[\Ext_{\mathcal{A}_*}(H_*X,H_*\mathrm{tmf}_0(3))\cong\Ext_{\mathcal{A}(2)_*}(M_2,M_2)\implies [X,\mathrm{tmf}_0(3)]^{\hat{}}_2\]
    the identity map $\iota$ is a permanent cycle. By Proposition \ref{proptoda}, the 4-skeleton of $X$ is the 2-cell complex $C(2\nu)$. Since $2\nu=0\in\pi_3\mathrm{tmf}_0(3)$ (see \cite{hillmeier} or \cite{LSWX}), the unit map $S^0\to \mathrm{tmf}_0(3)$ extends over $C(2\nu)$, which implies that the map of Adams spectral sequences
    \[\Ext_{\mathcal{A}_*}(H_*X,H_*\mathrm{tmf}_0(3))\to \Ext_{\mathcal{A}_*}(H_*C(2\nu),H_*\mathrm{tmf}_0(3))\]
    sends $\iota$ to a permanent cycle.
    
    We use Bruner's ext software to produce the former in Figure \ref{bruner} and the latter in Figure \ref{bruner1}. The map is an isomorphism in bidegrees $(0,0)$, $(-1,2)$ and $(-1,3)$, sending $\iota$ to the generator of the lefthand $h_0$-tower in Figure \ref{bruner1}. It follows by naturality that $\iota$ does not support a $d_{<4}$ in Figure \ref{bruner}. We claim the generator of the righthand $h_0$-tower in Figure \ref{bruner1} supports a nontrivial $d_2$. Given this, the nontrivial $d_2$ lifts to Figure \ref{bruner}, and it follows from $h_0$-linearity that the $-1$ stem is then trivial on the $E_3$-page.

    It remains to show the claimed $d_2$ in Figure \ref{bruner1}. This follows by naturality via the map
    \[\Ext_{\mathcal{A}_*}(H_*C(2\nu),\F_2)\to\Ext_{\mathcal{A}_*}(H_*C(2\nu),H_*\mathrm{tmf}_0(3))\]
    and we display the former $E_2$ page in Figure \ref{bruner2}. Since 
    \[D(C(2\nu))\not\simeq S^{-4}\vee S^0\]
    the class in bidegree $(0,0)$ must support a differential. It follows easily from the long exact sequence in homotopy that $\pi_{-1}D(C(2\nu))=\pi_3C(2\nu)=\Z/2$, which implies the claimed $d_2$.
\end{proof}

\begin{corollary}\label{cortmf03collapse}
    The Adams spectral sequence for $\mathrm{tmf}_0(3)$ does not collapse on $E_2$. In particular, we have $d_2(x)=h_0h_2$, where $x$ is the nontrivial class in bidegree $(4,0)$.
\end{corollary}

\begin{proof}
    The claimed $d_2$ follows by naturality via the map $C(2\nu)\to\mathrm{tmf}_0(3)$, where the Adams spectral sequence for $C(2\nu)$ supports the nontrivial $d_2$ discussed in the proof of Corollary \ref{cortmfsplitting}. We display the Adams $E_2$-page for $\mathrm{tmf}_0(3)$ in Figure \ref{bruner3}.
\end{proof}

\begin{remark}
    Corollary \ref{cortmf03collapse} stands in contrast to the motivic Adams spectral sequence, which collapses for $BPGL\langle m\rangle$ for all $m$ (similarly for the $C_2$-equivariant Adams spectral sequence for $BP_\R\langle m\rangle$).
\end{remark}

\begin{figure}[!htbp]
\centering
\includegraphics[scale=0.4]{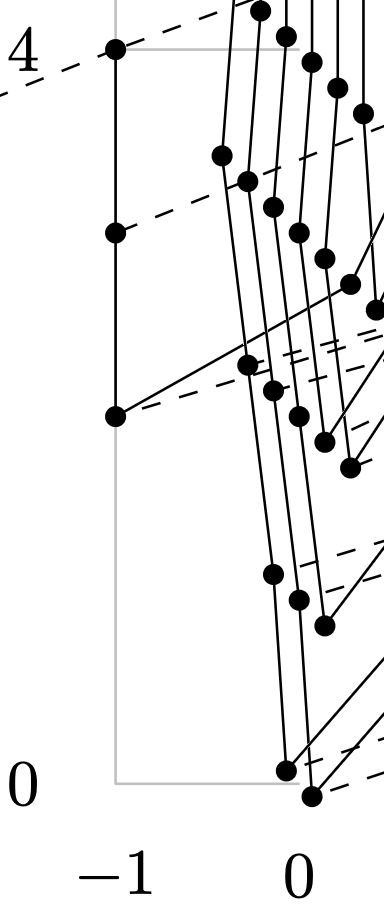}
\caption{$\Ext_{\mathcal{A}_*}(H_*X,H_*\mathrm{tmf}_0(3))$}
\label{bruner}
\end{figure}

\begin{figure}[!htbp]
\centering
\begin{subfigure}{.5\textwidth}
  \centering
  \includegraphics[width=\textwidth]{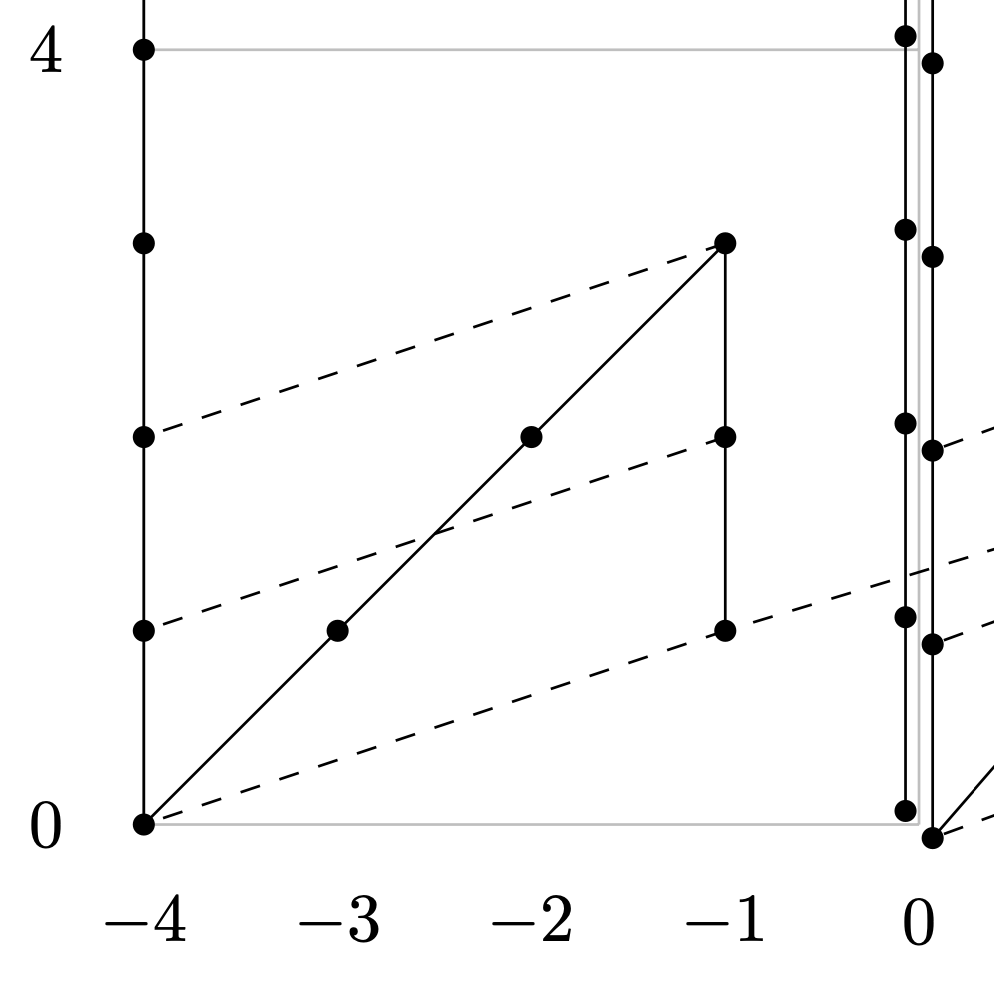}
\caption{$\Ext_{\mathcal{A}_*}(H_*C(2\nu),H_*\mathrm{tmf}_0(3))$}
\label{bruner1}
\end{subfigure}%
\begin{subfigure}{.5\textwidth}
  \centering
  \includegraphics[width=\textwidth]{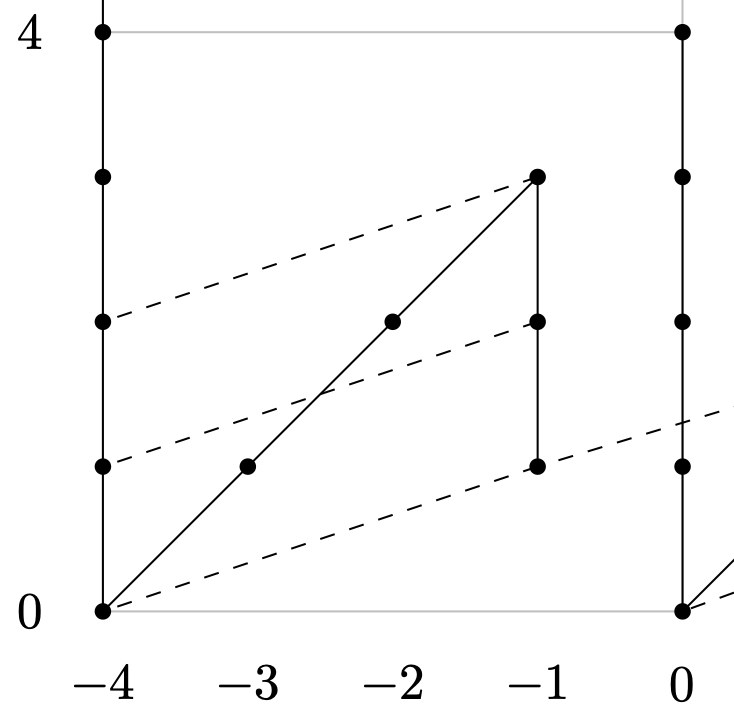}
\caption{$\Ext_{\mathcal{A}_*}(H_*C(2\nu),H_*S^0)$}
\label{bruner2}
\end{subfigure}
\caption{}
\end{figure}

\begin{figure}[!htbp]
\centering
\includegraphics[scale=0.4]{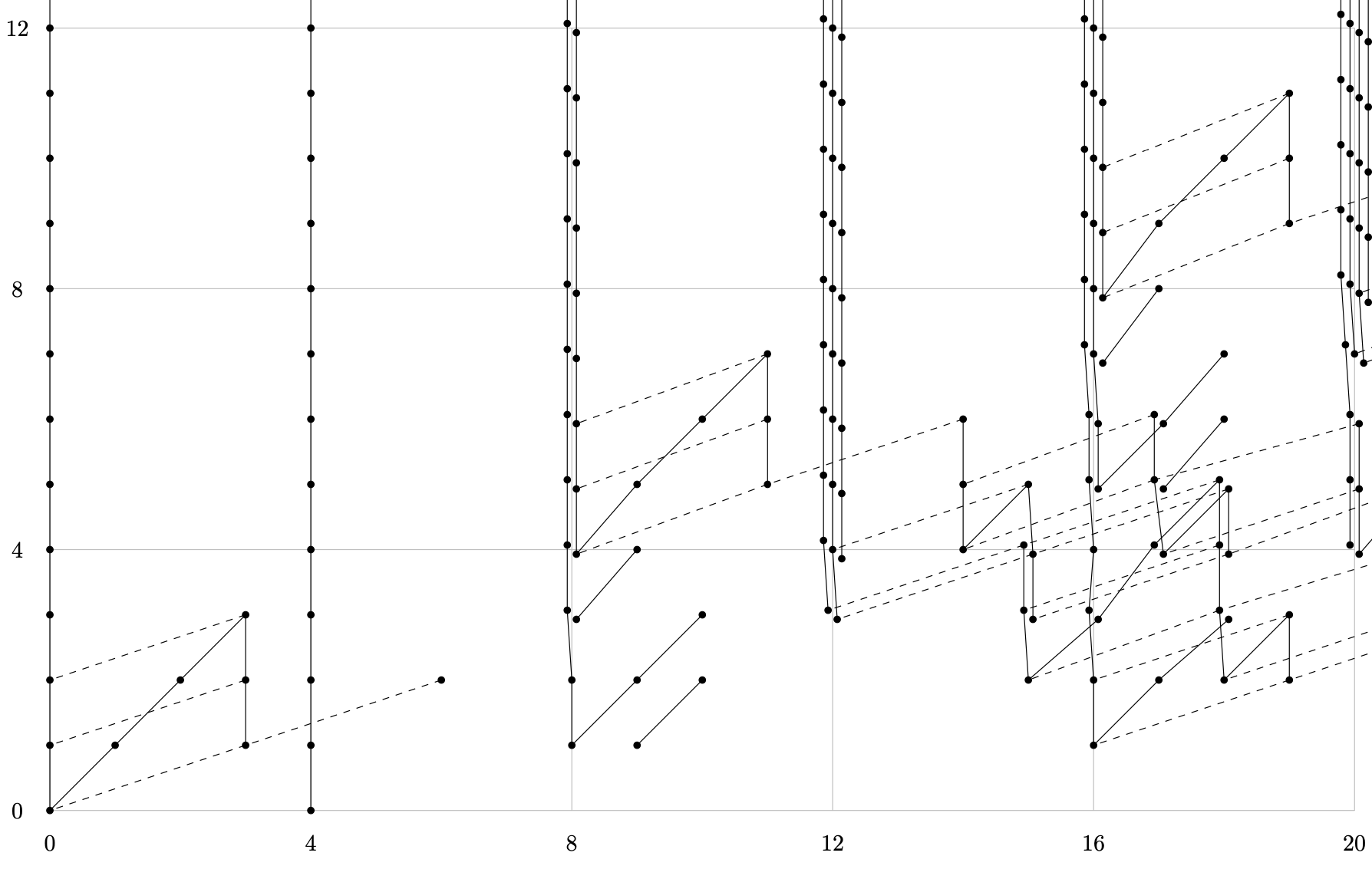}
\caption{$\Ext_{\mathcal{A}_*}(\F_2,H_*\mathrm{tmf}_0(3))$}
\label{bruner3}
\end{figure}

\begin{sseqdata}[ name = E7bp2, Adams grading, classes = {fill, show name=below},
grid = go, xrange ={0}{14},yrange={-14}{14},xscale=0.5,yscale=0.5,x tick step =2, y tick step =2,run off differentials = {->},struct lines = red!88!black ]
\class[name = 1](0,0)
\class[red!88!black,name = \rho^3\ov_2](3,3)
\DoUntilOutOfBoundsThenNMore{2}{
    \class[red!88!black](\lastx+3,\lasty+3)
    \structline
}

\structline(0,0)(3,3)
\class[red!88!black](4,2)
\DoUntilOutOfBoundsThenNMore{3}{
    \class[red!88!black](\lastx+3,\lasty+3)
    \structline
}
\class[red!88!black,name=\rho x_1^2\ov_2](5,1)
\DoUntilOutOfBoundsThenNMore{3}{
    \class[red!88!black](\lastx+3,\lasty+3)
    \structline
}
\class[red!88!black,name=x_1^3\ov_2](6,0)
\DoUntilOutOfBoundsThenNMore{3}{
    \class[red!88!black](\lastx+3,\lasty+3)
    \structline
}
\class[red!88!black](6,0)
\DoUntilOutOfBoundsThenNMore{3}{
    \class[red!88!black](\lastx+3,\lasty+3)
    \structline
}
\class[name = \zeta_1^4](4,-4)
\DoUntilOutOfBoundsThenNMore{3}{
    \d7
    \class[red!88!black](\lastx+3,\lasty+3)
    \structline
}
\class[red!88!black](7,-1)
\DoUntilOutOfBoundsThenNMore{3}{
    \class[red!88!black](\lastx+3,\lasty+3)
    \structline
}
\class[red!88!black](10,2)
\DoUntilOutOfBoundsThenNMore{3}{
    \class[red!88!black](\lastx+3,\lasty+3)
    \structline
}
\class[red!88!black](8,-2)
\DoUntilOutOfBoundsThenNMore{1}{\d7
    \class[red!88!black](\lastx+3,\lasty+3)
    \structline
}
\class[red!88!black](8,-2)
\DoUntilOutOfBoundsThenNMore{1}{
    \class[red!88!black](\lastx+3,\lasty+3)
    \structline
}
\class[red!88!black](11,1)
\DoUntilOutOfBoundsThenNMore{1}{
    \class[red!88!black](\lastx+3,\lasty+3)
    \structline
}
\class[name = \zeta_2^2](6,-6)
\DoUntilOutOfBoundsThenNMore{3}{
    \d7
    \class[red!88!black](\lastx+3,\lasty+3)
    \structline
}
\class[red!88!black](9,-3)
\DoUntilOutOfBoundsThenNMore{1}{
    \class[red!88!black](\lastx+3,\lasty+3)
    \structline
}
\class[red!88!black](12,0)
\DoUntilOutOfBoundsThenNMore{1}{
    \class[red!88!black](\lastx+3,\lasty+3)
    \structline
}
\class[red!88!black](12,0)
\DoUntilOutOfBoundsThenNMore{1}{
    \class[red!88!black](\lastx+3,\lasty+3)
    \structline
}
\class[name = \zeta_3](7,-7)
\DoUntilOutOfBoundsThenNMore{3}{
    \d7
    \class[red!88!black](\lastx+3,\lasty+3)
    \structline
}
\class[red!88!black](10,-4)
\DoUntilOutOfBoundsThenNMore{1}{\d7(\lastx,\lasty,2,2)
    \class[red!88!black](\lastx+3,\lasty+3)
    \structline
}
\class[red!88!black](10,-4)
\DoUntilOutOfBoundsThenNMore{1}{\d7(\lastx,\lasty,3,2)
    \class[red!88!black](\lastx+3,\lasty+3)
    \structline
}
\class[red!88!black](10,-4)
\DoUntilOutOfBoundsThenNMore{1}{
    \class[red!88!black](\lastx+3,\lasty+3)
    \structline
}
\class[red!88!black](13,-1)
\DoUntilOutOfBoundsThenNMore{1}{
    \class[red!88!black](\lastx+3,\lasty+3)
    \structline
}
\class[name = \zeta_1^8](8,-8)
\DoUntilOutOfBoundsThenNMore{3}{
    \class[red!88!black](\lastx+3,\lasty+3)
    \structline
}
\class[red!88!black](11,-5)
\DoUntilOutOfBoundsThenNMore{1}{\d7(\lastx,\lasty,2,2)
    \class[red!88!black](\lastx+3,\lasty+3)
    \structline
}
\class[red!88!black](11,-5)
\DoUntilOutOfBoundsThenNMore{1}{\d7(\lastx,\lasty,3,2)
    \class[red!88!black](\lastx+3,\lasty+3)
    \structline
}
\class[red!88!black](11,-5)
\DoUntilOutOfBoundsThenNMore{1}{\d7(\lastx,\lasty,4,3)
    \class[red!88!black](\lastx+3,\lasty+3)
    \structline
}
\class[red!88!black](12,-6)
\DoUntilOutOfBoundsThenNMore{3}{\d7(\lastx,\lasty,1,2)
    \class[red!88!black](\lastx+3,\lasty+3)
    \structline
}
\class[red!88!black](12,-6)
\DoUntilOutOfBoundsThenNMore{3}{\d7(\lastx,\lasty,2,3)
    \class[red!88!black](\lastx+3,\lasty+3)
    \structline
}
\class[red!88!black](12,-6)
\DoUntilOutOfBoundsThenNMore{3}{\d7(\lastx,\lasty,3,3)
    \class[red!88!black](\lastx+3,\lasty+3)
    \structline
}
\class[red!88!black](12,-6)
\DoUntilOutOfBoundsThenNMore{3}{
    \class[red!88!black](\lastx+3,\lasty+3)
    \structline
}
\class[red!88!black](12,-6)
\DoUntilOutOfBoundsThenNMore{3}{
    \class[red!88!black](\lastx+3,\lasty+3)
    \structline
}
\class(10,-10)
\DoUntilOutOfBoundsThenNMore{1}{
    \d7(\lastx,\lasty,1,2)
    \class[red!88!black](\lastx+3,\lasty+3)
    \structline
}
\class[red!88!black](13,-7)
\DoUntilOutOfBoundsThenNMore{2}{\d7(\lastx,\lasty,2,2)
    \class[red!88!black](\lastx+3,\lasty+3)
    \structline
}
\class[red!88!black](13,-7)
\DoUntilOutOfBoundsThenNMore{1}{\d7(\lastx,\lasty,3,3)
    \class[red!88!black](\lastx+3,\lasty+3)
    \structline
}
\class[red!88!black](13,-7)
\DoUntilOutOfBoundsThenNMore{1}{\d7(\lastx,\lasty,4,4)
    \class[red!88!black](\lastx+3,\lasty+3)
    \structline
}
\class[red!88!black](13,-7)
\DoUntilOutOfBoundsThenNMore{3}{
    \class[red!88!black](\lastx+3,\lasty+3)
    \structline
}
\class(11,-11)
\DoUntilOutOfBoundsThenNMore{3}{
    \d7
    \class[red!88!black](\lastx+3,\lasty+3)
    \structline
}
\class[red!88!black](14,-8)
\DoUntilOutOfBoundsThenNMore{2}{\d7(\lastx,\lasty,2,2)
    \class[red!88!black](\lastx+3,\lasty+3)
    \structline
}
\class[red!88!black](14,-8)
\DoUntilOutOfBoundsThenNMore{2}{\d7(\lastx,\lasty,3,3)
    \class[red!88!black](\lastx+3,\lasty+3)
    \structline
}
\class[red!88!black](14,-8)
\DoUntilOutOfBoundsThenNMore{2}{\d7(\lastx,\lasty,4,4)
    \class[red!88!black](\lastx+3,\lasty+3)
    \structline
}
\class[red!88!black](14,-8)
\DoUntilOutOfBoundsThenNMore{3}{\d7
    \class[red!88!black](\lastx+3,\lasty+3)
    \structline
}
\class[red!88!black](14,-8)
\DoUntilOutOfBoundsThenNMore{3}{\d7
    \class[red!88!black](\lastx+3,\lasty+3)
    \structline
}
\class(12,-12)
\DoUntilOutOfBoundsThenNMore{3}{\d7
    \class[red!88!black](\lastx+3,\lasty+3)
    \structline
}
\class(12,-12)
\DoUntilOutOfBoundsThenNMore{3}{
    \class[red!88!black](\lastx+3,\lasty+3)
    \structline
}
\class[red!88!black](15,-9)
\DoUntilOutOfBoundsThenNMore{3}{
    \d7(\lastx,\lasty,3,4)
    \class[red!88!black](\lastx+3,\lasty+3)
    \structline
}
\class[red!88!black](15,-9)
\DoUntilOutOfBoundsThenNMore{3}{
    \d7
    \class[red!88!black](\lastx+3,\lasty+3)
    \structline
}
\class[red!88!black](15,-9)
\DoUntilOutOfBoundsThenNMore{3}{
    \d7
    \class[red!88!black](\lastx+3,\lasty+3)
    \structline
}
\class[red!88!black](15,-9)
\DoUntilOutOfBoundsThenNMore{3}{
    \d7
    \class[red!88!black](\lastx+3,\lasty+3)
    \structline
}
\class(13,-13)
\DoUntilOutOfBoundsThenNMore{3}{
    \d7(\lastx,\lasty,1,5)
    \class[red!88!black](\lastx+3,\lasty+3)
    \structline
}
\class(14,-14)
\DoUntilOutOfBoundsThenNMore{1}{
    \d7(\lastx,\lasty,1,5)
    \class[red!88!black](\lastx+3,\lasty+3)
    \structline
}
\class(14,-14)
\DoUntilOutOfBoundsThenNMore{3}{
    \class[red!88!black](\lastx+3,\lasty+3)
    \structline
}
\class(15,-15)
\DoUntilOutOfBoundsThenNMore{3}{
    \d7(\lastx,\lasty,1,5)
    \class[red!88!black](\lastx+3,\lasty+3)
    \structline
}
\class(48,48)
\class(60,-60)
\structline[green!50!black](0,0)(60,-60)
\end{sseqdata}
\begin{figure}[!htbp]
\centering
\printpage[name = E7bp2]
\caption{The $E_7$ page of the HSSS for $BPGL\langle 2\rangle$.}
\label{e7bp2}
\end{figure}

\begin{sseqdata}[ name = Einftybp2, Adams grading, classes = {fill, show name=below},
grid = go, xrange ={0}{16},yrange={-16}{3},xscale=0.5,yscale=0.5,x tick step =2, y tick step =2, font = \tiny ]
\class[name = 1](0,0)
\class[name=\zeta_1^8](8,-8)
\class[name=\zeta_2^4](12,-12)
\class[name=\zeta_3^2](14,-14)
\class[red!88!black,name=\rho^2x_1\ov_2](4,2)
\class[red!88!black,name=\rho^2x_2\ov_2](6,0)
\class[red!88!black,name=\rho x_1x_2\ov_2](7,-1)
\class[red!88!black,name=x_1^2x_2\ov_2](8,-2)
\class[red!88!black,name=\rho^2x_3\ov_2](10,-4)
\class[red!88!black,name=\rho x_1x_3\ov_2](11,-5)
\class[red!88!black,name=x_1^2x_3\ov_2](12,-6)
\class[red!88!black](12,-6)
\class[red!88!black,name=\rho x_2x_3\ov_2](13,-7)
\class[red!88!black,name=x_1x_2x_3\ov_2](14,-8)
\class[red!88!black](14,-8)
\class[red!88!black,name=\zeta_1^8\rho x_1x_2\ov_2](15,-9)
\class[red!88!black,name=\zeta_1^8x_1^2x_2\ov_2](16,-10)
\class[red!88!black](16,-10)
\class[name=\zeta_4](15,-15)
\class(48,48)
\class(60,-60)
\structline[green!50!black](0,0)(48,48)
\structline[green!50!black](0,0)(60,-60)
\end{sseqdata}
\begin{figure}[!htbp]
\centering
\printpage[name = Einftybp2]
\caption{The $E_\infty$ page of the HSSS for $BPGL\langle 2\rangle$.}
\label{einftybp2}
\end{figure}

\subsection{The homology of $\Gamma(BPGL\langle 3\rangle)$}\label{subsecn=3computation}
We follow arguments similar to the $m=2$ case to compute the HSSS for $BPGL\langle 3\rangle$. The $E_{15}$ page is displayed in Figure \ref{e15bp3}, with conventions as before, where $\ov_3$ classes are represented by blue dots, and blue structure lines indicate multiplication by $\rho^7\ov_3$, which detects $\sigma\in\pi_7 S^0$.

\begin{theorem}\label{thmn=3SS}
There is a spectral sequence of \(\cA_{\ast}\)-comodule algebras with
\(E_{2}\)-term given by
\[E_2^{*,*,*}=\F_2[\zeta_1^2,\zeta_2,\ldots][\rho,x_1,\ov_1,\ov_2,\ov_3]\]
We have differentials 
\begin{align*}
d_3(\zeta_1^2)&=\rho\ov_1
&d_3(\zeta_2)&=x_1\ov_1
&d_7(\zeta_1^4)&=\rho^3\ov_2\\
d_7(\zeta_2^2)&=\rho x_1^2\ov_2
&d_7(\zeta_3)&=(x_1^3+\rho^2x_2)\ov_2
&d_{15}(\zeta_1^8)&=\rho^7\ov_3\\
d_{15}(\zeta_2^4)&=\rho^3 x_1^4\ov_3
&d_{15}(\zeta_3^2)&=(\rho x_1^6+\rho^5x_2^2)\ov_3
&d_{15}(\zeta_4)&=(x_1^7+\rho^2x_1^4x_2+\rho^4x_1x_2^2+\rho^6x_3)\ov_3
\end{align*}
The subalgebra of $E_2$ generated by the classes
\[\rho,x_1,x_2,x_3,x_4,\ov_1,\ov_2,\ov_3,\zeta_1^{16},\zeta_2^{8},\zeta_3^4,\zeta_4^2,\zeta_5,\ldots\]
consists of permanent cycles.
\end{theorem}

\begin{sseqdata}[ name = E15bp3, Adams grading, classes = {fill, show name=below},
grid = go, xrange ={0}{19},yrange={-19}{19},xscale=0.5,yscale=0.45,x tick step =2, y tick step =2,run off differentials = {->},struct lines = blue!85!black ]
\class[name = 1](0,0)
\class[red!88!black](4,2)
\class[red!88!black](6,0)
\class[red!88!black](7,-1)
\class[red!88!black](8,-2)
\class[red!88!black](10,-4)
\class[red!88!black](11,-5)
\class[red!88!black](12,-6)
\class[red!88!black](12,-6)
\class[red!88!black](13,-7)
\class[red!88!black](14,-8)
\class[red!88!black](14,-8)
\class[red!88!black](15,-9)
\class[red!88!black](16,-10)
\class[red!88!black](16,-10)
\class[red!88!black](18,-12)
\class[red!88!black](18,-12)
\class[red!88!black](18,-12)
\class[red!88!black](19,-13)
\class[red!88!black](19,-13)
\class[red!88!black](19,-13)
\class[blue!85!black,name = \rho^7\ov_3](7,7)
\DoUntilOutOfBoundsThenNMore{3}{
    \class[blue!85!black](\lastx+7,\lasty+7)
    \structline
}
\structline(0,0)(7,7)
\class[blue!85!black](8,6)
\DoUntilOutOfBoundsThenNMore{3}{
    \class[blue!85!black](\lastx+7,\lasty+7)
    \structline
}
\class[blue!85!black](9,5)
\DoUntilOutOfBoundsThenNMore{3}{
    \class[blue!85!black](\lastx+7,\lasty+7)
    \structline
}
\class[blue!85!black](10,4)
\DoUntilOutOfBoundsThenNMore{3}{
    \class[blue!85!black](\lastx+7,\lasty+7)
    \structline
}
\class[blue!85!black](10,4)
\DoUntilOutOfBoundsThenNMore{3}{
    \class[blue!85!black](\lastx+7,\lasty+7)
    \structline
}
\class[blue!85!black](11,3)
\DoUntilOutOfBoundsThenNMore{3}{
    \class[blue!85!black](\lastx+7,\lasty+7)
    \structline
}
\class[blue!85!black,name=\rho^3x_1^4\ov_3](11,3)
\DoUntilOutOfBoundsThenNMore{3}{
    \class[blue!85!black](\lastx+7,\lasty+7)
    \structline
}
\class[blue!85!black](12,2)
\DoUntilOutOfBoundsThenNMore{3}{
    \class[blue!85!black](\lastx+7,\lasty+7)
    \structline
}
\class[blue!85!black](12,2)
\DoUntilOutOfBoundsThenNMore{3}{
    \class[blue!85!black](\lastx+7,\lasty+7)
    \structline
}
\class[blue!85!black,name=\rho x_1^6\ov_3](13,1)
\DoUntilOutOfBoundsThenNMore{3}{
    \class[blue!85!black](\lastx+7,\lasty+7)
    \structline
}
\class[blue!85!black](13,1)
\DoUntilOutOfBoundsThenNMore{3}{
    \class[blue!85!black](\lastx+7,\lasty+7)
    \structline
}
\class[blue!85!black](13,1)
\DoUntilOutOfBoundsThenNMore{3}{
    \class[blue!85!black](\lastx+7,\lasty+7)
    \structline
}
\class[blue!85!black,name=x_1^7\ov_3](14,0)
\DoUntilOutOfBoundsThenNMore{3}{
    \class[blue!85!black](\lastx+7,\lasty+7)
    \structline
}
\class[blue!85!black](14,0)
\DoUntilOutOfBoundsThenNMore{3}{
    \class[blue!85!black](\lastx+7,\lasty+7)
    \structline
}
\class[blue!85!black](14,0)
\DoUntilOutOfBoundsThenNMore{3}{
    \class[blue!85!black](\lastx+7,\lasty+7)
    \structline
}
\class[blue!85!black](14,0)
\DoUntilOutOfBoundsThenNMore{3}{
    \class[blue!85!black](\lastx+7,\lasty+7)
    \structline
}
\class[name=\zeta_1^8](8,-8)
\DoUntilOutOfBoundsThenNMore{3}{
\d15
    \class[blue!85!black](\lastx+7,\lasty+7)
    \structline
}
\class[blue!85!black](15,-1)
\DoUntilOutOfBoundsThenNMore{3}{
    \class[blue!85!black](\lastx+7,\lasty+7)
    \structline
}
\class[blue!85!black](15,-1)
\DoUntilOutOfBoundsThenNMore{3}{
    \class[blue!85!black](\lastx+7,\lasty+7)
    \structline
}
\class[blue!85!black](15,-1)
\DoUntilOutOfBoundsThenNMore{3}{
    \class[blue!85!black](\lastx+7,\lasty+7)
    \structline
}
\class[blue!85!black](15,-1)
\DoUntilOutOfBoundsThenNMore{3}{
    \class[blue!85!black](\lastx+7,\lasty+7)
    \structline
}
\class[blue!85!black](16,-2)
\DoUntilOutOfBoundsThenNMore{3}{\d15
    \class[blue!85!black](\lastx+7,\lasty+7)
    \structline
}
\class[blue!85!black](16,-2)
\DoUntilOutOfBoundsThenNMore{3}{
    \class[blue!85!black](\lastx+7,\lasty+7)
    \structline
}
\class[blue!85!black](16,-2)
\DoUntilOutOfBoundsThenNMore{3}{
    \class[blue!85!black](\lastx+7,\lasty+7)
    \structline
}
\class[blue!85!black](16,-2)
\DoUntilOutOfBoundsThenNMore{3}{
    \class[blue!85!black](\lastx+7,\lasty+7)
    \structline
}
\class[blue!85!black](16,-2)
\DoUntilOutOfBoundsThenNMore{3}{
    \class[blue!85!black](\lastx+7,\lasty+7)
    \structline
}
\class[blue!85!black](23,5)
\DoUntilOutOfBoundsThenNMore{3}{
    \class[blue!85!black](\lastx+7,\lasty+7)
    \structline
}
\class[blue!85!black](17,-3)
\DoUntilOutOfBoundsThenNMore{3}{
    \class[blue!85!black](\lastx+7,\lasty+7)
    \structline
}
\class[blue!85!black](17,-3)
\DoUntilOutOfBoundsThenNMore{3}{\d15
    \class[blue!85!black](\lastx+7,\lasty+7)
    \structline
}
\class[blue!85!black](17,-3)
\DoUntilOutOfBoundsThenNMore{3}{
    \class[blue!85!black](\lastx+7,\lasty+7)
    \structline
}
\class[blue!85!black](17,-3)
\DoUntilOutOfBoundsThenNMore{3}{
    \class[blue!85!black](\lastx+7,\lasty+7)
    \structline
}
\class[blue!85!black](17,-3)
\DoUntilOutOfBoundsThenNMore{3}{
    \class[blue!85!black](\lastx+7,\lasty+7)
    \structline
}
\class[blue!85!black](24,4)
\DoUntilOutOfBoundsThenNMore{3}{
    \class[blue!85!black](\lastx+7,\lasty+7)
    \structline
}
\class[blue!85!black](18,-4)
\DoUntilOutOfBoundsThenNMore{3}{\d15
    \class[blue!85!black](\lastx+7,\lasty+7)
    \structline
}
\class[blue!85!black](18,-4)
\DoUntilOutOfBoundsThenNMore{3}{
    \class[blue!85!black](\lastx+7,\lasty+7)
    \structline
}
\class[blue!85!black](18,-4)
\DoUntilOutOfBoundsThenNMore{3}{
    \class[blue!85!black](\lastx+7,\lasty+7)
    \structline
}
\class[blue!85!black](18,-4)
\DoUntilOutOfBoundsThenNMore{3}{\d15(\lastx,\lasty,4,2)
    \class[blue!85!black](\lastx+7,\lasty+7)
    \structline
}
\class[blue!85!black](18,-4)
\DoUntilOutOfBoundsThenNMore{3}{
    \class[blue!85!black](\lastx+7,\lasty+7)
    \structline
}
\class[blue!85!black](18,-4)
\DoUntilOutOfBoundsThenNMore{3}{
    \class[blue!85!black](\lastx+7,\lasty+7)
    \structline
}
\class[name=\zeta_2^4](12,-12)
\DoUntilOutOfBoundsThenNMore{3}{
\d15
    \class[blue!85!black](\lastx+7,\lasty+7)
    \structline
}
\class[blue!85!black](19,-5)
\DoUntilOutOfBoundsThenNMore{3}{
    \class[blue!85!black](\lastx+7,\lasty+7)
    \structline
}
\class[blue!85!black](19,-5)
\DoUntilOutOfBoundsThenNMore{3}{
    \class[blue!85!black](\lastx+7,\lasty+7)
    \structline
}
\class[blue!85!black](19,-5)
\DoUntilOutOfBoundsThenNMore{3}{\d15(\lastx,\lasty,4,2)
    \class[blue!85!black](\lastx+7,\lasty+7)
    \structline
}
\class[blue!85!black](19,-5)
\DoUntilOutOfBoundsThenNMore{3}{
    \class[blue!85!black](\lastx+7,\lasty+7)
    \structline
}
\class[blue!85!black](19,-5)
\DoUntilOutOfBoundsThenNMore{3}{
    \class[blue!85!black](\lastx+7,\lasty+7)
    \structline
}
\class[name=\zeta_3^2](14,-14)
\DoUntilOutOfBoundsThenNMore{3}{
\d15
    \class[blue!85!black](\lastx+7,\lasty+7)
    \structline
}
\class[name=\zeta_4](15,-15)
\DoUntilOutOfBoundsThenNMore{3}{
\d15
    \class[blue!85!black](\lastx+7,\lasty+7)
    \structline
}
\class[name=\zeta_1^{16}](16,-16)
\DoUntilOutOfBoundsThenNMore{3}{
\d15
    \class[blue!85!black](\lastx+7,\lasty+7)
    \structline
}
\class(20,-20)
\DoUntilOutOfBoundsThenNMore{3}{
\d15
    \class[blue!85!black](\lastx+7,\lasty+7)
    \structline
}
\class[blue!85!black](20,-6)
\DoUntilOutOfBoundsThenNMore{3}{\d15
    \class[blue!85!black](\lastx+7,\lasty+7)
    \structline
}
\class[blue!85!black](20,-6)
\DoUntilOutOfBoundsThenNMore{3}{\d15(\lastx,\lasty,2,2)
    \class[blue!85!black](\lastx+7,\lasty+7)
    \structline
}
\class(48,48)
\class(60,-60)
\structline[green!50!black](0,0)(60,-60)
\end{sseqdata}
\begin{figure}
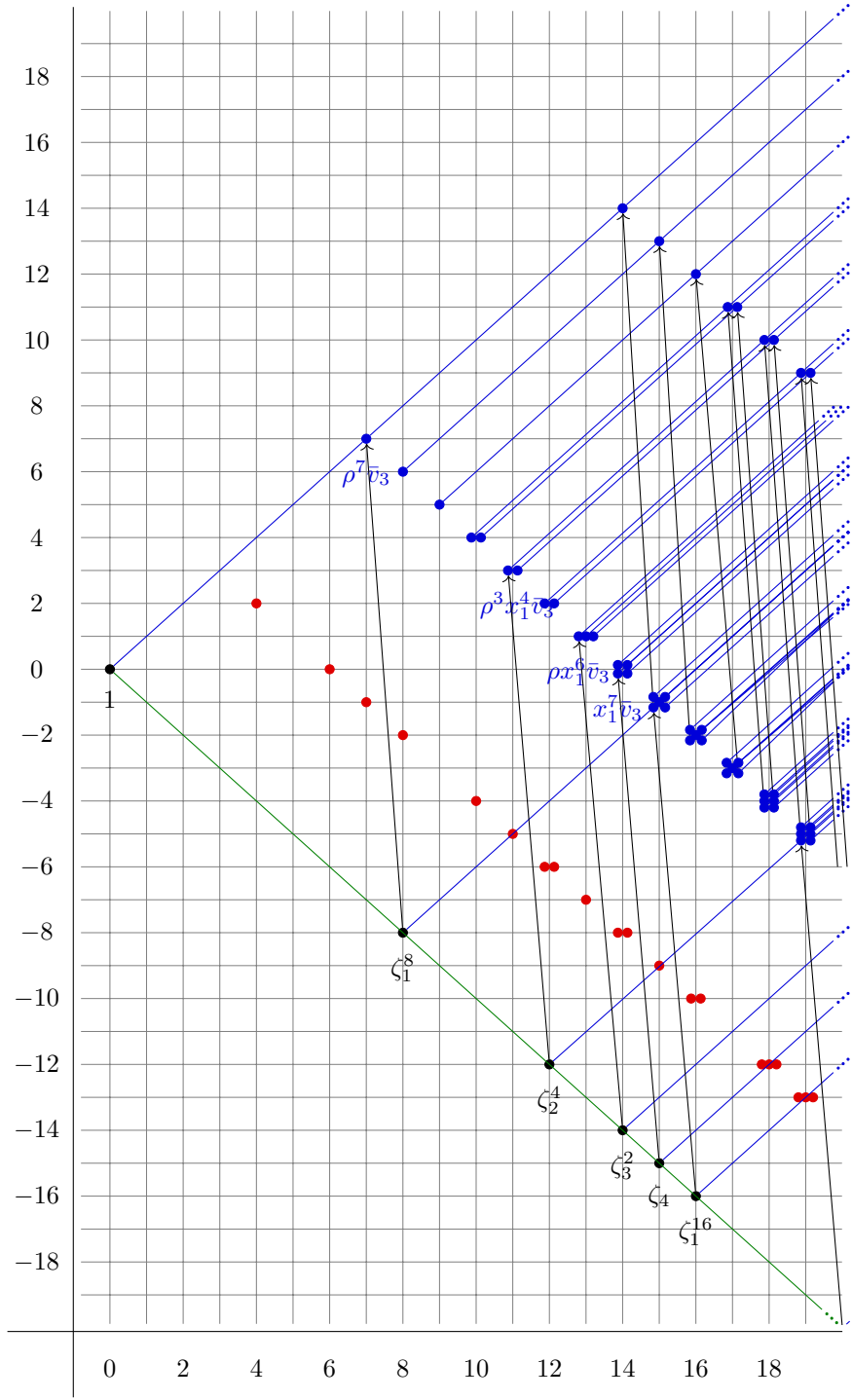

\centering
\printpage[name = E15bp3]
\caption{The $E_{15}$ page of the HSSS for $BPGL\langle 3\rangle$.}
\label{e15bp3}
\end{figure}

The formulas for $d_{\le 7}$ imply that, as a dga, one has an isomorphism
\[E_{15}^{*,*,*}(BPGL\langle 3\rangle;i_*H\F_2)\cong E_8^{*,*,*}(BPGL\langle 2\rangle;i_*H\F_2)[\ov_3]\]
The ideal $(\overline{v}_{1},\overline{v}_{2})\subset E_{15}$ is contained in $\ker(d_{15})$, and we have maps of dga's
\[E_{15}\to E_{15}/(\overline{v}_{1},\overline{v}_{2}^2)\to E_{15}/(\overline{v}_{1},\overline{v}_{2})\]
There is nothing in weight 0 contributed by $\overline{v}_{1}$ because of the $d_3$'s, and there is nothing in weight 0 contributed by $\overline{v}_{2}^2$ because of the $d_7$'s. In fact, the first map is an isomorphism in non-positive weights, and the second map is surjective with square zero kernel consisting of $d_{15}$-cycles (for degree reasons). Using this, it is not difficult to recover the $E_{16}$ page from $H_*(E_{15}/(\overline{v}_{1},\overline{v}_{2}))$, and we have an isomorphism
\[E_{15}/(\overline{v}_{1},\overline{v}_{2})\cong \frac{\F_2[\zeta_1^8,\zeta_2^4,\zeta_3^2,\zeta_4,\ldots][\rho,x_1,x_2,x_3,\overline{v}_{3}]}{(x_2^4+\zeta_1^8x_1^4+\zeta_2^4\rho^4,x_3^2+\zeta_1^8x_2^2+\zeta_2^4x_1^2+\zeta_3^2\rho^2)}\]
with $d_{15}$ determined by the four explicit differentials given above. 

We claim that $\ker(d_{15})$ is the subalgebra of $E_{15}/(\ov_1,\ov_2)$ generated by the classes
\[\zeta_1^{16},\zeta_2^8,\zeta_3^4,\zeta_4^2,\zeta_5,\ldots,\rho,x_1,x_2,x_3,x_4,\ov_3\]
It follows that we have the presentation
\[\ker(d_{15})=\frac{\F_2[\zeta_1^{16},\zeta_2^8,\zeta_3^4,\zeta_4^2,\zeta_5,\ldots][\rho,x_1,x_2,x_3,x_4,\ov_3]}{(x_2^8+\rho^8\zeta_2^8+x_1^8\zeta_1^{16},x_3^4+\rho^4\zeta_3^4+x_1^4\zeta_2^8+x_2^4\zeta_1^{16},x_4^2+\rho^2\zeta_4^2+x_1^2\zeta_3^4+x_2^2\zeta_2^8+x_3^2\zeta_1^{16})}\]

Writing our differentials in terms of this description of the kernel, we have
{
\allowdisplaybreaks
\aln{
d_{15}(\zeta_1^8)&=\rho^7\ov_3\\
d_{15}(\zeta_2^4)&=\rho^3 x_1^4\ov_3\\
d_{15}(\zeta_3^2)&=(\rho x_1^6+\rho^5 x_2^2)\ov_3\\
d_{15}(\zeta_4)&=(x_1^7+\rho^2x_1^4x_2+\rho^4x_1x_2^2+\rho^6x_3)\ov_3\\
d_{15}(\zeta_1^8\zeta_2^4)&=\rho^3x_2^4\ov_3\\
d_{15}(\zeta_1^8\zeta_3^2)&=(\rho^5x_3^2+\rho x_1^2x_2^4)\ov_3\\
d_{15}(\zeta_1^8\zeta_4)&=(x_1^3x_2^4+\rho^2x_2^5+\rho^4x_1x_3^2+\rho^6x_4)\ov_3\\
d_{15}(\zeta_2^4\zeta_3^3)&=(\rho x_1^4x_3^2+\rho x_2^6)\ov_3\\
d_{15}(\zeta_2^4\zeta_4)&=(x_1^5x_3^2+x_1x_2^6+\rho^2x_1^4x_4+\rho^2x_2^4x_3)\ov_3\\
d_{15}(\zeta_3^2\zeta_4)&=(x_2^7+x_1^6x_4+x_1^4x_2x_3^2+x_1^2x_2^4x_3+\rho^4x_2^2x_4+\rho^4x_3^3)\ov_3\\
d_{15}(\zeta_1^8\zeta_2^4\zeta_3^2)&=(\rho x_2^4x_3^2+\rho x_1^4x_2^2\zeta_1^{16}+\rho^5x_1^2\zeta_2^8)\ov_3\\
d_{15}(\zeta_1^8\zeta_2^4\zeta_4)&=(x_1x_2^4x_3^2+x_1^5x_2^2\zeta_1^{16}+\rho^2x_2^4x_4+\rho^2x_1^4x_3\zeta_1^{16}+\rho^4x_1^3\zeta_2^8+\rho^6x_2\zeta_2^8)\ov_3\\
d_{15}(\zeta_1^8\zeta_3^2\zeta_4)&=(x_1^2x_2^4x_4+x_2^5x_3^2+x_1^4x_2^3\zeta_1^{16}+\rho^4x_3^3x_4+\rho^4x_2^2x_3\zeta_1^{16}+\rho^4x_1^2x_2\zeta_2^8+\rho^6x_1\zeta_3^4)\ov_3\\
d_{15}(\zeta_2^4\zeta_3^2\zeta_4)&=(x_1^4x_3^2x_4+x_2^6x_4+x_2^4x_3^3+x_1^4x_2^2x_3\zeta_1^{16}+(x_1^6x_2+\rho^4x_1^2x_3+\rho^4x_2^3)\zeta_2^8+\rho^2x_1^5\zeta_3^4)\ov_3\\
d_{15}(\zeta_1^8\zeta_2^4\zeta_3^2\zeta_4)&=(x_2^4x_3^2x_4+x_1^2x_2^5\zeta_2^8+x_1^4x_2^2x_4\zeta_1^{16}+x_1^4x_3^3\zeta_1^{16}+x_2^6x_3\zeta_1^{16}+\rho^2x_1x_2^4\zeta_3^4\\
&+\rho^4x_1^2x_4\zeta_2^8+\rho^4x_2x_3^2\zeta_2^8)\ov_3
}
}
Letting $I_3$ be the ideal in $\ker(d_{15})$ generated by this list of relations, it follows that $H_*(E_{15}/(\ov_1,\ov_2))=\ker(d_{15})/I_3$. It remains to verify that the claimed list of elements indeed generate $\ker(d_{15})$. We achieve this by running the $\rho$-Bockstein spectral sequence as before, and our arguments follow closely those of the $m=2$ case.

\begin{proposition}\label{propn=3list}
    The $E_0$-page of the $\rho$-Bockstein spectral sequence is given by
    \[\frac{\F_2[\zeta_1^8,\zeta_2^4,\zeta_3^2,\zeta_4,\ldots][\rho,x_1,x_2,x_3,\ov_3]}{(x_2^4+\zeta_1^8x_1^4,x_3^2+\zeta_1^8x_2^2+\zeta_2^4x_1^2)}\]
We have differentials
\begin{align*}
\delta_0(\zeta_4)&=x_1^7\ov_3
&\delta_1(\zeta_3^2)&=\rho x_1^6\ov_3
&\delta_3(\zeta_2^4)&=\rho^3 x_1^4\ov_3
&\delta_7(\zeta_1^8)&=\rho^7\ov_3
\end{align*}
\end{proposition}

The annihilator ideal of $\delta_0(\zeta_4)$ in $E_0$ is trivial, so $\delta_0$ does not create any Massey product cycles. We conclude the following.

\begin{proposition}
    The $E_1$-page of the $\rho$-Bockstein spectral sequence is given by
    \[\frac{\F_2[\zeta_1^8,\zeta_2^4,\zeta_3^2,\zeta_4^2,\ldots][\rho,x_1,x_2,x_3,\ov_3]}{(x_2^4+\zeta_1^8x_1^4,x_3^2+\zeta_1^8x_2^2+\zeta_2^4x_1^2,x_1^7\overline{v}_{3})}\]
\end{proposition}

The annihilator ideal of $\delta_1(\zeta_3^2)$ in $E_1$ is $(x_1)$, which gives our Massey product $x_4$.

\begin{proposition}
The $E_3$-page of the $\rho$-Bockstein spectral sequence is given by the quotient of the polynomial ring
    \[\F_2[\zeta_1^8,\zeta_2^4,\zeta_3^4,\zeta_4^2,\zeta_5,\ldots][\rho,x_1,x_2,x_3,x_4,\ov_3]\]
    by the ideal generated by the elements
    \begin{align*}
        &x_2^4+\zeta_1^8x_1^4,
        &&x_3^2+\zeta_1^8x_2^2+\zeta_2^4x_1^2,\\
        &x_4^2+x_1^2\zeta_3^4+x_2^2\zeta_2^8+x_3^2\zeta_1^{16},
        &&x_1^7\ov_3,\\
        &(x_2^7+x_1^6x_4+x_1^4x_2x_3^2+x_1^2x_2^4x_3)\ov_3,
        &&\rho x_1^6\ov_3.
    \end{align*}
\end{proposition}
\begin{proof}
    The description of $\mathrm{ann}_{E_1}(\rho x_1^6\overline{v}_{3})$ implies that the map
    \[S:=\F_2[\zeta_1^8,\zeta_2^4,\zeta_3^4,\zeta_4^2,\zeta_5,\ldots][\rho,x_1,x_2,x_3,x_4,\ov_3]/J\to \ker(\delta_1)\]
    is surjective, where $J$ is the ideal generated by the list of elements in the statement of the proposition, with $\rho x_1^6\ov_3$ removed. To see that it is injective, note that the map
    \[S\to S[y]/(y^2+\zeta_3^4,x_1y+x_2\zeta_2^4+x_3\zeta_1^8+x_4)\]
    is injective, and the latter is easily identified with $E_1$ by setting $y=\zeta_3^2$.
\end{proof}

\begin{proposition}
The $E_7$-page of the $\rho$-Bockstein spectral sequence is given by the quotient of the polynomial ring
    \[\F_2[\zeta_1^8,\zeta_2^8,\zeta_3^4,\zeta_4^2,\zeta_5,\ldots][\rho,x_1,x_2,x_3,x_4,\ov_3]\]
    by the ideal generated by the elements
    \begin{align*}
        &x_2^4+\zeta_1^8x_1^4,
        &&x_3^4+\zeta_1^{16}x_2^4+\zeta_2^8x_1^4,\\
        &x_4^2+x_1^2\zeta_3^4+x_2^2\zeta_2^8+x_3^2\zeta_1^{16},
        &&x_1^7\ov_3,\\
        &(x_2^7+x_1^6x_4+x_1^4x_2x_3^2+x_1^2x_2^4x_3)\ov_3,
        &&\rho x_1^6\ov_3,\\
        &\rho^3 x_1^4\ov_3,
&&(x_1^5x_3^2+x_1x_2^6)\ov_3,\\
&(\rho x_1^4x_3^2+\rho x_2^6)\ov_3,
&&(x_1^4x_3^2x_4+x_2^6x_4+x_2^4x_3^3+x_1^4x_2^2x_3\zeta_1^{16}+x_1^6x_2\zeta_2^8)\ov_3.
    \end{align*}
\end{proposition}
\begin{proof}
    We claim that $\mathrm{ann}_{E_3}(\rho^3 x_1^4\ov_3)=(x_1^2)$. To see this, let
\[R=\frac{\F_2[\zeta_1^8,\zeta_2^4,\zeta_3^4,\zeta_4^2,\ldots][\rho,x_1,x_2,x_3,x_4,\ov_3]}{(x_2^4+\zeta_1^8x_1^4,x_3^2+\zeta_1^8x_2^2+\zeta_2^4x_1^2,x_4^2+x_1^2\zeta_3^4+x_2^2\zeta_2^8+x_3^2\zeta_1^{16})}\]
There is a surjection
\[\pi:R\to R/(x_1^7\overline{v}_{3},\rho x_1^6\ov_3,(x_2^7+x_1^6x_4+x_1^4x_2x_3^2+x_1^2x_2^4x_3)\ov_3)\cong E_3\]
So if $\pi(r)\cdot \rho^3 x_1^4\ov_3=0$, we have
\[r\rho^3 x_1^4\ov_3\in (x_1^7\overline{v}_{3},\rho x_1^6\ov_3,(x_2^7+x_1^6x_4+x_1^4x_2x_3^2+x_1^2x_2^4x_3)\ov_3)\subset (x_1^6)\]
Using that $R=T[\rho,\ov_3]$ for $T$ an $x_1$-torsion free ring, it follows that $rx_1^4\in (x_1^6)$ and therefore $r\in (x_1^2)$. Since $x_1^2\zeta_2^4=x_3^2+\zeta_1^8x_2^2$, this creates no new cycles.

It follows that the map
\[S:=\F_2[\zeta_1^8,\zeta_2^8,\zeta_3^4,\zeta_4^2,\zeta_5,\ldots][\rho,x_1,x_2,x_3,x_4,\ov_3]/J\to \ker(\delta_3)\]
is surjective, where $J$ is the ideal generated by the list of elements in the statement of the proposition, with $\rho^3 x_1^4\ov_3$ removed. To see that it is injective, note that the map
    \[S\to S[y]/(y^2+\zeta_2^8,x_1^2y+x_2^2\zeta_1^8+x_3^2)\]
    is injective, and the latter is easily identified with $E_3$ by setting $y=\zeta_2^4$.
\end{proof}

\begin{proposition}
The $E_8$-page of the $\rho$-Bockstein spectral sequence is given by the quotient of the polynomial ring
    \[\F_2[\zeta_1^{16},\zeta_2^8,\zeta_3^4,\zeta_4^2,\zeta_5,\ldots][\rho,x_1,x_2,x_3,x_4,\ov_3]\]
    by the ideal generated by the elements
    \begin{align*}
        &x_2^8+\zeta_1^{16}x_1^8,
        &&x_3^4+\zeta_1^{16}x_2^4+\zeta_2^8x_1^4,\\
        &x_4^2+x_1^2\zeta_3^4+x_2^2\zeta_2^8+x_3^2\zeta_1^{16},
        &&x_1^7\ov_3,\\
        &(x_2^7+x_1^6x_4+x_1^4x_2x_3^2+x_1^2x_2^4x_3)\ov_3,
        &&\rho x_1^6\ov_3,\\
        &\rho^3 x_1^4\ov_3,
&&(x_1^5x_3^2+x_1x_2^6)\ov_3,\\
&(\rho x_1^4x_3^2+\rho x_2^6)\ov_3,
&&(x_1^4x_3^2x_4+x_2^6x_4+x_2^4x_3^3+x_1^4x_2^2x_3\zeta_1^{16}+x_1^6x_2\zeta_2^8)\ov_3,\\
&\rho^7\ov_3,
&&\rho^3x_2^4\ov_3,\\
&(\rho^5x_3^2+\rho x_1^2x_2^4)\ov_3,
&&(x_1^3x_2^4+\rho^2x_2^5+\rho^4x_1x_3^2+\rho^6x_4)\ov_3,\\
&(\rho x_2^4x_3^2+\rho x_1^4x_2^2\zeta_1^{16}+\rho^5x_1^2\zeta_2^8)\ov_3.
\end{align*}
\begin{align*}
&(x_1x_2^4x_3^2+x_1^5x_2^2\zeta_1^{16}+\rho^2x_2^4x_4+\rho^2x_1^4x_3\zeta_1^{16}+\rho^4x_1^3\zeta_2^8+\rho^6x_2\zeta_2^8)\ov_3\\
&(x_1^2x_2^4x_4+x_2^5x_3^2+x_1^4x_2^3\zeta_1^{16}+\rho^4x_3^3x_4+\rho^4x_2^2x_3\zeta_1^{16}+\rho^4x_1^2x_2\zeta_2^8+\rho^6x_1\zeta_3^4)\ov_3\\
&(x_2^4x_3^2x_4+x_1^2x_2^5\zeta_2^8+x_1^4x_2^2x_4\zeta_1^{16}+x_1^4x_3^3\zeta_1^{16}+x_2^6x_3\zeta_1^{16}+\rho^2x_1x_2^4\zeta_3^4+\rho^4x_1^2x_4\zeta_2^8+\rho^4x_2x_3^2\zeta_2^8)\ov_3
    \end{align*}
\end{proposition}
\begin{proof}
    We claim that $\mathrm{ann}_{E_7}(\rho^7\ov_3)=(x_1^4)$. To see this, let
\[R=\frac{\F_2[\zeta_1^8,\zeta_2^8,\zeta_3^4,\zeta_4^2,\ldots][\rho,x_1,x_2,x_3,x_4,\ov_3]}{(x_2^4+\zeta_1^8x_1^4,x_3^4+\zeta_1^{16}x_2^4+\zeta_2^8x_1^4,x_4^2+x_1^2\zeta_3^4+x_2^2\zeta_2^8+x_3^2\zeta_1^{16})}\]
There is a surjection $\pi:R\to R/I\cong E_7$ where $I$ is the ideal generated by the elements 
\begin{align*}
        &x_1^7\ov_3,
     &&(x_2^7+x_1^6x_4+x_1^4x_2x_3^2+x_1^2x_2^4x_3)\ov_3,\\
        &\rho x_1^6\ov_3,
        &&\rho^3 x_1^4\ov_3,\\
&(x_1^5x_3^2+x_1x_2^6)\ov_3,
&&(\rho x_1^4x_3^2+\rho x_2^6)\ov_3,\\
&(x_1^4x_3^2x_4+x_2^6x_4+x_2^4x_3^3+x_1^4x_2^2x_3\zeta_1^{16}+x_1^6x_2\zeta_2^8)\ov_3.
    \end{align*}
    So if $\pi(r)\cdot \rho^7\ov_3=0\in E_7$, then $r\rho^7\ov_3\in I$. It is not hard to check that $I\subset (x_1^4)$, so since $R$ is of the form $T[\rho,\ov_3]$ for $T$ an $x_1$-torsion free ring we have $r\in (x_1^4)$. Arguing as before, we arrive at the claimed presentation.
\end{proof}

The remaining algebra generators are all permanent cycles, and so the $\rho$-Bockstein spectral sequence collapses on $E_8$. This verifies the discussion preceding Proposition \ref{propn=3list}, so that $H_*(E_{15}/(\ov_1,\ov_2))$ has the claimed description as an $\mathcal{A}_*$-comodule algebra. In weight zero, we have
\[H_*(E_{15}/(\ov_1,\ov_2))\cong \mathcal{A}_*\square_{\mathcal{A}(3)_*}M_3\]
and machine computation gives the following basis of the $\mathcal{A}(3)_*$-comodule algebra $M_3$, which has dimension 165.
{
\allowdisplaybreaks
\begin{align*}
1&\in H_0
&\rho^6x_1 \ov_3 &\in H_8
&\rho^5x_1^2 \ov_3 &\in H_9\\ 
\rho^6x_2 \ov_3 &\in H_{10}
&\rho^4x_1^3 \ov_3 &\in H_{10}
&\rho^5x_1x_2 \ov_3 &\in H_{11}\\
\rho^4x_1^2x_2 \ov_3 &\in H_{12}
&\rho^2x_1^5 \ov_3 &\in H_{12}
&\rho^5x_2^2 \ov_3 &\in H_{13}\\
\rho^3x_1^3x_2 \ov_3 &\in H_{13} 
&\rho^6x_3 \ov_3 &\in H_{14}
&\rho^4x_1x_2^2 \ov_3 &\in H_{14}\\
\rho^2x_1^4x_2 \ov_3 &\in H_{14}
&\rho^5x_1x_3 \ov_3 &\in H_{15}
&\rho^3x_1^2x_2^2 \ov_3 &\in H_{15}\\
\rho x_1^5x_2 \ov_3 &\in H_{15}
&\rho^4x_1^2x_3 \ov_3 &\in H_{16}
&\rho^4x_2^3 \ov_3 &\in H_{16}\\
\rho^2x_1^3x_2^2 \ov_3 &\in H_{16}
&x_1^6x_2 \ov_3 &\in H_{16}
&\rho^5x_2x_3 \ov_3 &\in H_{17}\\
\rho^3x_1^3x_3 \ov_3 &\in H_{17}
&\rho^3x_1x_2^3 \ov_3 &\in H_{17}
&\rho x_1^4x_2^2 \ov_3 &\in H_{17}\\ 
\rho^4x_1x_2x_3 \ov_3 &\in H_{18}
&\rho^2x_1^4x_3 \ov_3 &\in H_{18}
&\rho^2x_1^2x_2^3 \ov_3 &\in H_{18}\\
x_1^5x_2^2 \ov_3 &\in H_{18}
&\rho^3x_1^2x_2x_3 \ov_3 &\in H_{19}
&\rho x_1^5x_3 \ov_3 &\in H_{19}\\
\rho x_1^3x_2^3 \ov_3 &\in H_{19}
&\rho^4x_2^2x_3 \ov_3 &\in H_{20}
&\rho^2x_1^3x_2x_3 \ov_3 &\in H_{20}\\
\rho^2x_1x_2^4 \ov_3 &\in H_{20}
&x_1^6x_3 \ov_3 &\in H_{20}
&x_1^4x_2^3 \ov_3 &\in H_{20}\\
\rho^5x_3^2 \ov_3 &\in H_{21}
&\rho^3x_1x_2^2x_3 \ov_3 &\in H_{21}
&\rho x_1^4x_2x_3 \ov_3 &\in H_{21}\\
\rho^6x_4 \ov_3 &\in H_{22}
&\rho^4x_1x_3^2 \ov_3 &\in H_{22}
&\rho^2x_1^2x_2^2x_3 \ov_3 &\in H_{22}\\
\rho^2x_2^5 \ov_3 &\in H_{22}
&x_1^5x_2x_3 \ov_3 &\in H_{22}
&\rho^5x_1x_4 \ov_3 &\in H_{23}\\
\rho^3x_1^2x_3^2 \ov_3 &\in H_{23}
&\rho^3x_2^3x_3 \ov_3 &\in H_{23}
&\rho x_1^3x_2^2x_3 \ov_3 &\in H_{23}\\
\rho x_1x_2^5 \ov_3 &\in H_{23}
&\rho^4x_1^2x_4 \ov_3 &\in H_{24}
&\rho^4x_2x_3^2 \ov_3 &\in H_{24}\\
\rho^2x_1^3x_3^2 \ov_3 &\in H_{24}
&\rho^2x_1x_2^3x_3 \ov_3 &\in H_{24}
&x_1^4x_2^2x_3 \ov_3 &\in H_{24}\\
x_1^2x_2^5 \ov_3 &\in H_{24}
&\rho^5x_2x_4 \ov_3 &\in H_{25}
&\rho^3x_1^3x_4 \ov_3 &\in H_{25}\\
\rho^3x_1x_2x_3^2 \ov_3 &\in H_{25}
&\rho x_1^4x_3^2 \ov_3 &\in H_{25}
&\rho x_1^2x_2^3x_3 \ov_3 &\in H_{25}\\
\rho^4x_1x_2x_4 \ov_3 &\in H_{26}
&\rho^2x_1^4x_4 \ov_3 &\in H_{26}
&\rho^2x_1^2x_2x_3^2 \ov_3 &\in H_{26}\\
\rho^2x_2^4x_3 \ov_3 &\in H_{26}
&x_1^5x_3^2 \ov_3 &\in H_{26}
&x_1^3x_2^3x_3 \ov_3 &\in H_{26}\\ 
\rho^3x_1^2x_2x_4 \ov_3 &\in H_{27}
&\rho^3x_2^2x_3^2 \ov_3 &\in H_{27}
&\rho x_1^5x_4 \ov_3 &\in H_{27}\\
\rho x_1^3x_2x_3^2 \ov_3 &\in H_{27}
&\rho x_1x_2^4x_3 \ov_3 &\in H_{27}
&\rho^4x_3^3 \ov_3 &\in H_{28}\\
\rho^2x_1^3x_2x_4 \ov_3 &\in H_{28}
&\rho^2x_1x_2^2x_3^2 \ov_3 &\in H_{28}
&x_1^6x_4 \ov_3 &\in H_{28}\\
x_1^4x_2x_3^2 \ov_3 &\in H_{28}
&x_1^2x_2^4x_3 \ov_3 &\in H_{28}
&\rho^4x_2^2x_4\overline{v}_3&\in H_{28}\\
\rho^3x_1x_2^2x_4\overline{v}_3&\in H_{29}
&\rho^5x_3x_4 \ov_3 &\in H_{29}
&\rho^3x_1x_3^3 \ov_3 &\in H_{29}\\
\rho x_1^4x_2x_4 \ov_3 &\in H_{29}
&\rho x_1^2x_2^2x_3^2 \ov_3 &\in H_{29}
&\rho x_2^5x_3 \ov_3 &\in H_{29}\\
\rho^2x_1^2x_2^2x_4\overline{v}_3&\in H_{30}
&\rho^4x_1x_3x_4 \ov_3 &\in H_{30}
&\rho^2x_1^2x_3^3 \ov_3 &\in H_{30}\\
\rho^2x_2^3x_3^2 \ov_3 &\in H_{30}
&x_1^5x_2x_4 \ov_3 &\in H_{30}
&x_1^3x_2^2x_3^2 \ov_3 &\in H_{30}\\
x_1x_2^5x_3 \ov_3 &\in H_{30}
&\rho^3x_1^2x_3x_4 \ov_3 &\in H_{31}
&\rho^3x_2x_3^3 \ov_3 &\in H_{31}\\
\rho x_1^3x_2^2x_4 \ov_3 &\in H_{31}
&\rho x_1^3x_3^3 \ov_3 &\in H_{31}
&\rho x_1x_2^3x_3^2 \ov_3 &\in H_{31}\\
\rho^3x_2^3x_4\overline{v}_3&\in H_{31}
&\rho^4x_2x_3x_4 \ov_3 &\in H_{32}
&\rho^2x_1^3x_3x_4 \ov_3 &\in H_{32}\\
\rho^2x_1x_2^3x_4 \ov_3 &\in H_{32}
&\rho^2x_1x_2x_3^3 \ov_3 &\in H_{32}
&x_1^4x_2^2x_4 \ov_3 &\in H_{32}\\
x_1^4x_3^3 \ov_3 &\in H_{32}
&x_1^2x_2^3x_3^2 \ov_3 &\in H_{32}
&x_2^6x_3 \ov_3 &\in H_{32}\\
\rho^3x_1x_2x_3x_4 \ov_3 &\in H_{33}
&\rho x_1^4x_3x_4 \ov_3 &\in H_{33}
&\rho x_1^2x_2^3x_4 \ov_3 &\in H_{33}\\
\rho x_1^2x_2x_3^3 \ov_3 &\in H_{33}
&\rho^2x_1^2x_2x_3x_4 \ov_3 &\in H_{34}
&\rho^2x_2^2x_3^3 \ov_3 &\in H_{34}\\
x_1^5x_3x_4 \ov_3 &\in H_{34}
&x_1^3x_2^3x_4 \ov_3 &\in H_{34}
&x_1^3x_2x_3^3 \ov_3 &\in H_{34}\\
\rho^2x_2^4x_4\overline{v}_3&\in H_{34}
&\rho x_1^3x_2x_3x_4 \ov_3 &\in H_{35}
&\rho x_1x_2^4x_4 \ov_3 &\in H_{35}\\
\rho x_1x_2^2x_3^3 \ov_3 &\in H_{35}
&\rho^3x_2^2x_3x_4\overline{v}_3&\in H_{35}
&\rho^4x_3^2x_4 \ov_3 &\in H_{36}\\
x_1^4x_2x_3x_4 \ov_3 &\in H_{36}
&x_1^2x_2^4x_4 \ov_3 &\in H_{36}
&x_1^2x_2^2x_3^3 \ov_3 &\in H_{36}\\
\rho^2x_1x_2^2x_3x_4\overline{v}_3&\in H_{36}
&\rho^3x_1x_3^2x_4 \ov_3 &\in H_{37}
&\rho x_1^2x_2^2x_3x_4 \ov_3 &\in H_{37}\\
\rho x_2^5x_4 \ov_3 &\in H_{37}
&\rho x_2^3x_3^3 \ov_3 &\in H_{37}
&\rho^2x_1^2x_3^2x_4 \ov_3 &\in H_{38}\\
x_1^3x_2^2x_3x_4 \ov_3 &\in H_{38}
&x_1x_2^5x_4 \ov_3 &\in H_{38}
&x_1x_2^3x_3^3 \ov_3 &\in H_{38}\\
\rho^2x_2^3x_3x_4\overline{v}_3&\in H_{38}
&\rho^3x_2x_3^2x_4 \ov_3 &\in H_{39}
&\rho x_1^3x_3^2x_4 \ov_3 &\in H_{39}\\
\rho x_1x_2^3x_3x_4 \ov_3 &\in H_{39}
&\rho^2x_1x_2x_3^2x_4\ov_3&\in H_{40}
&x_1^4x_3^2x_4\ov_3&\in H_{40}\\
x_1^2x_2^3x_3x_4\ov_3&\in H_{40}
&x_2^6x_4\ov_3&\in H_{40}
&\rho^6x_1^3x_2^3x_3^2\ov_3^2&\in H_{40}\\
\rho x_1^2x_2x_3^2x_4\ov_3&\in H_{41}
&\rho x_2^4x_3x_4\ov_3&\in H_{41}
&x_1^3x_2x_3^2x_4\ov_3&\in H_{42}\\
x_1x_2^4x_3x_4\ov_3&\in H_{42}
&\rho^2x_2^2x_3^2x_4\overline{v}_3&\in H_{42}
&\rho^3x_3^3x_4\ov_3&\in H_{43}\\
\rho x_1x_2^2x_3^2x_4\ov_3&\in H_{43}
&\rho^2x_1x_3^3x_4\ov_3&\in H_{44}
&x_1^2x_2^2x_3^2x_4\ov_3&\in H_{44}\\
x_2^5x_3x_4\ov_3&\in H_{44}
&\rho^6x_1^3x_2^2x_3^3\ov_3^2&\in H_{44}
&\rho x_1^2x_3^3x_4\ov_3&\in H_{45}\\
\rho x_2^3x_3^2x_4\ov_3&\in H_{45}
&\rho^2x_2x_3^3x_4\ov_3&\in H_{46}
&x_1^3x_3^3x_4\ov_3&\in H_{46}\\
x_1x_2^3x_3^2x_4\ov_3&\in H_{46}
&\rho^6x_1^2x_2^3x_3^3\ov_3^2&\in H_{46}
&\rho x_1x_2x_3^3x_4\ov_3&\in H_{47}\\
\rho^5x_1^3x_2^3x_3^3\ov_3^2&\in H_{47}
&x_1^2x_2x_3^3x_4\ov_3&\in H_{48}
&\rho^6x_1^3x_2^3x_3x_4\overline{v}_3^2&\in H_{48}\\
\rho x_2^2x_3^3x_4\ov_3&\in H_{49}
&x_1x_2^2x_3^3x_4\ov_3&\in H_{50}
&x_2^3x_3^3x_4\ov_3&\in H_{52}
\end{align*}
}
Note that multiplication in $M_3$ is not square zero, as there are five generators divisible by $\ov_3^2$. For example, there is the nontrivial product $H_8\otimes H_{32}\to H_{40}$
\[\rho^6x_1\ov_3\cdot x_1^2x_2^3x_3^2\ov_3=\rho^6x_1^3x_2^3x_3^2\ov_3^2\]
However, all products in $M_3$ are determined by the relations in the discussion preceding Proposition \ref{propn=3list}. The coactions follow from Corollary \ref{corcoactionx_m} and the fact that $\rho$ and $\ov_3$ are primitive.

For the following, we refer to Figure \ref{einftybp3}. Here we use rectangles in a given bidegree if it has rank $\ge 3$.

\begin{theorem}\label{einftym=3}
    The HSSS for $BPGL\langle 3\rangle$ collapses on $E_{16}$, and there is an isomorphism of $\mathcal{A}_*$-comodule algebras (in weight zero)

\[E_\infty\cong(\mathcal{A}_*\square_{\mathcal{A}(3)_*}M_3)\oplus(\mathcal{A}_*\square_{\mathcal{A}(2)_*}\overline{M}_{2})\]
where the latter is a square zero extension with $\overline{M}_{3}\cdot\overline{M}_{2}=0$, where $\overline{M}_{i}$ denotes the augmentation ideal in $M_i$.
\end{theorem}
\begin{proof}
    We run the $\ov_2$-Bockstein spectral sequence to calculate $E_{16}$, using that we have an isomorphism in weight zero
\[E_{16}\cong H_*(E_{15}/(\ov_1,\ov_2^2))\]
The description of the $d_{15}$'s implies that, on $E_{15}$
\[\mathrm{image}(d_{15})\cap(\ov_2)=0\]
and so this spectral sequence collapses. All of the  relations in $H_*(E_{15}/(\ov_2))$ hold in $E_{16}$ again by the differentials given in  Theorem \ref{thmn=3SS}, and the claimed description follows for $E_{16}$.

It remains to show that the HSSS collapses on $E_{16}$, and the only classes on $E_{16}$ that are not in the subalgebra of permanent cycles described in Theorem \ref{thmn=3SS} lie in the subspace
\[V\cdot E(\zeta_1^8,\zeta_2^4,\zeta_3^2,\zeta_4)\]
where $V$ is the subspace of $E_{16}$ generated by classes divisible by $\ov_2$, which lie on the line of red dots (the line $y=6-x$) in Figure \ref{einftybp3}. It suffices therefore to show that the subspace
\[
\ov_2\cdot E(\zeta_1^8,\zeta_2^4,\zeta_3^2,\zeta_4)
\]
of $E_2$ consists of permanent cycles. Since this is a spectral sequence of $\mathcal{A}_*$-comodules, the coproduct formulas imply that it suffices to show the class
\[\ov_2\zeta_1^8\zeta_2^4\zeta_3^2\zeta_4\]
is a permanent cycle. It follows for degree reasons that the only possible differential on this class is a $d_{23}$ with target a sum of monomials of the form
\[p(\zeta_i)\rho^a x_1^b\ov_3^2\]
written as an element of $E_2$. It is straightforward to compute the $\ov_3^2$ part of $E_{16}$ in this weight following closely the computation above, from which a machine computation shows there is no class in the required degree on $E_{16}$.
\end{proof}

\begin{remark}
Modulo comodule algebra extension problems, this describes
\[H_*\Gamma(BPGL\langle 3\rangle)\]
Unlike the $m=2$ computation, there is room here for nontrivial comodule algebra extensions. We can resolve these case by case using ad hoc restriction and transfer arguments, but we do not include this, as the size of $M_3$ makes this prohibitively lengthy.

For the same reasons, we do not include a schematic for the comodule $M_3$, although all of its structure may be determined from the coactions on the $x_m$'s described in Corollary \ref{corcoactionx_m}.
\end{remark}

\begin{sseqdata}[ name = Einftybp3, Adams grading, classes = {fill, show name=below},
grid = go, xrange ={0}{50},yrange={-50}{6},xscale=0.2,yscale=0.2,x tick step =2, y tick step =2, font = \tiny ]
\class[name = 1](0,0)
\class[name=\zeta_1^{16}](16,-16)
\class[name=\zeta_2^8](24,-24)
\class[name=\zeta_3^4](28,-28)
\class[name=\zeta_4^2](30,-30)
\class[name=\zeta_5](31,-31)
\class(32,-32)
\class(40,-40)
\class(44,-44)
\class(46,-46)
\class(47,-47)
\class(48,-48)
\class(48,-48)
\class(52,-52)
\class[red!88!black](4,2)
\class[red!88!black](6,0)
\class[red!88!black](7,-1)
\class[red!88!black](8,-2)
\class[red!88!black](10,-4)
\class[red!88!black](11,-5)
\class[red!88!black](12,-6)
\class[red!88!black](12,-6)
\class[red!88!black](13,-7)
\class[red!88!black](14,-8)
\class[red!88!black](14,-8)
\class[red!88!black](15,-9)
\class[red!88!black](16,-10)
\class[red!88!black](16,-10)
\class[red!88!black,rectangle](18,-12)
\DoUntilOutOfBounds{
\class[red!88!black,rectangle](\lastx+1,\lasty-1)
}
\class[blue!85!black](8,6)
\class[blue!85!black](9,5)
\class[blue!85!black](10,4)
\class[blue!85!black](10,4)
\class[blue!85!black](11,3)
\class[blue!85!black](12,2)
\class[blue!85!black](12,2)
\class[blue!85!black](13,1)
\class[blue!85!black](13,1)
\class[blue!85!black,rectangle](14,0)
\DoUntilOutOfBounds{
\class[blue!85!black,rectangle](\lastx+1,\lasty-1)
}
\class[blue!85!black](40,-12)
\class[blue!85!black](44,-16)
\class[blue!85!black](46,-18)
\class[blue!85!black](47,-19)
\class[blue!85!black](48,-20)
\class(48,48)
\class(60,-60)
\structline[green!50!black](0,0)(48,48)
\structline[green!50!black](0,0)(60,-60)
\end{sseqdata}
\begin{figure}
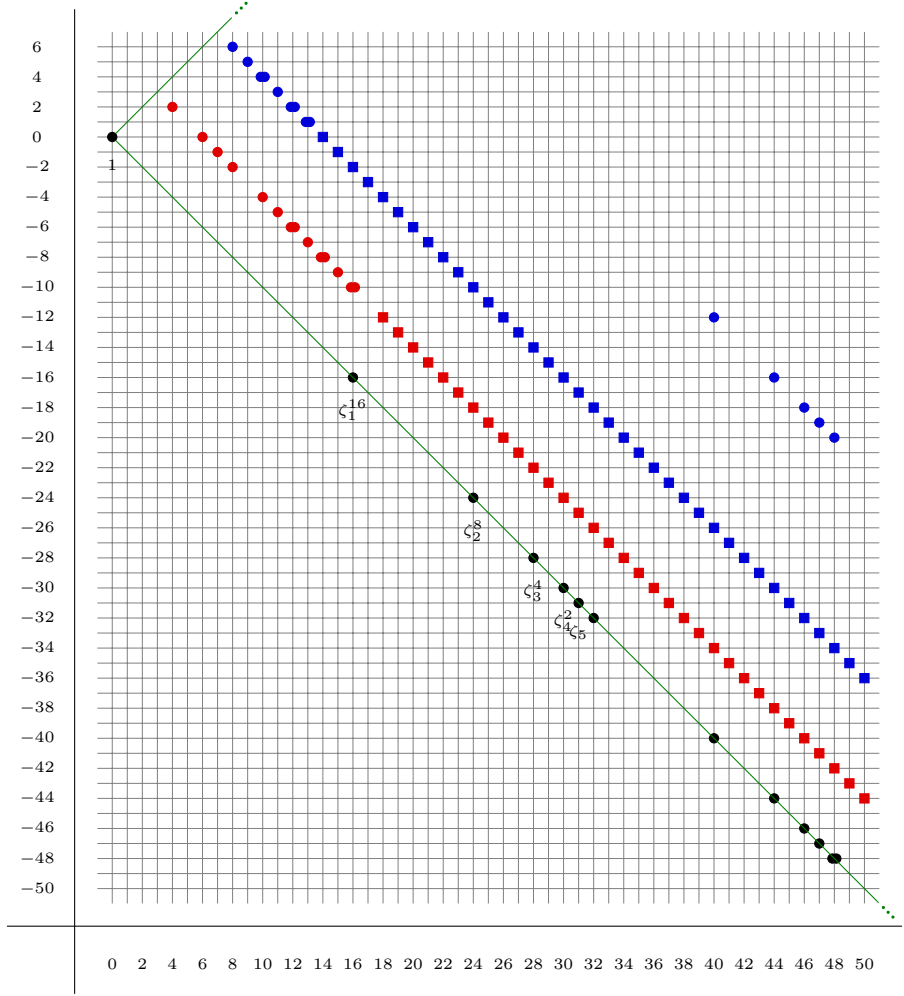

\centering
\printpage[name = Einftybp3]
\caption{The $E_{\infty}$ page of the HSSS for $BPGL\langle 3\rangle$.}
\label{einftybp3}
\end{figure}

\section{The HSSS in equivariant homotopy}\label{secequivariant}
In the equivariant stable homotopy category $\Sp^G$ for a finite group $G$, Hill--Hopkins--Ravenel \cite{HHR} constructed a filtration analogous to Voevodsky's slice filtration. Later, Ullman \cite{ullman} constructed a filtration with better multiplicative properties called the \textit{regular} slice filtration. To fix notions, we work with Ullman's filtration, though for the spectra we consider, the filtrations coincide.

Most of our results on the arithmetic square and the HSSS for $BPGL\langle m\rangle$ apply without change in the equivariant setting for $\BPRn$ by applying Betti realization and using Corollary \ref{cor:betaE}. As in Section \ref{sec2}, we can smash the slice tower of $E$ with any $G$-spectrum $K$, obtaining a spectral sequence. The proofs in Section \ref{sec2} go through in this setting essentially verbatim, so we omit details. Whereas the slice spectral sequence gives a very powerful tool to compute the homotopy groups of the fixed points $E^G$ of a $G$-spectrum, our spectral sequence computes the \textit{homology} of $E^G$.



Computations in $\R$-motivic homotopy and $C_2$-equivariant homotopy often differ due to the presence of the negative cone in the equivariant homology of a point. We fully compute the $RO(C_2)$-graded HSSS for $k_\R$ and we find an exotic differential (Corollary \ref{corexoticdiff}) of the form
\[d_5(2u_{2\sigma}^{-1}\cdot \zeta_1^2\zeta_2)=\ov_1^2\]
This differential, in particular, originates in the negative cone and lands in the positive cone. We deduce this differential from the following fact.

\begin{proposition}\label{v_nnilpotent}
The map
\[\pi_\star \BPR\to\pi_\star(i_*H\mathbb F_2\otimes \BPR)\]
sends $\overline{v}_{m}^{2^m}\mapsto0$.
\end{proposition}
\begin{proof}
By Theorem \ref{differentialsthm}, we have a differential
\[
d_{2^{m+1}-1}(\zeta_1^{2^m})=a_\sigma^{2^m-1}\overline{v}_{m}
\]
on $E_{2^{m+1}-1}(\BPR;i_*H\mathbb F_2)$. There is nothing in higher filtration in this stem, so we have the relation 
\[
a_\sigma^{2^m-1}\overline{v}_{m}=0\in \pi_\star(i_*H\mathbb F_2\otimes \BPR)
\]
By induction, we may assume that $a_\sigma^{2^m-i}\overline{v}_{m}^i=0$ for $1\le i<2^m$, which implies that the class $a_\sigma^{2^m-(i+1)}\overline{v}_{m}^i$ is in the image of the transfer. Since the restriction of $\overline{v}_{m}$ is zero, the Frobenius relation implies that $a_\sigma^{2^m-(i+1)}\overline{v}_{m}^{i+1}=0$.
\end{proof}

\begin{remark}
Note that the map
\[\pi_\star\BPR\to\pi_\star(H\underline{\F_2}\otimes\BPR)\]
sends $\overline{v}_{m}\mapsto 0$, since the latter is a constant Mackey functor in degrees $*\rho$, and $v_m=0$ in $H_*(BP;\F_2)$. However, it is not hard to see from the HSSS that 
\[\overline{v}_{m}\neq0\in \pi_\star(i_*H\mathbb F_2\otimes \BPR)\]
It seems likely that $2^m$ is the minimal nilpotence degree for all $m$.
\end{remark}



The computation of the $E_2$-page of the homological slice {\sseq} in the equivariant setting follows in the same way, where now we include the negative cone. We recall Stong's computation of $\pi_\star H\mZ$.

\begin{proposition}\label{propstong}
    The $RO(C_2)$-graded homotopy $\pi_\star H\mZ$ is the square zero extension of the ring $\Z[a_\sigma,u_{2\sigma}]/(2a_\sigma)$ by the module
    \[NC:=\Z\{e_{2n\sigma}|n>0\}\oplus\F_2\{e_{(2n+1)\sigma}/a_\sigma^j|n>0,j\ge0\}\]
    where $e_{k\sigma}=\mathrm{tr}(u_\sigma^{-k})$. The module structure on $NC$ is determined by the relations 
    \[a_\sigma e_{k\sigma}=u_{2\sigma}e_{2\sigma}=u_{2\sigma}e_{3\sigma}=0\]
    and
    \[u_{2\sigma}e_{k\sigma}=e_{(k-2)\sigma}\]
    for $k\ge 4$.
    \end{proposition}


\begin{proposition}
There is an isomorphism of $RO(C_2)$-graded rings
\[
\pi_\star(i_*H\mathbb F_2\otimes H\mZ)\cong (\mathcal A_*\square_{\mathcal A(0)_*}\F_2)[a_\sigma,x_1]\oplus NC\]
where the latter ring is a square zero extension of the ring $(\mathcal A_*\square_{\mathcal A(0)_*}\F_2)[a_\sigma,x_1]$ by the module
\[NC:=(\mathcal A_*\square_{\mathcal A(0)_*}\F_2)\{e_{(2i+1)\sigma}/a_\sigma^j,x_1\cdot e_{(2i+1)\sigma}/a_\sigma^j\st i\ge1,j\ge0\}\]
defined by the relation
\[x_1\cdot (x_1\cdot e_{(2i+1)\sigma}/a_\sigma^j)=\zeta_1^2\cdot e_{(2i+1)\sigma}/a_\sigma^{j-2}+e_{(2i-1)\sigma}/a_\sigma^j\]
and the usual relations in $\pi_\star H\mZ$.
\end{proposition}
\begin{proof}
This follows from Proposition \ref{propstong} and the isomorphisms

\aln{
\mathcal A_*\{e_{(2i+1)\sigma}/a_\sigma^j\}&\cong (\mathcal A_*\square_{\mathcal A(0)_*}\F_2)\{e_{(2i+1)\sigma}/a_\sigma^j\}\oplus \zeta_1\cdot(\mathcal A_*\square_{\mathcal A(0)_*}\F_2)\{e_{(2i+1)\sigma}/a_\sigma^j\}\\
&\cong(\mathcal A_*\square_{\mathcal A(0)_*}\F_2)\{e_{(2i+1)\sigma}/a_\sigma^j\}\oplus x_1\cdot(\mathcal A_*\square_{\mathcal A(0)_*}\F_2)\{e_{(2i+1)\sigma}/a_\sigma^{j+1}\}
}
and
\[(\mathcal A_*\square_{\mathcal A(0)_*}\F_2)\{e_{2i\sigma}\}\cong (\mathcal A_*\square_{\mathcal A(0)_*}\F_2)\{x_1\cdot e_{(2i+1)\sigma}\}\]
For the latter, we are using that
\[x_1\cdot e_{(2i+1)\sigma}=e_{2i\sigma}\]
as follows from the Frobenius relation and that $x_1$ restricts to $u_\sigma$. The relations (and the fact that $NC$ is square zero) then follow from the fact that
\[\pi_\star H\mZ\to \pi_\star(i_*H\mathbb F_2\otimes H\mZ)\]
is a ring map, along with Corollary \ref{corx_1rel}.
\end{proof}

\begin{corollary}
We have an isomorphism
\[E_2^{*,\star}(\BPRn;i_*H\mathbb F_2)\cong \bigg((\mathcal A_*\square_{\mathcal A(0)_*}\F_2)[a_\sigma,x_1]\oplus NC\bigg)[\overline{v}_{1},\ldots,\overline{v}_{m}]\]
An element $p\in(\mathcal A_*\square_{\mathcal A(0)_*}\F_2)_i$ has bidegree $(i,-i)$, and we have the following bidegrees
\aln{
|\overline{v}_{i}|&=((2^i-1)\rho,0)\\
|a_\sigma|&=(-\sigma,1)\\
|x_1|&=(1-\sigma,0)\\
|e_{(2i+1)\sigma}/a_\sigma^j|&=(-(2i+1)+(2i+j+1)\sigma,-j)
}
\end{corollary}

\begin{proposition}\label{prop:ROpermcycles}
In the spectral sequence
\[E_r^{*,\star}(k_{\mathbb R};i_*H\mathbb F_2)\]
the subalgebra
\[\bigg((\mathcal A_*\square_{\mathcal A(1)_*}\F_2)[a_\sigma,x_1]\oplus NC\bigg)[\overline{v}_{1},\ldots,\overline{v}_{m}]\]
of the $E_2$-page consists of permanent cycles.
\end{proposition}
\begin{proof}
This follows from Corollary \ref{corx_mpermanent}, Theorem \ref{edgethm}, and the fact that the negative cone in $\pi_\star H\mZ$ consists of permanent cycles in the slice spectral sequence for $\BPR$.
\end{proof}

By Theorem \ref{differentialsthm}, the first differentials we see are
\aln{
d_3(\zeta_1^2)&=a_\sigma\overline{v}_{1}\\
d_3(\zeta_2)&=x_1\overline{v}_{1}
}
and $d_3$ vanishes on the rest of the algebra generators by Proposition \ref{prop:ROpermcycles}. This gives the following. 
\begin{proposition}
$E_4^{*,\star}(k_{\mathbb R};i_*H\mathbb F_2)$ is the subalgebra of
\[\frac{\bigg((\mathcal A_*\square_{\mathcal A(0)_*}\F_2)[a_\sigma,x_1]\oplus NC\bigg)[\overline{v}_{1}]}{(a_\sigma\overline{v}_{1},x_1\overline{v}_{1})}\]
generated by 
\[\frac{\bigg((\mathcal A_*\square_{\mathcal A(1)_*}\F_2)[a_\sigma,x_1]\oplus NC\bigg)[\overline{v}_{1}]}{(a_\sigma\overline{v}_{1},x_1\overline{v}_{1})}\]
\aln{
x_2&:=x_1\zeta_1^2+a_\sigma\zeta_2\in\langle a_\sigma,\overline{v}_{1},x_1\rangle\\
y_n&:=e_{n\sigma}\zeta_1^2=\mathrm{tr}(u_\sigma^{-n}\zeta_1^2)\\
z_m&:=x_1\cdot\frac{e_{-3\sigma}}{a_\sigma^m}\zeta_2\\
w&:=e_{2\sigma}\zeta_1^2\zeta_2=\mathrm{tr}(u_\sigma^{-2}\zeta_1^2\zeta_2)
}
for $n\ge 2$ and $m\ge0$
\end{proposition}
The subalgebra
\[\frac{\bigg((\mathcal A_*\square_{\mathcal A(0)_*}\F_2)[a_\sigma,x_1]\oplus NC\bigg)[\overline{v}_{1}]}{(a_\sigma\overline{v}_{1},x_1\overline{v}_{1})}\]
consists of permanent cycles, $x_2$ is a permanent cycle by Corollary \ref{corx_mpermanent}, the $y_n$'s are permanent cycles since they are transfers of permanent cycles, and we show that the $z_m$'s are permanent cycles for degree reasons.

\begin{proposition}\label{prop:z_mperm}
The classes $z_m$ are permanent cycles for all $m\ge0$.
\end{proposition}
\begin{proof}
For this, note that
\[|z_m|=(1+(m+2)\sigma,-3-m)\]
So that
\[|d_r(z_m)|=((m+2)\sigma,-3-m+r)\]
Using our description of $E_4$, and the relations
\[\overline{v}_{1}x_2=\overline{v}_{1}y_n=\overline{v}_{1}z_m=\overline{v}_{1}w=0\]
 the target $d_r(z_m)$ can be written as a sum of monomials of the form
\[m(\zeta_i)\cdot x_1^\epsilon\cdot \frac{e_{(2i+1)\sigma}}{a_\sigma^j}\]
written as elements in a subquotient of $E_2$, for $\epsilon=0,1$ and $m(\zeta_i)$ a monomial in $\mathcal A_*\square_{\mathcal A(0)_*}\F_2$. If $m\in (\mathcal A_*\square_{\mathcal A(0)_*}\F_2)_k$, this monomial has bidegree
\[(k,-k)+(-(2i+1-\epsilon)+(2i+1-\epsilon+j)\sigma,-j)\]
It follows that $k=2i+1-\epsilon$ so that the filtration is $-(2i+1-\epsilon+j)$. On the other hand, looking at the $\sigma$ degree, we have
\[2i+1-\epsilon+j=m+2\]
This is a contradiction, as $r\ge 4$ tells us the filtration must be at least $1-m$.
\end{proof}

\begin{corollary}\label{corexoticdiff}
The differential $d_5$ is determined by
\[d_5(w)=\overline{v}_{1}^2\]
\end{corollary}
\begin{proof}
By Proposition \ref{prop:z_mperm}, the only algebra generator of $E_4$ that is not a permanent cycle is $w$. By Proposition \ref{v_nnilpotent}, there must be a differential killing $\overline{v}_{1}^2$, and for degree reasons, the stated differential is the only possibility.
\end{proof}
 
\begin{corollary}
The spectral sequence $E_r^{*,\star}(k_{\mathbb R};i_*H\mathbb F_2)$ collapses on $E_6$, which is the subalgebra of
\[\frac{\bigg((\mathcal A_*\square_{\mathcal A(0)_*}\F_2)[a_\sigma,x_1]\oplus NC\bigg)[\overline{v}_{1}]}{(a_\sigma\overline{v}_{1},x_1\overline{v}_{1},\overline{v}_{1}^2)}\]
generated by 
\[\frac{\bigg((\mathcal A_*\square_{\mathcal A(1)_*}\F_2)[a_\sigma,x_1]\oplus NC\bigg)[\overline{v}_{1}]}{(a_\sigma\overline{v}_{1},x_1\overline{v}_{1},\overline{v}_{1}^2)}\]
and the classes $x_2,y_n,z_m$ for $n\ge 2$ and $m\ge0$.
\end{corollary}
\begin{proof}
There are no new cycles because 
\[w^2=a_\sigma w=u_\sigma w=wz_m=wy_n=wx_2=0\in E_6\]
since these relations hold on $E_2$. The generators are all permanent cycles, so $E_6=E_\infty$.
\end{proof}
\begin{remark}
In this section, we relied heavily on ad hoc degree arguments to control the algebra generators of the $RO(C_2)$-graded HSSS for $\BPRn[1]$. For larger values of $m$, it becomes harder to determine when classes coming from the negative cone in the HSSS for $\BPRn$ are permanent cycles, so it seems that one needs additional input in these cases. However, by Proposition \ref{v_nnilpotent} there must be some pattern of differentials leaving the negative cone that kill powers of $\ov_m$.
\end{remark}

\printbibliography

\end{document}